\documentclass[11pt,letterpaper]{article}

\usepackage[margin=0.8in]{geometry}
\usepackage{amsmath}
\usepackage[normalem]{ulem}
\usepackage{hyperref}
\hypersetup{
    colorlinks=true,
    urlcolor = {darkred},
    linkcolor = {darkred},
    citecolor = {darkred},
    linkcolor = {darkred},
    linktoc=all
}
\usepackage{amssymb}
\usepackage{fourier}
\usepackage[utf8]{inputenc}
\usepackage{framed}
\usepackage{color}
\usepackage[table,xcdraw]{xcolor}
\usepackage{amsthm}
\usepackage{graphicx}
\usepackage{tikz}
\usepackage{tikz-cd}
\usepackage[inline]{enumitem}
\usepackage{expdlist}
\usepackage{mathtools}
\usepackage{authblk}
\usepackage{wrapfig,color}
\usepackage{extpfeil}
\usepackage{makecell}
\usepackage{multirow}

\usetikzlibrary{matrix,arrows,decorations.pathmorphing,patterns,shapes.geometric}
\theoremstyle{definition}
\newtheorem{theorem}{Theorem}[section]

\newtheorem{convention}[theorem]{Convention}
\newtheorem{lemma}[theorem]{Lemma}

\newtheorem{corollary}[theorem]{Corollary}
\newtheorem{notation}[theorem]{Notation}
\newtheorem{proposition}[theorem]{Proposition}
\newtheorem{question}[theorem]{Question}

\newtheorem{example}[theorem]{Example}
\newtheorem{remark}[theorem]{Remark}
\newtheorem{definition}[theorem]{Definition}

\definecolor{darkblue}{rgb}{0.0, 0.0, 0.8}
\definecolor{darkred}{rgb}{0.8, 0.0, 0.0}
\definecolor{darkgreen}{rgb}{0.0, 0.5, 0.0}

\newcommand{\ok}[1]	{{{#1}}}

\title{Extracting Persistent Clusters in Dynamic Data via M\"obius inversion}

\author[1]{Woojin Kim}
\author[2]{Facundo M\'emoli}
\affil[1]{Department of Mathematics,
	Duke University\\
	\texttt{woojin@math.duke.edu}}
\affil[2]{Department of Mathematics and Department of Computer Science and Engineering,
	The Ohio State University\\ 
	\texttt{facundo.memoli@gmail.com}}

\newcommand{\veewedge}{\mathpalette\superimpose{{\bigvee}{\bigwedge}}}
\newcommand{\pidgm}{\Pi\mathrm{dgm}^{\bigwedge}}
\newcommand{\full}{\mathrm{full}}

\newcommand{\ob}{\mathrm{ob}}
\newcommand{\C}{\mathbf{C}}

\newcommand{\Pb}{\mathbf{P}}
\newcommand{\dero}{d_{\mathrm{E}}}
\newcommand{\bettiy}{\abs{\veewedge \theta_Y}}
\newcommand{\betti}{\abs{\veewedge \theta_X}}
\newcommand{\subpart}{\mathrm{SubPart}}
\newcommand{\RNum}[1]{\uppercase\expandafter{\romannumeral #1\relax}}
\newcommand{\rNum}[1]{\lowercase\expandafter{\romannumeral #1\relax}}
\newcommand{\incc}{\mathcal{J}}
\newcommand{\gammay}{\gamma_Y=(Y,d_Y(\cdot))}
\newcommand{\gammax}{\gamma_X=(X,d_X(\cdot))}

\newcommand{\dyngy}{\mathcal{G}_Y=\left(V_Y(\cdot),E_Y(\cdot)\right)}
\newcommand{\dyngx}{\mathcal{G}_X=\left(V_X(\cdot),E_X(\cdot)\right)}
\newcommand{\dgx}{\mathcal{G}_X}
\newcommand{\dgy}{\mathcal{G}_Y}
\newcommand{\dgw}{\mathcal{G}_W}

\newcommand{\dyndis}{\mathrm{dis}^\mathrm{dyn}}
\newcommand{\dintll}{d_{\mathrm{I},\lambda'}^\mathrm{dynM}}
\newcommand{\dintl}{d_{\mathrm{I},\lambda}^\mathrm{dynM}}
\newcommand{\B}{\mathrm{barc}}

\newcommand{\barc}{\mathrm{barc}_{\ZZ}}

\newcommand{\reeb}{\mathrm{Reeb}}

\newcommand{\dintg}{d_\mathrm{I}^{\mathrm{dynG}}}
\newcommand{\dintm}{d_\mathrm{I}^{\mathrm{dynM}}}

\newcommand{\ZZ}{\mathbf{ZZ}}
\newcommand{\Po}{\mathbf{P}}
\newcommand{\vect}{\mathbf{vec}}
\newcommand{\Vect}{\mathbf{Vec}}

\newcommand{\free}{\mathcal{F}}

\newcommand{\dynG}{\mathcal{G}}
\newcommand{\Sets}{\mathbf{Set}}

\newcommand{\graph}{\mathrm{Graph}}

\newcommand{\X}{\mathbf{X}}
\newcommand{\pow}{\mathrm{pow}}
\newcommand{\dgh}{d_\mathrm{GH}}
\newcommand{\F}{\mathbb{F}}
\newcommand{\im}{\mathrm{im}}
\newcommand{\coim}{\mathrm{coim}}

\newcommand{\dis}{\mathrm{dis}}

\newcommand{\bott}{d_\mathrm{B}}
\newcommand{\dintf}{d_{\mathrm{I}}^\mathrm{F}}
\newcommand{\dintreeb}{d_{\mathrm{I}}^\mathrm{R}}
\newcommand{\dintwreeb}{d_{\mathrm{I}}^{\omega\mathrm{R}}}

\newcommand{\Hrm}{\mathrm{H}}

\newcommand{\T}{\mathbf{R}}
\newcommand{\Z}{\mathbf{Z}}
\newcommand{\R}{\mathbf{R}}
\newcommand{\Rop}{\mathbf{R}^\mathrm{op}}
\newcommand{\U}{\mathbf{Int}}
\newcommand{\N}{\mathbf{N}}
\newcommand{\NN}{\mathbb{N}}

\newcommand{\eps}{\varepsilon}
\newcommand{\dint}{d_{\mathrm{I}}}
\newcommand{\dgm}{\mathrm{dgm}}
\newcommand{\rips}{\mathcal{R}_\delta}

\newcommand{\emb}{\mathcal{C}}
\newcommand{\collapsing}{collapsing functor}
\newcommand{\unlabel}{\mathcal{U}_L^X}
\newcommand{\unlabelY}{\mathcal{U}_L^Y}
\newcommand{\unweight}{\mathcal{U}_{\omega}}
\newcommand{\A}{\mathcal{A}}
\newcommand{\sets}{\mathbf{set}}
\newcommand{\wsets}{\omega\sets}
\newcommand{\id}{\mathrm{id}}
\newcommand{\crit}{\mathrm{crit}}
\newcommand{\lmulti}{\left\{}
\newcommand{\rmulti}{\right\}}
\newcommand{\abs}[1]{\left\lvert#1\right\rvert}
\newcommand{\norm}[1]{\left\lVert#1\right\rVert}
\newcommand{\tripod}{R:\ X \xtwoheadleftarrow{\varphi_X} Z \xtwoheadrightarrow{\varphi_Y} Y}
\newcommand{\tripodone}{R_1: \ X \xtwoheadleftarrow{\varphi_X} Z_1 \xtwoheadrightarrow{\varphi_Y} Y}
\newcommand{\tripodtwo}{R_2: \ Y \xtwoheadleftarrow{\psi_Y} Z_2 \xtwoheadrightarrow{\psi_W} W}

\newcommand{\Int}{\mathbf{Int}}

\newcommand{\Intzz}{\mathbf{Int}(\ZZ)}
\newcommand{\Intzzop}{\mathbf{Int}(\ZZ)^{\mathrm{op}}}

\newcommand{\cosheaf}{cosheaf-inducing}
\newcommand{\Cosheaf}{Cosheaf-inducing}

\newcommand{\fpow}{\mathbb{Z}\pow}
\newcommand{\PP}{\mathbf{P}}
\newcommand{\Part}{\mathbf{Part}}
\newcommand{\Partin}{\mathbf{Part}}

\newcommand{\supp}{\mathrm{supp}}

\makeatletter
\def\my@overarrow@#1#2#3{\vbox{\ialign{##\crcr #1#2\crcr \noalign{\kern-\p@\nointerlineskip}$\m@th \hfil #2#3\hfil $\crcr}}}

\newcommand{\amsvectb}{%
	\mathpalette {\my@overarrow@\vectfillb@}}
\newcommand{\vecbar}{%
	\scalebox{0.8}{$\relbar$}}
\def\vectfillb@{\arrowfill@\vecbar\vecbar{\raisebox{-4.35pt}[\p@][\p@]{$\mathord\mathchar"017E$}}}
\makeatother

\newcommand{\amsvect}{%
	\mathpalette {\my@overarrow@\vectfill@}}
\def\vectfill@{\arrowfill@\relbar\relbar{\raisebox{-3.81pt}[\p@][\p@]{$\mathord\mathchar"017E$}}}

\makeatletter
\newcommand{\superimpose}[2]{%
  {\ooalign{$#1\@firstoftwo#2$\cr\hfil$#1\@secondoftwo#2$\hfil\cr}}}
\makeatother

\usepackage{tocloft}
\begin{document}
\maketitle

\begin{abstract}
Identifying and representing \emph{clusters} in time-varying network data is of particular importance when studying collective behaviors emerging in nature, in mobile device networks or in social networks. Based on combinatorial, categorical, and persistence theoretic viewpoints,  we establish a stable functorial pipeline for the summarization of the evolution of clusters in a time-varying network.  

We first construct a complete summary of the evolution of clusters in a given time-varying network over a set of entities $X$ of which takes the form of a \emph{formigram}. This formigram can be understood as a certain Reeb graph $\mathcal{R}$ which is labeled by subsets of $X$.  By applying M\"obius inversion to the formigram in two different manners, we obtain two dual notions of diagram: the  \emph{maximal group diagram}  and the \emph{persistence clustergram}, both of which are in the form of an `annotated' barcode. The maximal group diagram consists of time intervals annotated by their corresponding \emph{maximal groups} --- a notion due to Buchin et al., implying that we recognize the notion of maximal groups as a special instance of \emph{generalized persistence diagram} by Patel. On the other hand, the persistence clustergram is mostly obtained by annotating the intervals in the zigzag barcode of the Reeb graph $\mathcal{R}$ with certain merging/disbanding events in the given time-varying network. 

We show that both diagrams are complete invariants of formigrams (or equivalently of \emph{trajectory grouping structure} by Buchin et al.) and thus contain more information than the Reeb graph $\mathcal{R}$.

\end{abstract}


	\tableofcontents

	\section{Introduction} 
	
One of the most frequent tasks in data analysis consists of finding clusters in datasets. Since data in the real world are often  \emph{spatiotemporal} (i.e. data encode for both space and time), identifying and analyzing clusters in such data is of critical importance. 
Spatiotemporal data includes flocking, swarming, or herding behavior of animals \cite{benkert2008reporting,gudmundsson2006computing,gudmundsson2007efficient,huang2008modeling,li2010swarm,parrish1997animal,sumpter-collective,vieira2009line}, sensor networks \cite{adams2015evasion,adams2021efficient,de2006coordinate,de2007homological,gamble2012applied},  swarm robots \cite{rubenstein2012kilobot}, convoys \cite{jeung2008discovery}, moving clusters \cite{kalnis2005discovering}, or mobile groups \cite{hwang2005mining,wang2008efficient}.
Furthermore, with the development of tracking and mobile devices with positioning capabilities, numerous object trajectory data are being collected, which makes the task of identifying and analyzing clusters in spatiotemporal datasets increasingly more important.

Buchin et al. \cite{buchin2013trajectory} introduced a formal definition of \emph{groups} in a given set $\X$ of trajectories and  introduced a notion of \emph{trajectory grouping structure} which encodes all groups in $\X$ represented as a Reeb graph along the timeline that is decorated by \emph{maximal groups}. 

\begin{figure}
    \centering
    \includegraphics[width=\textwidth]{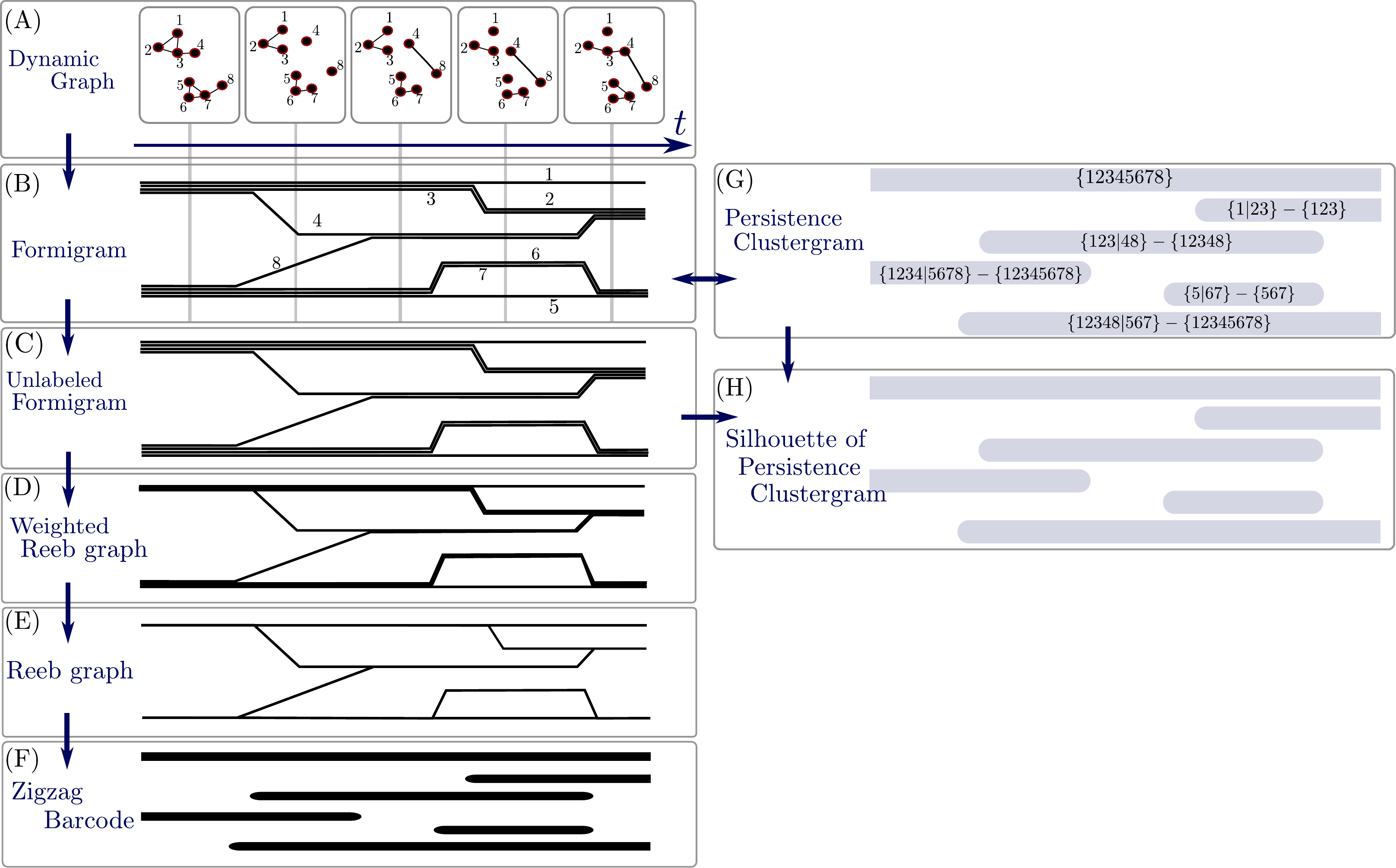}
    \caption{\textbf{Summarization process of a dynamic graph.} \ok{The evolution of connected components in a dynamic graph (A) is encoded in a formigram (B).
    A summarization procedure of the formigram yields invariants depicted in (C), (D), (E), and (F) in order. On the other hand, a persistence clustergram (G) is obtained from the formigram via M\"obius inversion. Persistence clustergrams are complete invariants of formigrams, i.e. persistence clustergrams permit reconstructing formigrams (hence the two-way arrow between (B) and (G)). 
    When a given formigram satisfies a certain property, the silhouette of persistence clustergram (H) is isomorphic to the zigzag barcode (F).
    Observe from (B), (G), and (F) that the label on an interval in the persistence clustergram indicates which merging and disbanding events in the formigram account for the corresponding interval in the zigzag barcode. In (F), (G), and (H) a rounded ending on an interval indicates an open endpoint while a straight ending indicates a closed endpoint.}}
    \label{fig:entire picture}
\end{figure}

The  \textbf{first} goal of this work is to make clear identification of the trajectory grouping structure (and similar concepts that appeared in applications; e.g. \cite{rossetti2018community,vehlow2015visualizing})) as a mathematical object, which we call \emph{formigrams} \cite{kim2018CCCG}.\footnote{The name formigram is a combination of the words {formicarium} and {diagram}. A formicarium or ant farm is an enclosure for keeping ants under semi-natural conditions \cite{wiki}. Visually, a formigram is reminiscent of a formicarium (cf. Figure \ref{fig:entire picture} (B)). } A formigram is a \emph{constructible cosheaf} over the real line   \cite{curry2014sheaves,curry2020classification} valued in \emph{the category of partitions}; in plain language, a formigram is a `zigzag' diagram of partitions.  This category can be seen as either a lattice (when an underlying set is specified) or a category whose morphisms are monic (when an underlying set is \emph{not} specified). This identification allows us to put formigrams into the context of the theories of persistence \cite{carlsson2009topology,edelsbrunner2008persistent}, lattices \cite{birkhoff1948lattice}, and categories \cite{bubenik2014categorification,mac2013categories}. The \textbf{second} goal is to exploit this established identity of formigrams for metrizing the space of formigrams, summarizing formigrams, and smoothing formigrams. In particular, the metrization of the space of formigrams addresses a question formulated by Buchin et al. regarding the quantification of the difference between two collective motions via trajectory grouping structures \cite[Section 6]{buchin2013trajectory}. In order to achieve our second goal, we deploy recent ideas from zigzag or multiparameter persistence \cite{bjerkevik2016stability,botnan2018algebraic,zigzag,dey2021computing}, constructible cosheaves \cite{curry2020classification,de2016categorified}, and generalized persistence diagrams \cite{kim2018generalized,patel2018generalized}.

\paragraph{Contributions.}	In a nutshell, we establish a stable and (mostly) functorial summarization process illustrated in Figures \ref{fig:entire picture} and \ref{fig:maximal group diagram}. This entails appropriate identification of  categories (Table \ref{table:categories}), functors (Table \ref{table:functors}), and metrics (Table \ref{table:names of cosheaves}), as well as appropriate usage of \emph{generalized rank} and \emph{M\"obius inversion} (Table \ref{tab:comparison}). More details follow: 

\begin{itemize}[leftmargin=*] \setlength\itemsep{0em}
	    \item We adapt recent ideas from generalized persistence diagrams and M\"obius inversion \cite{rota1964foundations} in order to define the \emph{persistence clustergram} of a formigram (cf. Figure \ref{fig:entire picture} (B),(G)). The persistence clustergram is a \emph{finer} invariant than the  Reeb graph of a formigram and thus it is also finer than the zigzag barcode of the Reeb graph (cf. Figure \ref{fig:entire picture} (E),(F)). In fact, the persistence clustergram is a \emph{complete invariant} of formigrams. See \textbf{Definition \ref{def:persistence clustergram}} and  \textbf{Remark \ref{rem:persistence clustergram is complete}}.

	    \item 	   We show that the maximal groups~\cite{buchin2013trajectory} of a set $\mathbf{X}$ of trajectories can be obtained from the formigram $\theta_\mathbf{X}$ associated to $\mathbf{X}$. 
	   More specifically, to obtain the maximal groups, we use the M\"obius inversion of a suitable notion of \emph{rank invariant} associated to the formigram.  This leads us to the fact that the collection of all maximal groups (which we call the \emph{maximal group diagram}) is another complete invariant of formigrams. 
	    This also implies that the collection of all maximal groups defined by Buchin et al.  \cite{buchin2013trajectory} is not just an invariant of collections of trajectories but it is also an invariant of formigrams.
	    See \textbf{Definition \ref{def:maximal group diagram}} and \textbf{Remark \ref{rem:max diagram is complete}}.	 
	    
	    \item  By utilizing the join operation on the lattice of (sub)partitions, we metrize the space of formigrams with a distance function $\dintf$ (the interleaving distance between formigrams). 
	    $\dintf$ dovetails with celebrated metrics in topological data analysis, which makes    the summarization pipeline that is illustrated in Figure \ref{fig:entire picture} is entirely stable.  
	   See \textbf{Proposition \ref{prop:formigram interleaving is mor discriminative than the reeb interleaving}, Corollary \ref{cor:df-db stability}, Theorems \ref{prop:stability from dg to formi}} and \textbf {\ref{thm:stability of betti}}. 
 \end{itemize}

 \begin{figure}
    \centering
    \includegraphics[width=0.5\textwidth]{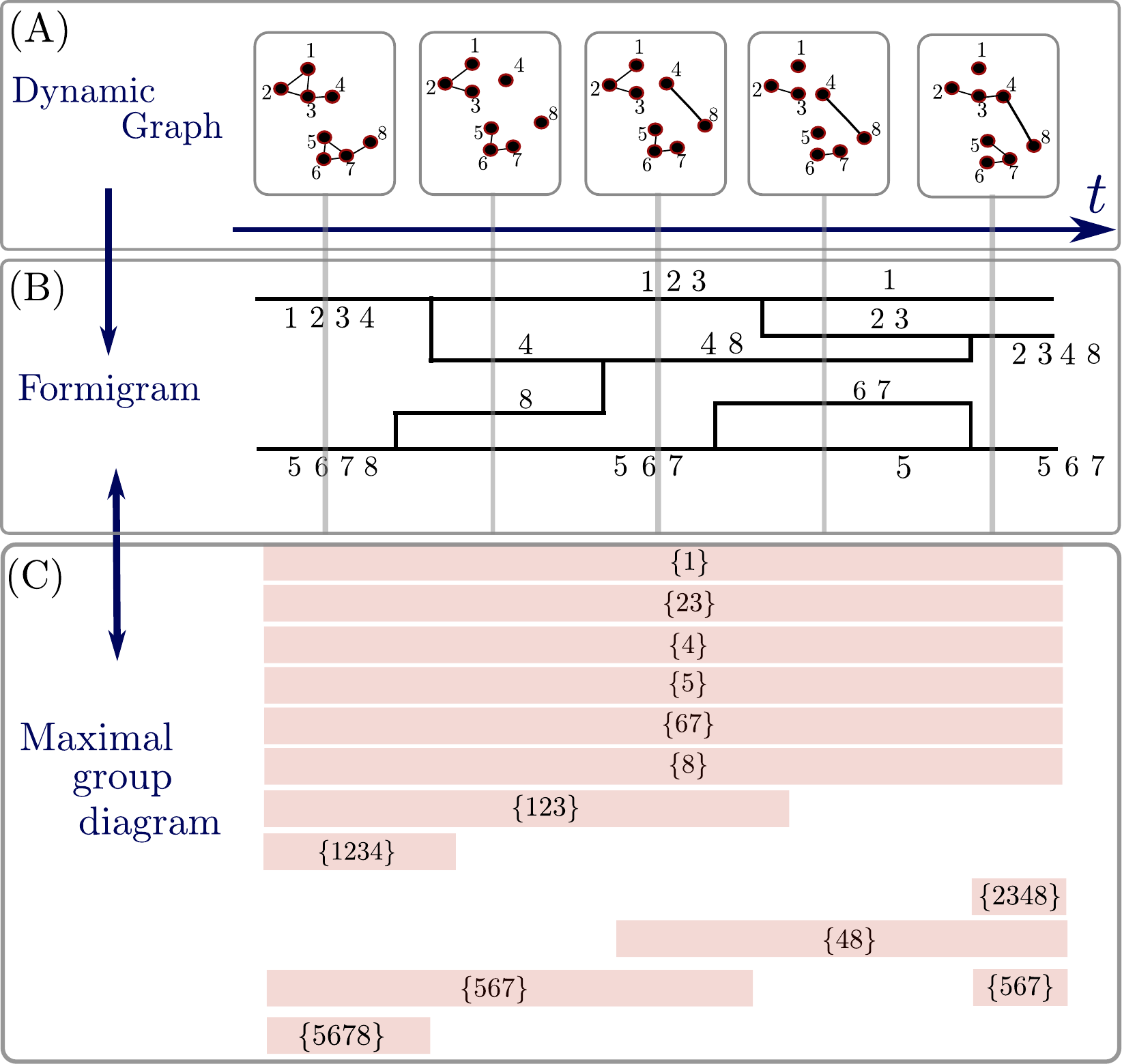}
    \caption{The \emph{maximal groups} introduced by Buchin et al. \cite{buchin2013trajectory} can be obtained via the M\"obius inversion of the $\bigwedge$-rank function of a formigram (cf. Definition \ref{def:maximal group diagram}). NB: In panel (C), commas are omitted in each set, e.g. $\{23\}$ means the set $\{2,3\}$.}
    \label{fig:maximal group diagram}
\end{figure}

	    Several subsidiary contributions follow:
	    \begin{itemize}[leftmargin=*] \setlength\itemsep{0em}
	    \item We define \emph{dynamic graphs} as constructible cosheaves over $\R$ valued in the lattice of all subgraphs of a \ok{given} complete graph.
	    By doing this, we are able to harness machinery from  \ok{applied topology} which allows us to both quantify the difference between dynamic graphs and induce the invariants of dynamic graphs illustrated in Figure \ref{fig:entire picture}. See \textbf{Definition \ref{def:dyn graphs}.}
	    \item 
	    We identify a sufficient condition on a dynamic metric space $\gamma_X$ \cite{kim2020spatiotemporal} so that its  dynamic \emph{Rips graph}  $\rips^{1}(\gamma_X)$, $\delta\geq 0$ (1-skeleton of Rips complexes), define a dynamic graph as described in the previous item. For example, our result implies that a finite collection of piecewise linear trajectories in Euclidean space induces a dynamic graph via the Rips graph functor. See \textbf{Proposition \ref{prop:from dms to filt2}.}
	    \item We show that the $\lambda$-slack interleaving distance between dynamic metric spaces \cite{kim2020spatiotemporal} is a metric. Also we show that different choice of $\lambda>0$ does not change the topology of the metric space of dynamic metric spaces. See \textbf{Theorem \ref{thm:lambda metric}} and \textbf{Proposition \ref{prop:equivalence}.}
	    \item 	We observe that \emph{robust grouping structure} by Buchin et al. \cite{buchin2013trajectory} can be formulated as an \emph{intrinsic} smoothing operation on formigrams. Smoothing operations on formigrams is entirely compatible with the Reeb graph smoothing operation.  We also clarify the effect of smoothing operations on the zigzag barcode of an input formigram. See \textbf{Remark \ref{rem:difference from Buchin's smoothing}, Propositions \ref{prop:barcode} and \ref{prop:compatibility}.} 
	    \item We clarify the computational complexity of every metric that is introduced in this paper. See \textbf{Theorem \ref{thm:DG metric-complexity}, Remark \ref{rem:DG and formi distance computation}  (\ref{item:DG and formi distance computation2})}, and \textbf{Theorems \ref{thm:complex}} and \textbf{ \ref{thm:complexity-dIF}} (Theorem \ref{thm:complexity-dIF} appeared in the conference paper by the same authors \cite{kim2018CCCG}). 
	\end{itemize}

	\paragraph{Other related work.} 	In \cite{kim2020analysis}, a collection of dynamic graphs that naturally arise from an artificial life program, called \emph{Boids} \cite{boids}, were successfully classified by the bottleneck distance on their zigzag barcodes (cf. Figure \ref{fig:entire picture} (F)). An implementation is available in \cite{clause2021zigzag}.

	\ok{In \cite{vangoethem2016grouping}, van Goethem et al. presented algorithms and data structures that support the interactive analysis of the trajectory grouping structure.} In \cite{van2018refined,wiratma2019experimental}, Van Kreveld et al. proposed an alternative to the aforementioned definition of \emph{groups} of \cite{buchin2013trajectory} together with experimental evaluation related to identifying groups from real or simulated pedestrian data.

	Sinhuber and Ouellette carried out statistical analysis on time-varying connectivity graphs in order to characterize swarms of midges \cite{sinhuber2017phase}. Munch established stability of time-varying persistence diagrams derived from dynamic point clouds in $\R^d$ \cite{munch2013applications}.  Kim and M\'emoli proposed a method to encode spatiotemporal topology of dynamic metric spaces into multiparameter persistence modules  \cite{kim2020spatiotemporal}. While \cite{kim2020spatiotemporal} provides useful tools for characterizing and classifying dynamic metric spaces according to their spatiotemporal homological features \cite{clause2020spatiotemporal}, no visualizable summary was provided which is one of the novel contributions of the present work.

We remark that, in this paper, only clustering features of dynamic \emph{graphs} \cite{adams2015evasion,de2006coordinate,gamble2012applied,hajij2017visual} are studied  (i.e. the zeroth order homological features), whereas \cite{kim2020spatiotemporal} considers arbitrary order homological features of dynamic \emph{metric spaces} \cite{sinhuber2017phase,topaz,xian2020capturing}.
	
	R. González-Díaz et al. devised the so-called spatiotemporal barcode as a visual tool to encode the lifespan of connected components on an image sequence over time \cite{gonzalez2015spatiotemporal}. Dey and Hou proposed efficient algorithms for computing zigzag persistent homology of dynamic graphs \cite{dey2021computing}. Kim, M\'emoli, and Stefanou showed in \cite{kim2019interleaving} that the aforementioned distance $\dintf$ between formigrams is  closely related to both the Hausdorff distance \cite{buragobook} and the erosion distance \cite{patel2018generalized}. Rolle and Scoccola  considered a variation of $\dintf$ for comparisons of multiparameter hierarchical clusterings \cite{rolle2020stable}. Ghrist and Riess studied a Hodge theory for cellular sheaves valued in lattices (i.e. functors from the face posets of cell complexes to lattices) \cite{ghrist2020cellular}.

\paragraph{Acknowledgements.} This work was partially supported by NSF grants IIS-1422400, CCF-1526513, DMS-1723003, and CCF-1740761. We thank Zane Smith for providing an example of non-planar formigram in Example \ref{example:non-planar}. Also, we thank Michael Lesnick for useful comments about the paper. WK thanks Amit Patel for beneficial discussions regarding topics related to Section \ref{sec:summarizing formigrams via mobius}.

\paragraph{Outline.} Section \ref{sec:prelim} includes preliminaries from category theory, poset theory, and persistence theory. Section \ref{sec:DGs} introduces notions of dynamic graphs, formigrams, and invariants of formigrams that are relevant to an algebraic structure of zigzag persistence. Section \ref{sec:metrics-new} introduces interleaving distances between dynamic graphs and between formigrams. Also, the stability of algebraic invariants that were introduced in the previous section is shown therein. \ok{Section \ref{sec:about partition category}  establishes a few categorical results that will be useful in later sections.} \ok{Section \ref{sec:summarizing formigrams via mobius} is the core of this paper. We introduce several combinatorial invariants of formigrams such as \emph{maximal group diagrams} and \emph{persistence clustergrams} and discuss their stability. Section \ref{sec:DMSs} we discuss how to convert dynamic metric spaces into dynamic graphs via the Rips graph functor and its stability. Section \ref{sec:discussion} concludes this work. For readability we
have relegated some proofs, remarks, and examples to an appendix: Sections \ref{sec:details}, \ref{sec:distance between weighted reeb graphs}, \ref{sec:smoothing}, \ref{sec:0-slack DMS intereleaving}, and \ref{sec:higher dimensional barcodes}.}

	
	\section{Preliminaries} \label{sec:prelim}
	In Sections \ref{sec:category} and \ref{sec:categories} we introduce the posets, categories, and functors that are considered throughout the paper. In Section \ref{sec:intervaldecomp} we review the notion of interval decomposable persistence modules. In Section \ref{sec:constructible cosheaves} we recall both the notion of constructible cosheaf over the real line and a suitable notion of interleaving.

	\subsection{Poset theory elements}\label{sec:category}
	
Any nonempty set with a partial order determines a \emph{poset}. Given a poset $\Po$, we call any subset $\mathbf{Q}$ of $\Po$ with the partial order obtained by restricting that of $\Po$ to $\mathbf{Q}$ a \emph{subposet of $\Po.$} We introduce posets that are considered in this paper.  First, the following posets will be used as domain posets of order-preserving maps or indexing posets of  functors.
	
	\begin{enumerate}[itemsep=-1ex]
\item A poset $\PP$ is said to be a \emph{join-semilattice} (resp. \emph{meet-semilattice}) if $\PP$ contains the least upper bound $p\vee q$ (resp. greatest lower bound $p\wedge q$) of any pair of points  $p,q\in \Po$. If $\PP$ is both join- and meet-semilattice, then $\PP$ is said to be a \emph{lattice}. \label{item:lattice}
		
		\item The poset $\R^n$ with  order $(a_1,a_2,\cdots,a_n)\leq (b_1,b_2,\cdots, b_n)$ if and only if $a_i\leq b_i$ for all $1\leq i\leq n$.
		
		\item $\Z_+$ and $\R_+$ will denote the sets of nonnegative integers and nonnegative real numbers, respectively.
		
		\item Let $\Pb$ be a poset. By $\Pb^\mathrm{op}$ we mean the opposite poset of $\Pb$, i.e. for any $p,q\in \Pb$,  $p\leq q$ in $\Pb^{\mathrm{op}}$ if and only if $q\leq p$ in $\Pb$.
		\item The poset $\R^{\mathrm{op}}\times \R$ where $(a_1,a_2)\leq (b_1,b_2)$ if and only if $a_1\geq b_1$ and $a_2\leq b_2$ for the usual order $\leq$ on $\R$.  	
		
		\item The poset $\U$ consists of nonempty finite closed real intervals ordered by inclusion.  $\U$ will be identified with the subposet of $\R^\mathrm{op}\times \R$ consisting of $(a,b)$ with $a\leq b$ in $\R$ via the bijection $[a,b](\in \Int)\leftrightarrow (a,b)(\in \R^\mathrm{op}\times \R)$ (cf. Figure \ref{fig:kan poset} (A)).\label{U}
	\end{enumerate}

	\paragraph{Lattices of subgraphs and subpartitions.} In the rest of this section, $X$ will stand for some nonempty finite sets. We review lattice structures on the collection of subgraphs of the complete graph on $X$ and on the collection of (sub)partitions of $X$. 

\begin{definition}[Lattice of subgraphs]\label{def:graph poset}
	    By $\graph(X)$ we denote the lattice of subgraphs of the complete graph on the vertex set $X$ ordered by inclusion. For any $G,H\in \graph(X)$, the union of $G$ and $H$ is the join $G\vee H$, i.e. the graph whose vertex set (resp. edge set) is the union of the vertex sets (resp. edge sets) of $G$ and $H$. The intersection of $G$ and $H$ is the meet $G\wedge H$.
\end{definition}
	
\begin{definition}[Lattice of subpartitions]\label{def:subpart}  We call any partition $P$ of a subset $X'$ of $X$ a \emph{subpartition} of $X$. In this case we call $X'$ the \emph{underlying set of $P$}, i.e. $X'=\bigcup P$.  Each element of $P$ is called a \emph{block}. A partition of the empty set is defined as the empty set. 
		By $\subpart(X)$, we denote the set of \emph{all subpartitions of $X$}.	By $\Part(X)$, we denote the subcollection of $\subpart(X)$ consisting solely of partitions of the entire $X$. 
\end{definition}
	For example,  for $X:=\{x_1,x_2,x_3\}$, both $\{\{x_1\},\{x_2\}\}$ and $\{\{x_1,x_2,x_3\}\}$ belong to $\subpart(X)$. 
	These subpartitions will henceforth be  written simply as $\{x_1|x_2\}$ and $\{x_1x_2x_3\}$ respectively. 
	
	The collection $\subpart(X)$ forms a lattice. Given $P,Q\in\subpart(X)$, by  $P\leq Q$ we mean ``$P$ is finer than or equal to $Q$'', i.e. for all $B\in P$, there exists $C\in Q$ such that $B\subset C$. Given any $P,Q\in \subpart(X)$, the join $P\vee Q$ is the finest common coarsening of $P$ and $Q$. The meet $P\wedge Q$ is the coarsest common refinement of $P$ and $Q$.

\begin{example}Let $X:=\{x_1,x_2,x_3,x_4\}$. For $P_1=\{x_1|x_2\}, P_2=\{x_2x_3\},\ \mbox{and}\  P_3=\{x_1x_4\}$ in $\subpart(X)$, we have: 
 $\bigwedge_{i=1}^2 P_i = \{x_2\}$,  $\bigwedge_{i=1}^3 P_i = \emptyset$, $\bigvee_{i=1}^2 P_i=\{x_1|x_2x_3\}$, and $\bigvee_{i=1}^3 P_i=\{x_1x_4|x_2x_3\}$.
\end{example}

We define the \emph{path components functor} (i.e. order-preserving map) $\pi_0:\graph(X)\rightarrow \subpart(X)$, \ok{which is utilized for the summarization process depicted as the arrow (A)$\rightarrow$(B) in Figure \ref{fig:entire picture}.}
	
	\begin{definition}\label{def:pi zero} Given any graph $G_X=(X,E_X)$, define the partition $\pi_0(X,E_X):=X/\sim$ of $X$ where $\sim$ stands for the equivalence relation on $X$ defined by $x\sim x'$ if and only if there exists a sequence $x=x_1,\ x_2, \ldots, x_n=x'$ of points in $X$ such that $\{x_i,x_{i+1}\} \in E_X$ for each $i\in\{1,\ldots,n-1\}$. 
	\end{definition}
	Whenever $G_X\leq H_X$ in $\graph(X)$, we have that $\pi_0(G_X)\leq \pi_0(H_X)$ in $\subpart(X)$.

	\subsection{Category theory elements}\label{sec:categories}
	
		In this section we introduce categories and functors that appear in Tables \ref{table:categories} and \ref{table:functors} respectively. Consult \cite{mac2013categories} for general definitions related to category theory.  Given any category $\mathbf{C}$, we will denote the collection of objects in $\mathbf{C}$ by $\mathrm{Ob}(\mathbf{C}).$

\begin{table}[]
\begin{center}
\begin{tabular}{|l|l|l|l|}
\hline
  \cellcolor[HTML]{EFEFEF}Categories & \cellcolor[HTML]{EFEFEF}Objects                                                                  & \cellcolor[HTML]{EFEFEF}Morphisms     & \cellcolor[HTML]{EFEFEF}  \\ \hline
                  
 $\graph(X)$                    & Subgraphs of the complete graph on $X$                                                                                & Inclusion                
& Def. \ref{def:graph poset}   \\ \hline
 $\subpart(X)$                    & Subpartitions of $X$                                                                                & Refinement       & Def. \ref{def:subpart}               
\\ \hline
  $\Part$                          & \thead{Pairs $(X,P_X)$ where $X$ is a finite set \\ and $P_X$ is a partition of $X$}                        & Block-preserving injective maps        & Def. \ref{def:partition category} \\ \hline
  $\wsets$                         & \thead{Pairs $(X,w_X)$ where $X$ is a finite set and \\ a weight function  $w_X:X\rightarrow \mathbb{N}$} & Weight-observing maps        & Def. \ref{def:weighted sets}      \\ \hline
  $\sets$                          & Finite sets                                                                                      & Set maps                                                          & \\ \hline
  $\vect$                          & Fin. dim. vector spaces                                                                 & Linear maps                                               &     \\ \hline 
\end{tabular}

\end{center}
\caption{\ok{Categories that are considered throughout the paper. Those are connected by the functors in Table \ref{table:functors}.}}\label{table:categories}
\end{table}

\begin{table}[]
    \centering
\begin{tabular}{|c|}
\hline
\begin{tikzcd}
{\graph(X)} \arrow[rr,"\pi_0","\mbox{Def. \ref{def:pi zero}}"']&& \subpart(X) \arrow[rr,"\unlabel","\mbox{Def. \ref{def:three functors} \ref{item:unlabeling functor}}"']
&& \Partin \arrow[rr,"\mathcal{A}","\mbox{Def. \ref{def:three functors} \ref{item:agglomerating functor}}"']  && \wsets \arrow[rr,"\unweight","\mbox{Def. \ref{def:three functors} \ref{item:unweighting functor}}"'] && \sets \arrow[rr,"\free","\mbox{Def. \ref{def:free}}"']&& \vect
\end{tikzcd}\\
\hline
\end{tabular}

    \caption{Functors that connect the categories in Table \ref{table:categories}.}
    \label{table:functors}
\end{table}
		
\begin{remark}[Posets as categories]\label{rem:posets are categories}\begin{enumerate}[label=(\roman*)]
    \item Any poset $\Po$ will be regarded as a category: Objects are elements in $\Po$. For any $p,q\in \Po$, there exists a unique morphism $p\rightarrow q$ if and only if $p\leq q$. Therefore, $p\leq q$ in $\Po$ will denote the unique morphism $p\rightarrow q$. In this perspective, any order-preserving map between posets is a functor. Any subposet $\mathbf{Q}$ of $\Po$ is a full subcategory of $\Po$. \label{item:posets are categories1}
    \item 	\ok{Since $\graph(X)$ is a poset, given \emph{any} poset $\Pb$ and \emph{any} functor $F:\Pb\rightarrow \graph(X)$ (i.e. order-preserving map), the limit and colimit of $F$ are $\bigcup_{p\in \Pb} F_p$ and $\bigcap_{p\in \Pb} F_p$,  respectively. Similarly, since $\subpart(X)$ is a poset, given \emph{any} poset $\Pb$ and \emph{any} functor $F:\Pb\rightarrow \subpart(X)$, the limit and colimit of $F$ are $\bigwedge_{p\in \Pb} F_p$ and $\bigvee_{p\in \Pb} F_p$,  respectively.} 
\end{enumerate}

\end{remark}

		The category $\sets$ consists of  finite sets with set maps. The category $\vect$ consists of finite-dimensional vector spaces over a fixed field $\F$ with linear maps. 	
\begin{definition}\label{def:free}
The \emph{free functor} $\free:\sets\rightarrow \vect$ sends any set $S$ to the vector space $\free(S)$ which consists of formal linear combinations $\sum_i a_i s_i\ (a_i\in \mathbb{F},\ s_i\in S)$ of finite terms of elements in $S$ over the field $\mathbb{F}$. Also, given a set map $f:S\rightarrow T$, $\free(f)$ is the linear map from $\free(S)$ to $\free(T)$ obtained by linearly extending $f$. 
\end{definition}	

		For two categories $\mathbf{C}$ and $\mathbf{D}$, the category $\mathbf{C}^\mathbf{D}$ stands for the category of functors from $\mathbf{D}$ to $\mathbf{C}$ with objects being functors $\mathbf{C}\rightarrow\mathbf{D}$ and arrows being natural transformations. 		For two functors $F,G:\mathbf{C}\rightarrow \mathbf{D}$, we write $F\cong G$ whenever $F$ and $G$ are \emph{naturally isomorphic}, i.e. there exists a natural transformation $\tau:F\rightarrow G$ such that $\tau_c:F(c)\rightarrow G(c)$ is invertible for each $c\in \mathrm{Ob}(\mathbf{C})$. 
		
We introduce the category of partitions (without specifying a set to partition) and the category of weighted sets. These categories appear in Figure \ref{fig:entire picture} (B) and (C).

\begin{definition}\label{def:partition category}
In the \emph{category of partitions}, denoted by $\Part$, an object is a pair $(X,P_X)$ of a finite set $X$ and a partition $P_X$ of $X$. A morphism $f:(X,P_X)\rightarrow (Y,P_Y)$ is an \emph{injective} map $f:X\hookrightarrow Y$ which preserves the equivalence relation on $X$ induced by $P_X$, i.e. \ok{if $x$ and $x'$ belong to the same block in $P_X$, then so do $f(x)$ and $f(x')$ in $P_Y$.}
\end{definition}

Note that $f$ induces the map $f_\ast:P_X\rightarrow P_Y$ that sends each block $B$ to the unique block $C$ containing the image $f(B)$. The morphism $f$ is an isomorphism if both $f$ and $f_\ast$ are bijective. For example, let $X:=\{x_1,x_2,x_3\}$. Whereas $P_X=\{x_1x_2|x_3\}$ and $Q_X=\{x_1|x_2x_3\}$ are non-isomorphic objects in $\subpart(X)$, the pairs $(X,P_X)$ and $(X,Q_X)$ are isomorphic in $\Part$.

\ok{Next we define the category $\wsets$ of \emph{weighted sets}.}

\begin{definition}\label{def:weighted sets} In $\wsets$, an object is a pair $(X,w_X)$ of a finite set $X$ and a positive-integer-valued \emph{weight} map $w_X:X\rightarrow \NN$. A morphism $f:(X,w_X)\rightarrow (Y,w_Y)$ is a weight-observing set map $f:X\rightarrow Y$, i.e. for all $y\in Y$, $\sum_{x\in f^{-1}(y)} w_X(x) \leq w_Y(y)$.
\end{definition}

Notice the following: (1) If there exists a morphism from $(X,w_X)$ to $(Y,w_Y)$, then the \emph{total weight} of $(Y,w_Y)$ is at least that of $(X,w_X)$, i.e. $\sum_{x\in X}w_X(x)\leq \sum_{y\in Y}w_Y(y)$.  (2) Two objects $(X,w_X)$ and $(Y,w_Y)$ are isomorphic if and only if there exists a weight-\emph{preserving} bijection $X\rightarrow Y$. (3) If there exists a pair of morphisms in the opposite directions between two objects in $\wsets$, then the two object must be isomorphic.

\ok{We introduce the following three functors which are utilized for the summarization procedure depicted as the arrows (B) $\rightarrow$ (C) $\rightarrow$ (D) $\rightarrow$ (E)} in  Figure \ref{fig:entire picture}. See also Table \ref{table:functors}. Let $X$ be any nonempty finite set. 

\begin{definition}[Three functors]\label{def:three functors} 
\begin{enumerate}[label=(\roman*),itemsep=-1ex]
    \item The \emph{unlabeling functor} $\unlabel$ sends each $P\in \subpart(X)$ to the pair $(\bigcup P,P)$. For any pair $P\leq Q$ in $\subpart(X)$, the corresponding morphism $\unlabel(P\leq Q):(\bigcup P, P)\rightarrow (\bigcup Q, Q)$ is the inclusion map $\bigcup P \hookrightarrow \bigcup Q$. \label{item:unlabeling functor}
    \item The \emph{agglomerating functor} $\A$ sends each $(X,P_X)\in \ob(\Partin)$ to the pair $(P_X,|-|)\in\ob(\wsets)$ where $|-|:P_X\rightarrow \N$ is the size function, i.e. any block in $P_X$ is sent to its size (i.e. cardinality). Since any morphism $f:(X,P_X)\rightarrow (Y,P_Y)$ in $\Partin$ is an injective set map $X\hookrightarrow Y$, the induced map $f_\ast:P_X\rightarrow P_Y$ is a weight-observing map. \label{item:agglomerating functor}
    \item The \emph{unweighting functor} $\unweight$ simply forgets the weight function (i.e. the second entry) from each object in $\wsets$. \label{item:unweighting functor} 
\end{enumerate}
\end{definition}

	\subsection{Interval decomposable persistence modules}\label{sec:intervaldecomp}
	In this section we review the notion of interval decomposability of persistence modules.

	Let $\Pb$ be a poset. Any functor $F:\Po\rightarrow \vect$ will be called a \emph{$\Po$-indexed module}. This means that each $p\in \Pb$ is sent to a finite dimensional vector space $F(p)$ and each $p\leq q$ in $\Pb$ is sent to a linear map $F(p\leq q):F(p)\rightarrow F(q)$. In particular, for any $p\in \Pb$, $p\leq p$ is sent to the identity map on $F(p)$. Also, for any $p\leq q\leq r$ in $\Pb$, we have $F(p\leq r)=F(q\leq r)\circ F(p\leq q)$. 
	
	\begin{definition}[Intervals \cite{botnan2018algebraic}]\label{interval} Given a poset $\Po$, an \emph{interval} $\mathcal{J}$ of $\Po$ is any subset $\mathcal{J}\subset \Po$ such that
		\begin{enumerate}[label=(\roman*),itemsep=-1ex]
			\item $\mathcal{J}$ is nonempty.
			\item If $p,r\in \mathcal{J}$ and $p\leq q\leq r$, then $q\in \mathcal{J}$.
			\item (connectivity) For any $p,q\in \mathcal{J}$, there is a sequence $p=p_0,p_1,\cdots,p_l=q$ of elements of $\mathcal{J}$ with $p_i$ and $p_{i+1}$ comparable for $0\leq i\leq l-1$.\label{item:connectivity}
		\end{enumerate}
	\end{definition}

	\label{interval module}For $\mathcal{J}$ an interval in $\Po$, the \emph{interval module} ${I}^{\mathcal{J}}:\Po\rightarrow \vect$ is the $\Po$-indexed module where 
	
	$${I}^{\mathcal{J}}(p)=\begin{cases}
	\mathbb{F}&\mbox{if}\ p\in \mathcal{J},\\0
	&\mbox{otherwise.} 
	\end{cases}\hspace{20mm} {I^{\mathcal{J}}}(p\leq q)=\begin{cases} \mathrm{id}_\mathbb{F}& \mbox{if} \,\,p,q\in\incc,\\ 0&\mbox{otherwise.}\end{cases}$$
	
The direct sums of $\Po$-indexed modules are defined pointwise at each index $t\in \Pb$.	We say that a $\Po$-indexed module $F$ is \textit{decomposable} if $F$ is naturally isomorphic to $G_1\bigoplus G_2$ for some non-trivial $\Po$-indexed modules $G_1$ and $G_2$. Otherwise, we say that $F$ is \textit{indecomposable}. Any interval module is indecomposable \cite[Proposition 2.2]{botnan2018algebraic}. 
	A $\Po$-indexed module $F$ is \emph{interval decomposable} if there exists a multiset $\B(F)$ of intervals in $\Po$ such that 
	$$F\cong \bigoplus_{\mathcal{J}\in \B(F)}I^{\mathcal{J}}.$$\label{barcode}
	
	By Azumaya-Krull-Remak-Schmidt \cite{azumaya1950corrections}, the multiset $\B(F)$ is unique if $F$ is interval decomposable.  In this case, we call $\B(F)$ the \emph{barcode} of $F.$
	
	\subsection{Constructible cosheaves over $\R$,  their interleavings, and their \ok{smoothing operations}}\label{sec:constructible cosheaves}
	
	In this section we review the notion of constructible cosheaves over $\R$,  their interleavings, and their \ok{smoothing} operations.  \ok{Constructible cosheaves over $\R$ are fully characterized by ``zigzag diagrams over $\R$"  \cite{curry2020classification,de2016categorified}.} 
	
	\begin{definition}[Zigzag poset]\label{def:intervals in zz} Let $\ZZ$ be the subposet of $\R^{\mathrm{op}}\times \R$ given by $\ZZ:=\{(k,l): k\in \Z,\ l\in \{k,k-1\}\}$  (cf. Figure \ref{fig:intervals}).  \ok{By $\Int(\ZZ)$ we denote the poset of all \emph{finite} intervals in $\ZZ$ ordered by inclusion.} 
	\end{definition}
	
	\begin{notation}[\cite{botnan2018algebraic}]\label{not:intervals in zz}Letting $<$ denote the strict partial order on $\Z^2$ (not on $\Z^{\mathrm{op}}\times \Z$), every interval in $\ZZ$ falls into one of the four types below:
\begin{align*}
		(b,d)_{\ZZ}&:=\{(i,j)\in \ZZ: (b,b)<(i,j)<(d,d)\}&&\mbox{for}\   b<d\in\Z ,\\
		[b,d)_{\ZZ}&:=\{(i,j)\in \ZZ: (b,b)\leq (i,j)<(d,d)\}&&\mbox{for}\   b<d\in\Z,\\
		(b,d]_{\ZZ}&:=\{(i,j)\in \ZZ: (b,b)< (i,j)\leq(d,d)\}&& \mbox{for}\   b<d\in\Z,\\
		[b,d]_{\ZZ}&:=\{(i,j)\in \ZZ: (b,b)\leq (i,j)\leq(d,d)\}&& \mbox{for}\   b\leq d\in\Z.
		\end{align*}
		
		See Figure \ref{fig:intervals} for examples. We let $\langle b,d \rangle_{\ZZ}$ denote any of the above sets without specifying its type. 	Also, by $\langle b,d\rangle$ for $b\in \{-\infty\}\cup\R$ and $d\in \R\cup\{\infty\}$, we denote any of the intervals $(b,d), [b,d), (b,d]$ or $[b,d]$ in $\R$. 
	\end{notation}

	\begin{figure}
		\begin{center}
			\begin{tikzpicture}[thick,scale=0.6, every node/.style={scale=0.6}]
			\fill[pink, opacity=0.4] (-0.5,-1.5)--(0.5,-1.5)--(0.5,-0.5)--(1.5,-0.5)--(1.5,0.5)--(0.5,0.5)--(-0.5,0.5); 
			\fill (-2,-2) circle (0.7mm);
			\fill (-1,-2) circle (0.7mm);	
			\fill (-1,-1) circle (0.7mm); 
			\fill (0,-1) circle (0.7mm);
			\fill (0,0) circle (0.7mm); 
			\fill (1,0) circle (0.7mm); 
			\fill (1,1) circle (0.7mm); 
			\fill (2,1) circle (0.7mm);
			\fill (2,2) circle (0.7mm);
			\fill (2.5,2.5) circle (0.4mm);
			\fill (2.6,2.6) circle (0.4mm);
			\fill (2.7,2.7) circle (0.4mm);
			\fill (-2.5,-2.5) circle (0.4mm);
			\fill (-2.6,-2.6) circle (0.4mm);
			\fill (-2.7,-2.7) circle (0.4mm);
			\draw[->,>=stealth'] (-1,-2)--(-1.9,-2);
			\draw[->,>=stealth'] (-1,-2)--(-1,-1.1);
			\draw[->,>=stealth'] (0,-1)--(-0.9,-1);
			\draw[->,>=stealth'] (0,-1)--(0,-0.1);
			\draw[->,>=stealth'] (1,0)--(0.1,0);
			\draw[->,>=stealth'] (1,0)--(1,0.9);
			\draw[->,>=stealth'] (2,1)--(1.1,1);
			\draw[->,>=stealth'] (2,1)--(2,1.9);
			\draw[->,>=stealth'] (-2,-2.5)--(-2,-2.1);
			\draw[->,>=stealth'] (2.5,2)--(2.1,2);
			\draw (-2,-2) node [anchor=south] {$(-2,-2)$};
			\draw (-1,-1) node [anchor=south] {$(-1,-1)$};
			\draw (0,0) node [anchor=south] {$(0,0)$};  
			\draw (1,1) node [anchor=south] {$(1,1)$};
			\draw (2,2) node [anchor=south] {$(2,2)$};
			\draw (-0.5,-2) node [anchor=north] {$(-1,-2)$};
			\draw (0.5,-1) node [anchor=north] {$(0,-1)$};
			\draw (1.5,0) node [anchor=north] {$(1,0)$};
			\draw (2.5,1) node [anchor=north] {$(2,1)$};	
			\draw (0,-4) node {\Large{$(-1,1)_\ZZ$}};			
			\end{tikzpicture}				
			\begin{tikzpicture}[thick,scale=0.6, every node/.style={scale=0.6}]
			\fill[pink, opacity=0.4] (-1.5,-1.5)--(0.5,-1.5)--(0.5,-0.5)--(1.5,-0.5)--(1.5,0.5)--(-0.5,0.5)--(-0.5,-0.5)--(-1.5,-0.5); 
			\fill (-2,-2) circle (0.7mm);
			\fill (-1,-2) circle (0.7mm);	
			\fill (-1,-1) circle (0.7mm); 
			\fill (0,-1) circle (0.7mm);
			\fill (0,0) circle (0.7mm); 
			\fill (1,0) circle (0.7mm); 
			\fill (1,1) circle (0.7mm); 
			\fill (2,1) circle (0.7mm);
			\fill (2,2) circle (0.7mm);
			\fill (2.5,2.5) circle (0.4mm);
			\fill (2.6,2.6) circle (0.4mm);
			\fill (2.7,2.7) circle (0.4mm);
			\fill (-2.5,-2.5) circle (0.4mm);
			\fill (-2.6,-2.6) circle (0.4mm);
			\fill (-2.7,-2.7) circle (0.4mm);
			\draw[->,>=stealth'] (-1,-2)--(-1.9,-2);
			\draw[->,>=stealth'] (-1,-2)--(-1,-1.1);
			\draw[->,>=stealth'] (0,-1)--(-0.9,-1);
			\draw[->,>=stealth'] (0,-1)--(0,-0.1);
			\draw[->,>=stealth'] (1,0)--(0.1,0);
			\draw[->,>=stealth'] (1,0)--(1,0.9);
			\draw[->,>=stealth'] (2,1)--(1.1,1);
			\draw[->,>=stealth'] (2,1)--(2,1.9);
			\draw[->,>=stealth'] (-2,-2.5)--(-2,-2.1);
			\draw[->,>=stealth'] (2.5,2)--(2.1,2);
			\draw (-2,-2) node [anchor=south] {$(-2,-2)$};
			\draw (-1,-1) node [anchor=south] {$(-1,-1)$};
			\draw (0,0) node [anchor=south] {$(0,0)$};  
			\draw (1,1) node [anchor=south] {$(1,1)$};
			\draw (2,2) node [anchor=south] {$(2,2)$};
			\draw (-0.5,-2) node [anchor=north] {$(-1,-2)$};
			\draw (0.5,-1) node [anchor=north] {$(0,-1)$};
			\draw (1.5,0) node [anchor=north] {$(1,0)$};
			\draw (2.5,1) node [anchor=north] {$(2,1)$};	
			\draw (0,-4) node {\Large{$[-1,1)_\ZZ$}};			
			\end{tikzpicture}	\begin{tikzpicture}[thick,scale=0.6, every node/.style={scale=0.6}]
			\fill[pink, opacity=0.4] (-0.5,-1.5)--(0.5,-1.5)--(0.5,-0.5)--(1.5,-0.5)--(1.5,1.5)--(0.5,1.5)--(0.5,0.5)--(-0.5,0.5); 
			\fill (-2,-2) circle (0.7mm);
			\fill (-1,-2) circle (0.7mm);	
			\fill (-1,-1) circle (0.7mm); 
			\fill (0,-1) circle (0.7mm);
			\fill (0,0) circle (0.7mm); 
			\fill (1,0) circle (0.7mm); 
			\fill (1,1) circle (0.7mm); 
			\fill (2,1) circle (0.7mm);
			\fill (2,2) circle (0.7mm);
			\fill (2.5,2.5) circle (0.4mm);
			\fill (2.6,2.6) circle (0.4mm);
			\fill (2.7,2.7) circle (0.4mm);
			\fill (-2.5,-2.5) circle (0.4mm);
			\fill (-2.6,-2.6) circle (0.4mm);
			\fill (-2.7,-2.7) circle (0.4mm);
			\draw[->,>=stealth'] (-1,-2)--(-1.9,-2);
			\draw[->,>=stealth'] (-1,-2)--(-1,-1.1);
			\draw[->,>=stealth'] (0,-1)--(-0.9,-1);
			\draw[->,>=stealth'] (0,-1)--(0,-0.1);
			\draw[->,>=stealth'] (1,0)--(0.1,0);
			\draw[->,>=stealth'] (1,0)--(1,0.9);
			\draw[->,>=stealth'] (2,1)--(1.1,1);
			\draw[->,>=stealth'] (2,1)--(2,1.9);
			\draw[->,>=stealth'] (-2,-2.5)--(-2,-2.1);
			\draw[->,>=stealth'] (2.5,2)--(2.1,2);
			\draw (-2,-2) node [anchor=south] {$(-2,-2)$};
			\draw (-1,-1) node [anchor=south] {$(-1,-1)$};
			\draw (0,0) node [anchor=south] {$(0,0)$};  
			\draw (1,1) node [anchor=south] {$(1,1)$};
			\draw (2,2) node [anchor=south] {$(2,2)$};
			\draw (-0.5,-2) node [anchor=north] {$(-1,-2)$};
			\draw (0.5,-1) node [anchor=north] {$(0,-1)$};
			\draw (1.5,0) node [anchor=north] {$(1,0)$};
			\draw (2.5,1) node [anchor=north] {$(2,1)$};	
			\draw (0,-4) node {\Large{$(-1,1]_\ZZ$}};			
			\end{tikzpicture}	\begin{tikzpicture}[thick,scale=0.6, every node/.style={scale=0.6}]
			\fill[pink, opacity=0.4] (-1.5,-1.5)--(0.5,-1.5)--(0.5,-0.5)--(1.5,-0.5)--(1.5,1.5)--(0.5,1.5)--(0.5,0.5)--(-0.5,0.5)--(-0.5,-0.5)--(-1.5,-0.5); 
			\fill (-2,-2) circle (0.7mm);
			\fill (-1,-2) circle (0.7mm);	
			\fill (-1,-1) circle (0.7mm); 
			\fill (0,-1) circle (0.7mm);
			\fill (0,0) circle (0.7mm); 
			\fill (1,0) circle (0.7mm); 
			\fill (1,1) circle (0.7mm); 
			\fill (2,1) circle (0.7mm);
			\fill (2,2) circle (0.7mm);
			\fill (2.5,2.5) circle (0.4mm);
			\fill (2.6,2.6) circle (0.4mm);
			\fill (2.7,2.7) circle (0.4mm);
			\fill (-2.5,-2.5) circle (0.4mm);
			\fill (-2.6,-2.6) circle (0.4mm);
			\fill (-2.7,-2.7) circle (0.4mm);
			\draw[->,>=stealth'] (-1,-2)--(-1.9,-2);
			\draw[->,>=stealth'] (-1,-2)--(-1,-1.1);
			\draw[->,>=stealth'] (0,-1)--(-0.9,-1);
			\draw[->,>=stealth'] (0,-1)--(0,-0.1);
			\draw[->,>=stealth'] (1,0)--(0.1,0);
			\draw[->,>=stealth'] (1,0)--(1,0.9);
			\draw[->,>=stealth'] (2,1)--(1.1,1);
			\draw[->,>=stealth'] (2,1)--(2,1.9);
			\draw[->,>=stealth'] (-2,-2.5)--(-2,-2.1);
			\draw[->,>=stealth'] (2.5,2)--(2.1,2);
			\draw (-2,-2) node [anchor=south] {$(-2,-2)$};
			\draw (-1,-1) node [anchor=south] {$(-1,-1)$};
			\draw (0,0) node [anchor=south] {$(0,0)$};  
			\draw (1,1) node [anchor=south] {$(1,1)$};
			\draw (2,2) node [anchor=south] {$(2,2)$};
			\draw (-0.5,-2) node [anchor=north] {$(-1,-2)$};
			\draw (0.5,-1) node [anchor=north] {$(0,-1)$};
			\draw (1.5,0) node [anchor=north] {$(1,0)$};
			\draw (2.5,1) node [anchor=north] {$(2,1)$};
			\draw (-2.5,2.5) node {};	
			\draw (0,-4) node {\Large{$[-1,1]_\ZZ$}};			
			\end{tikzpicture}
		\end{center}
		\caption{\label{fig:intervals}The points falling into the shaded regions comprise the intervals $(-1,1)_\ZZ, [-1,1)_\ZZ, (-1,1]_\ZZ$ and $[-1,1]_\ZZ$ of the poset $\ZZ$, respectively in order.}
	\end{figure}

\paragraph{Interleavings between $\Int$-indexed functors \cite{botnan2018algebraic}.}

	 For $\eps\geq 0$, and $I=[b,d]\in \Int$, let $I^\eps:=[b-\eps,d+\eps]\in \Int$.

	\begin{definition}[Interleaving distance]\label{def:interleaving} 
	Let $\eps\geq 0$ and $\mathbf{C}$ be an arbitrary category. Two functors $F,G:\U\rightarrow \mathbf{C}$ are said to be \emph{$\eps$-interleaved} if there exist collections of morphisms $f=(f_I:F(I)\rightarrow G(I^{\eps}))_{I\in \U}$ and $g=(g_J:G(J)\rightarrow F(J^{\eps}))_{J\in \U}$ satisfying the following: \begin{enumerate}
			\item For all $I,J\in \U,$\hspace{3mm}   $g_{I^{\eps}}\circ f_I= F(I\leq I^{2\eps}) \hspace{2mm}\mbox{and}\hspace{2mm} f_{J^{\eps}}\circ g_J= G(J\leq J^{2\eps}).$
			\item For all $I\leq J\in \U,$ $G(I^{\eps}\leq J^{\eps})\circ f_I=f_J\circ F(I\leq J)$ and \hbox{$F(I^{\eps}\leq J^{\eps})\circ g_I=g_J\circ G(I\leq J).$}
		\end{enumerate} In this case, we call $(f,g)$ an \emph{$\eps$-interleaving pair}.
		For $F,G\in \mathrm{Ob}(\mathbf{C}^\U),$ the interleaving distance between them is
		$$\dint(F,G):=\inf\{\eps\geq 0: F\ \mbox{and}\ G\ \mbox{are $\eps$-interleaved}\}.$$
		If $F$ and $G$ are not $\eps$-interleaved for any $\eps\geq0$, then $\dint(F,G)=+\infty$ by definition. 
		\end{definition}

\begin{figure}
    \centering
    \includegraphics[width=1\textwidth]{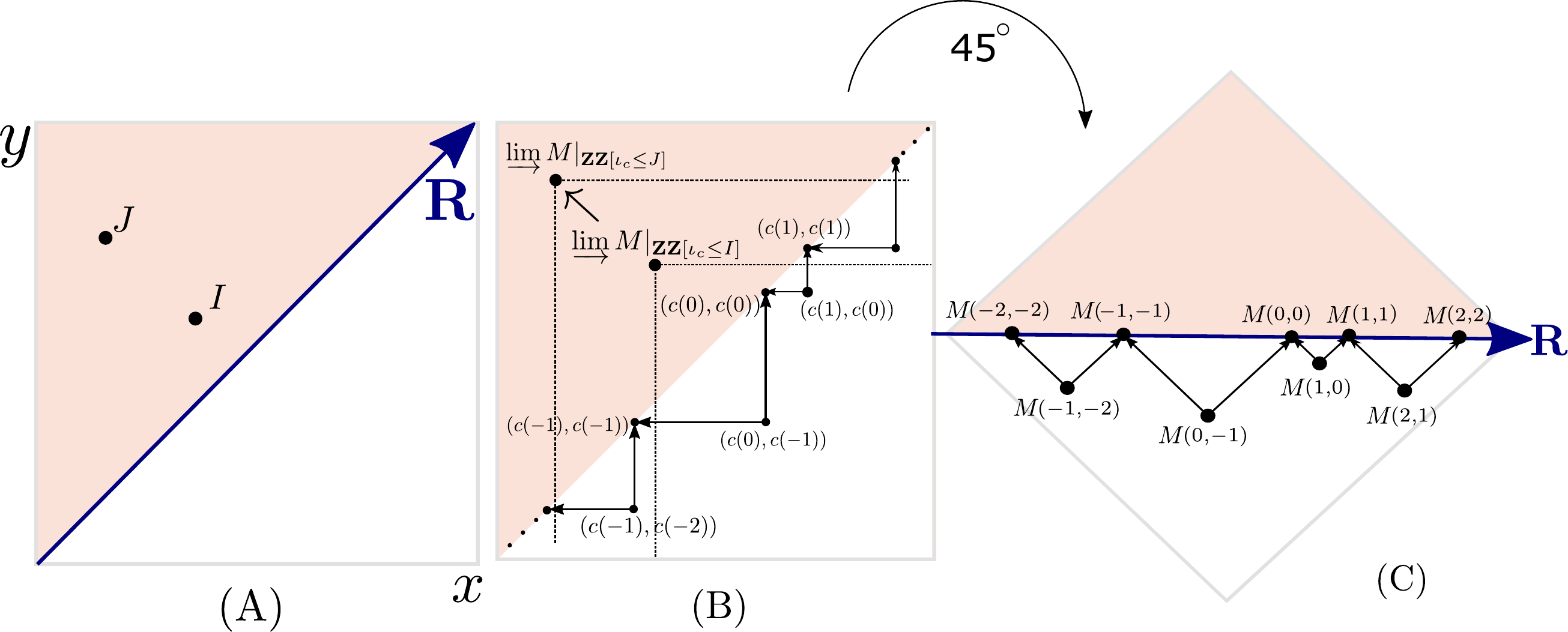}
    \caption{\textbf{Illustration for Definition \ref{def:constructible cosheaf}.} (A) The shaded region stands for the poset $\U$.  We have $I\leq J$ for the intervals $I,J\in \U$ shown in the figure. The diagonal line is identified with the real line via the bijection $(t,t)\leftrightarrow t$. (B) The canonical morphism $F(I)\rightarrow F(J)$ for $I\leq J$ in $\Int$. (C) The  zigzag diagram anchored over the real line completely determines the cosheaf $F:\Int\rightarrow \C$ defined by $I\mapsto \protect\varinjlim M|_{\ZZ[\iota_c\leq I]}$.}
    \label{fig:kan poset}
\end{figure}

\paragraph{Constructible cosheaves over $\R$ \cite{curry2020classification,curry2014sheaves,de2016categorified}.} 

We describe constructible cosheaves over $\R$ by adapting notation from \cite{botnan2018algebraic}. In a nutshell, a constructible cosheaf over $\R$ is a cosheaf over $\R$ \cite{bredon2012sheaf} which is fully determined by a certain ``zigzag diagram" over $\R$.

Given a strictly increasing function $c:\Z\rightarrow \R$ such that $\lim_{i\rightarrow\pm \infty}c(i)=\pm \infty$ and $I:=[b,d]\in \Int$, we define $$\ZZ[\iota_c\leq I]:=\{(i,j)\in \ZZ:(c(i),c(j))\leq (b,d) \mbox{ in } \Rop\times \R\}.$$ 
\ok{The following definition is depicted in Figure \ref{fig:kan poset}.}

\begin{definition}[Constructible cosheaves and critical points]\label{def:constructible cosheaf}A functor $F:\Int\rightarrow \C$ is called a \emph{constructible cosheaf} over $\R$ valued in $\C$, if there exist a strictly increasing function $c:\Z\rightarrow \R$ such that $\lim_{i\rightarrow\pm \infty}c(i)=\pm \infty$ and a  functor $M:\ZZ\rightarrow \C$ such that for all $I\in \Int$, $F(I)=\varinjlim M|_{\ZZ[\iota_c\leq I]}$ (cf. Figure \ref{fig:kan poset} (B)). For all $I\leq J$ in $\Int$, the morphism $F(I)\rightarrow F(J)$ is specified by the initial property of the colimit $F(I)$.  We call $\im(c)$ a set of \emph{critical points} of $F$.\footnote{The function $c$ associated to $F$ is not unique and thus we refer to $\im(c)$ as  `a' set of critical points and not as `the' set of critical points.}

\end{definition} 
Constructible cosheaves over $\R$ are fully characterized by ``zigzag diagrams over $\R$"  (cf. Figure \ref{fig:kan poset} (C)). 

 \begin{definition}[Names of cosheaves]\label{def:reeb graph as a cosheaf} \ok{We call a constructible cosheaf $F:\Int\rightarrow \C$ as given in Table \ref{table:names of cosheaves}, depending on the target category $\C$}. 
\end{definition}

    \begin{table}[t]
    \begin{center}
    \renewcommand{\arraystretch}{1.3}
\begin{tabular}{|l|l|l|l|l|}
\hline
                            & \cellcolor[HTML]{EFEFEF}$\C$ & \cellcolor[HTML]{EFEFEF}Name of constructible cosheaves $\Int\rightarrow \C$& \multicolumn{2}{l|}{\cellcolor[HTML]{EFEFEF}Metric} \\ \hline
(A) & $\graph(X)$                  & Dynamic graph*                             &  $\dintg$  & Def. \ref{def:DG interleaving distance}      \\ \hline
(B) & $\subpart(X)$                & Formigram*                                  &     $\dintf$ & Def. \ref{def:interleaving distance2}       \\ \hline
(C) & $\Part$                      & Unlabeled formigram                         &    $\dintf$  & Def. \ref{def:interleaving distance2} and Rem. \ref{rem:df generalizes dgh} \ref{item:df generalizes dgh1}    \\ \hline
(D) & $\wsets$                     & Weighted Reeb graph                          &     $\dintwreeb$ & Def. \ref{def:distance between weighted Reeb graphs}    \\ \hline
(E) & $\sets$                      & Reeb graph    \cite{botnan2018algebraic,de2016categorified}                    &    $\dintreeb$ & Def. \ref{def:interleaving} and Rem. \ref{rem:interleaving between Reeb}    \\ \hline
(F) & $\vect$                      & Interlevelset persistence module**                    &    $\dint^{\vect}$ or $\bott$ &  Def. \ref{def:interleaving} and \ref{def:bottleneck}    \\ \hline
\end{tabular}

\end{center}
\caption{\textbf{Objects in panels (A),(B),(C),(D),(E),(F) of Figure \ref{fig:entire picture} and their corresponding metrics.} \ok{Metrics in this table are totally ordered, i.e. in Figure \ref{fig:entire picture}, each process in (A)$\rightarrow$(B)$\rightarrow$(C)$\rightarrow$(D)$\rightarrow$(E)$\rightarrow$(F) is stable.} * Extra assumptions are necessary that are given in Definitions \ref{def:dyn graphs} and \ref{def:formigram}. **Visualized by zigzag barcode (Definition \ref{def:zigzag barcode}).}\label{table:names of cosheaves}
\end{table}

\begin{framed}
\begin{convention}\label{con:conversion-cosheaf}
In the rest of the paper, every functor $\Int\rightarrow \C$ will be assumed to be a constructible cosheaf and thus will be often simply called a cosheaf.
\end{convention}
\end{framed}

\paragraph{About cosheaves $\Int\rightarrow \C$ when $\C$ is a join-semilattice, $\vect$, or $\sets$.} Next we provide several remarks and review known results about constructible cosheaves $\Int\rightarrow \C$ when $\C$ is either a join-semilattice, $\vect$, or $\sets$.  

\begin{definition}[\Cosheaf{} map]\label{def:cosheaf conditions}
Let $\C$ be a join-semilattice  (Item \ref{item:lattice} in Section \ref{sec:category} and Remark \ref{rem:posets are categories} \ref{item:posets are categories1}). A map $f:\R\rightarrow \ob(\C)$ is said to \emph{be \cosheaf} if the following hold: There exists a locally finite set $C\subset \R$ of `critical points' such that (1) $f$ is constant between any two consecutive points in $C$, and (2) for every $c\in C$, $f(c)$ is \emph{locally maximal}, i.e. there exists $r_c>0$ with $f(c-\eps)\leq f(c) \geq f(c+\eps)$ for all $\eps\in [0,r_c)$.
\end{definition}

\begin{remark}\label{rem:join-semilattice cosheaf} Let $\C$ be a join-semilattice. Then
\begin{enumerate}[label=(\roman*),itemsep=-1ex]
    \item A constructible cosheaf $F:\Int\rightarrow \C$ induces the map $f:\R\rightarrow \ob(\C)$ defined by $t\mapsto F([t,t])$. Then, $f$ is \cosheaf{}. Note also that $F$ is recovered from $f$ via the formula $F(I)=\bigvee\{f(t):t\in I\}$ for $I\in \Int$.\footnote{$F([t,t])$ is said to be the \emph{costalk} of $F$ at $t$ in the literature (e.g. \cite{bredon2012sheaf,curry2014sheaves}).}\label{item:join-semilattice cosheaf1}
    \item Conversely, assume that a map $g:\R\rightarrow \ob(\C)$  is \cosheaf{}.  Then, $G:\Int\rightarrow \C$ defined by $I\rightarrow \bigvee\{G(t):t\in I\}$ is a constructible cosheaf.
  \label{item:join-semilattice cosheaf2}
\end{enumerate}

For $t\in \R$ and $\eps\geq 0$, let $[t]^\eps:=[t-\eps,t+\eps]$.

\begin{enumerate}[resume,label=(\roman*),itemsep=-1ex]
    \item \label{item:join-semilattice cosheaf3} For any two constructible cosheaves  $F,G:\Int\rightarrow \C$, the interleaving distance between them is:
\begin{align*}
    \dint(F,G)&=\inf\left\{\eps\geq 0:\mbox{for all $I\in \U$, } F(I)\leq G(I^\eps) \ \mbox{and } G(I)\leq  F(I^\eps) \right\}.
    \\&=\inf\left\{\eps\geq 0:\mbox{for all $t\in \R$, } F([t,t])\leq G([t]^\eps) \ \mbox{and } G([t,t])\leq  F([t,t]^\eps) \right\}.
\end{align*}. 
\end{enumerate}
\end{remark}

The following proposition follows from the fact that any $\ZZ$-indexed module is interval decomposable \cite{botnan2015interval,gabriel1972unzerlegbare}.

\begin{proposition}[\cite{botnan2018algebraic}]\label{prop:interval decomposability of constructible cosheaf}
    Any constructible cosheaf $F:\Int\rightarrow \vect$ is interval decomposable.
\end{proposition}

Let $\R_{\ell:y=x}$ be the diagonal line $y=x$ in $\R^2$. By Proposition \ref{prop:interval decomposability of constructible cosheaf}, we have:

\begin{definition}\label{def:zigzag barcode} Given any constructible cosheaf $F:\Int\rightarrow \vect$, the \emph{zigzag barcode} of $F$ is defined as 
\[\barc(F):=\B(F)\cap \R_{\ell:y=x}.\]
$\barc(F)$ will be regarded as a multiset of real intervals by identifying $\R_{\ell:y=x}$ with $\R$ via the bijection $(t,t)\leftrightarrow t$.
\end{definition}

\begin{remark}\label{rem:interleaving between Reeb} In Definition \ref{def:interleaving}, when $\C=\sets$ and $M,N$ are constructible, we obtain the interleaving distance $\dintreeb$ between Reeb graphs $M$ and $N$ \cite{botnan2018algebraic,de2016categorified}.
\end{remark}

Bauer, Ge and Wang \cite{bauer2014measuring} observed that the $0$-th levelset barcode of a Reeb graph encodes all non-trivial persistent homology information of the Reeb graph. A recent theorem on stability of the $0$-th levelset barcode of a Reeb graph is the following  (see also \cite{bauer2014measuring,bauer2015strong,botnan2018algebraic,de2016categorified,di2016edit} and cf. the arrow (E) $\rightarrow$ (F) in Figure \ref{fig:entire picture}).

\begin{theorem}[{\cite{ bjerkevik2016stability,botnan2018algebraic}}]\label{thm:reeb graph barcode stability} For any two Reeb graphs $M,N:\Int\rightarrow \sets$, we have that
\[\bott(\barc(\free\circ M),\barc(\free\circ N))\leq 2\cdot \dintreeb(M,N), \]
where $\free$ is the free functor $\Sets\rightarrow \Vect$ (Definition  \ref{def:free}) and $\bott$ is the bottleneck distance (Definition \ref{def:bottleneck}).
\end{theorem}

\paragraph{Smoothing of cosheaves.} The following definition is a straightforward adaptation of the notion of \emph{Reeb graph smoothing} \cite[Section 4.4]{de2016categorified}. 

\begin{definition}\label{def:cosheaf smoothing}For any $\eps\geq 0$, the \emph{$\eps$-smoothing} of a functor $F:\Int\rightarrow \C$ is the functor $S_\eps F:\Int\rightarrow \C$ defined by $I\mapsto F(I^\eps)$ and $(I\leq J) \mapsto F(I^\eps\leq J^\eps)$.
\end{definition}

In Appendix \ref{sec:smoothing} we study basic properties of this smoothing operation on formigrams.	

\begin{remark}\label{rem:smoothing forms semigroup}
For $\eps_1,\eps_2\geq 0$ and $I\in \Int$, since $I^{\eps_1+\eps_2}=(I^{\eps_1})^{\eps_2}$, we have  $S_{\eps_1+\eps_2} F=S_{\eps_1}S_{\eps_2}F$. This is a straightforward generalization of \cite[Proposition 4.13]{de2016categorified}.
\end{remark}

\begin{proposition}[{\cite[Propositions 4.16 and 4.17]{de2016categorified}}]\label{prop:smooting costructible yields constructible}
    If a given functor $F:\Int\rightarrow \C$ is constructible, then $S_\eps F$ is also constructible. In particular, if $C$ is a set of critical points of $F$, then one set of critical points of $S_\eps F$ is  $(C-\eps)\cup (C+\eps)$ where $C\pm\eps:=\{c\pm\eps\in\R:c\in C\}$.  
	\end{proposition}

	\section{From dynamic graphs to formigrams, Reeb graphs, and zigzag barcodes}\label{sec:DGs}
	\ok{In this section we define objects that appear in Figure \ref{fig:entire picture} (A)--(F).}  In Sections \ref{subsec:DGs} and \ref{sec:formi and barcode} we provide rigorous definitions of dynamic graphs and formigrams  respectively.  Formigrams are used to encode the evolution of connected components of dynamic graphs. In Section \ref{sec:the reeb} we define several invariants of formigrams. Throughout this section $X$ and $Y$ are nonempty finite sets.

	\subsection{Dynamic graphs (DGs)}\label{subsec:DGs}
	
	In this section we define the notion of \emph{dynamic graphs} (cf. Figure \ref{fig:entire picture} (A)) and also a suitable notion of isomorphism.

	\begin{definition}[Dynamic graphs]\label{def:dyn graphs}
	A \emph{dynamic graph (DG)} over $X$ is any \cosheaf{} map $\dynG_X:\R\rightarrow \graph(X)$ (cf. Definition \ref{def:cosheaf conditions}) which, in addition, also satisfies the condition that every $x\in X$ admits a closed interval lifespan $I_x\subset \R$, i.e. $x$ belongs to the vertex set of $\dynG_X(t)$ if and only if $t\in I_x$. When all $x\in X$ have the same lifespan, the dynamic graph $\dgx$ is said to be \emph{saturated} (e.g. the DG illustrated in Figure \ref{fig:entire picture} (A)). 
	\end{definition}
	
Remark \ref{rem:join-semilattice cosheaf} \ref{item:join-semilattice cosheaf2} allows us to view any DG as a constructible cosheaf over $\R$ valued in the lattice $\graph(X)$. This will in turn allow us to quantify the difference between DGs using the interleaving distance  (Remark \ref{rem:join-semilattice cosheaf}  \ref{item:join-semilattice cosheaf3}).

	\begin{example}\label{ex:special DGs} 
		Let us consider the two dynamic point clouds depicted in Figure \ref{fig:weakly isomorphic DMSs} (A) and (B).  For $\delta=1$, their time-varying $\delta$-Rips complexes are the DGs depicted in Figure \ref{fig:weakly isomorphic DMSs} (A') and (B')  (see Proposition \ref{prop:from DMS to DG} for a general statement).\footnote{In \cite{sinhuber2017phase}, this type of DGs was utilized for topological characterization of insect swarms. A sensor network \cite{de2006coordinate,de2007homological} is another example of such DGs arising from viewing each sensor as a point in the dynamic metric space of sensors.} 
		\label{item:special DG; derived from DMSs}
	\end{example}
	
	We now specify a suitable notion of isomorphism in the class of DGs.
	
	\begin{definition}[Isomorphism for DGs]\label{def:isom for graphs} Two DGs $\dynG_X$ and $\dynG_Y$ over $X$ and $Y$ respectively are \emph{isomorphic} if there exists a bijection $\varphi:X\rightarrow Y$ such that for all $t\in \T$, the map $\varphi$ serves as a graph isomorphism between $\dynG_X(t)$ and $\dynG_Y(t)$. Namely, for all $t\in \R$, $\varphi$ restricted to the vertex set of $\dynG_X(t)$ is a graph isomorphism from $\dynG_X(t)$ to $\dynG_Y(t)$.
	\end{definition}
	
	\begin{example}\label{ex:non-isomorphic DGs}
Let $\dgx$ and $\dgy$ be the two DGs in Example \ref{ex:special DGs}. Although the two graphs $\dgx(t)$ and $\dgy(t)$ are isomorphic for each $t\in \R$, it is not difficult to see that $\dgx$ and $\dgy$ \emph{are not isomorphic as DGs}. It is important to point out that $\dintg$ and also all the  invariants of DGs which we consider in this paper (see Figures \ref{fig:entire picture} and \ref{fig:maximal group diagram}) are able to discriminate these two DGs.
	\end{example}
	
\begin{figure}
    \centering
    \includegraphics[width=\textwidth]{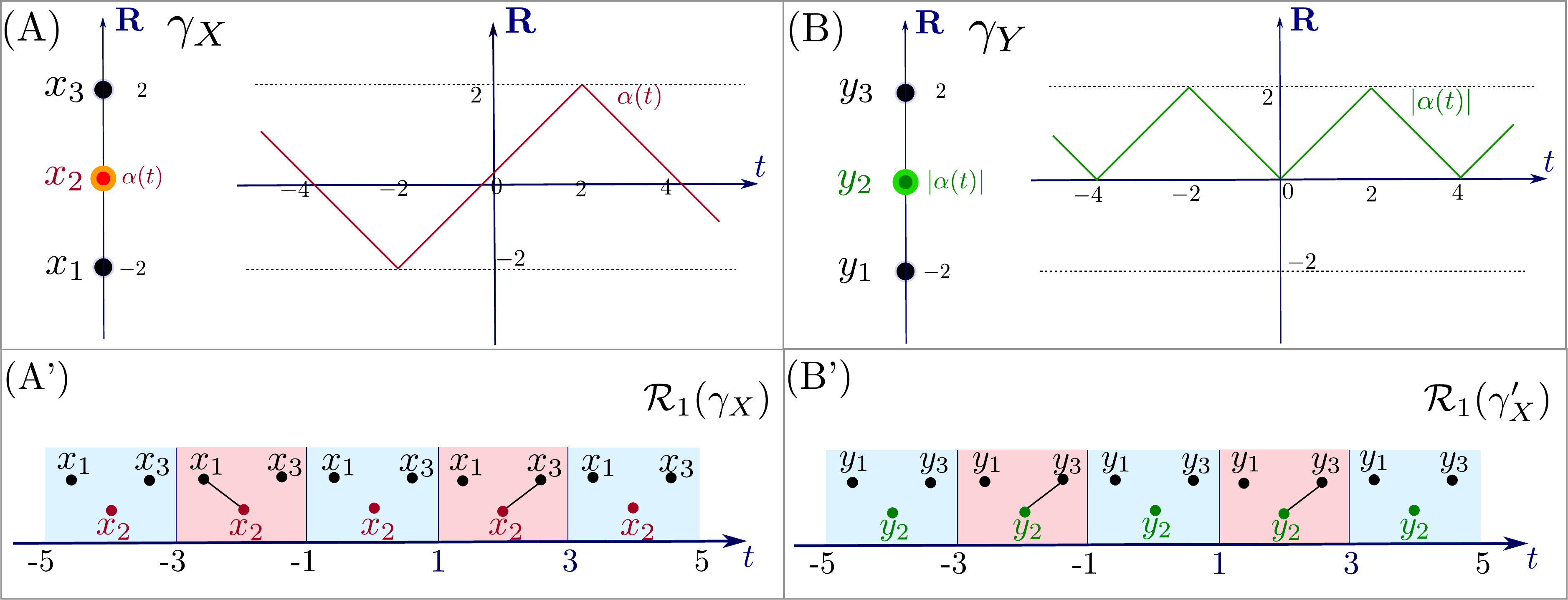}
    \caption{An illustration for Example \ref{ex:special DGs}. (A) and (B): Two dynamic point clouds $\gamma_X$ and $\gamma_Y$  each consisting of three points $x_1,x_2,x_3$ and $y_1,y_2,y_3$, respectively. While the two points $x_1$ and $x_3$ (resp. $y_1$ and $y_3$) are fixed at vertical coordinate values $-2$ and $2$, the other point $x_2$ (resp. $y_2$) moves according to $\alpha(t)$ (resp. $\abs{\alpha(t)}$). (A') and (B') show the $1$-Rips complexes $\mathcal{R}_1(\gamma_X(t))$ and $\mathcal{R}_1(\gamma_Y(t))$ for $t\in [-5,5]$.}
    \label{fig:weakly isomorphic DMSs}
\end{figure}

	\subsection{Formigrams}\label{sec:formi and barcode}
	 
	 We introduce a notion of \emph{formigram} to encode the evolution of connected components in dynamic graphs. 

	\begin{figure}
	\begin{center}
		\includegraphics[width=0.45\linewidth]{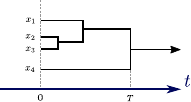}
	\end{center}
		\caption{\label{fig:dendro and tree} An example of a dendrogram over the set  $X=\{x_1,x_2,x_3,x_4\}$.}
	\end{figure}

	\begin{definition}\label{def:formigram} 
		A \emph{formigram}\footnote{The name formigram is a combination of the words {formicarium} and {diagram}.}  over a finite set $X$ is any \cosheaf{} map $\theta_X:\R \rightarrow \subpart(X)$ (cf. Definition  \ref{def:cosheaf conditions}) which, in addition, also satisfies that every $x\in X$ admits a  closed interval lifespan $I_x\subset \R$, i.e. $x$ belongs to $\bigcup\theta_X(t)$ if and only if $t\in I_x$. By the \emph{support} of $\theta_X$, denoted by $\supp(\theta_X)$, we mean the union $\bigcup_{x\in X} I_x$ of all lifespans of points in $X$. 
	\end{definition}

Remark \ref{rem:join-semilattice cosheaf} \ref{item:join-semilattice cosheaf2} allows us to view any formigram as a constructible cosheaf over $\R$ valued in the lattice $\subpart(X)$.  We will interchangeably write both $\theta_X:\R\rightarrow \subpart(X)$ and $\theta_X:\Int\rightarrow \subpart(X)$ as needed. 

Recall the unlabeling functor $\unlabel:\subpart(X)\rightarrow \Part$ from Definition \ref{def:three functors} \ref{item:unlabeling functor}.

\begin{remark}[Essentially unique labeling] \label{rem:trivial labeling}
    \ok{Given any unlabeled formigram $\theta:\Int\rightarrow \Partin$ (cf. Definition \ref{def:reeb graph as a cosheaf}), one can always find a formigram $\theta_X$ such that $\unlabel\circ \theta_X$. For example, in Figure \ref{fig:entire picture} (C), by labeling the eight `orbits' using any set $X$ of eight elements, we obtain such a formigram $\theta_X$. Assuming that the colimit of $\theta$ is isomorphic to a pair $(A,P_A)\in\ob(\Part)$, the size of $A$ equals the size of the labeling set.} 
\end{remark}

\begin{remark}\label{rem:about the definition of formigrams}
\ok{Formigrams  generalize the classical notion of \emph{dendrogram}, a 1-parameter nested family of partitions \cite{clustum,jardine-sibson}. Namely, any formigram $\theta_X:\R\rightarrow \subpart(X)$ is called a dendrogram if the following properties hold:
(1) every $x\in X$ has the lifespan $[0,\infty)$, (2) if $t_1\leq t_2$, then \hbox{$\theta_X(t_1) \leq \theta_X(t_2)$}, (3) there exists $T>0$ such that $\theta_X(t)=\{X\}$ for $t\geq T$. See Figure \ref{fig:dendro and tree} for an illustrative example.} 	
\end{remark}
	
\begin{definition} \label{def:special cases2} 
	Given a formigram $\theta_X:\R\rightarrow \subpart(X)$, 
	 we call $\theta_X$ \emph{saturated} if all $x\in X$ have the same lifespan \ok{(e.g. any dendrogram or the formigram that is depicted in Figure \ref{fig:entire picture} (B)).}
\end{definition}

\begin{definition}\label{def:isomorphism between formigrams}
	Two formigrams $\theta_X$ and $\theta_Y$ over $X$ and $Y$, respectively, are \emph{isomorphic} if there exists a bijection $\varphi:X\rightarrow Y$ such that for all $t\in \T$,  $\theta_X(t)=\varphi_\ast(\theta_Y(t)):=\{\varphi^{-1}(B)\subset X: B\in \theta_Y(t)\}$. 
\end{definition}

\ok{We remark that $\theta_X$ and $\theta_Y$ are isomorphic if and only if their unlabeled counterparts $\unlabel\circ\theta_X, \unlabel\circ\theta_Y:\Int\rightarrow \Partin$ are naturally isomorphic.}

\subsection{Functorial summarization of dynamic graphs and formigrams}\label{sec:the reeb}

We can summarize a dynamic graph at different levels by post-composing the functors from Table \ref{table:functors}. First of all, a formigram serves as a summary of the evolution of connected components in a DG via the path components functor $\pi_0:\graph(X)\rightarrow\subpart(X)$ (cf. Definition \ref{def:pi zero}):
	
	\begin{definition}\label{def:from DG to formi}
	    Given a DG $\dynG_X:\R\rightarrow \graph(X)$, the formigram of $\dynG_X$ is defined as $\pi_0\circ \dynG_X:\R\rightarrow \subpart(X)$ (cf. Figure \ref{fig:entire picture} (A) and (B)).
	\end{definition}

\begin{remark}\label{rem:from DG to formi}
    	We remark that the lifespan of any $x\in X$ in $\dynG_X$ is inherited by the formigram $\pi_0\circ \dynG_X$. In particular, the formigram of a saturated DG is itself saturated (Definitions \ref{def:dyn graphs} and \ref{def:special cases2}). 
\end{remark}

	\begin{figure}
		\includegraphics[width=\textwidth]{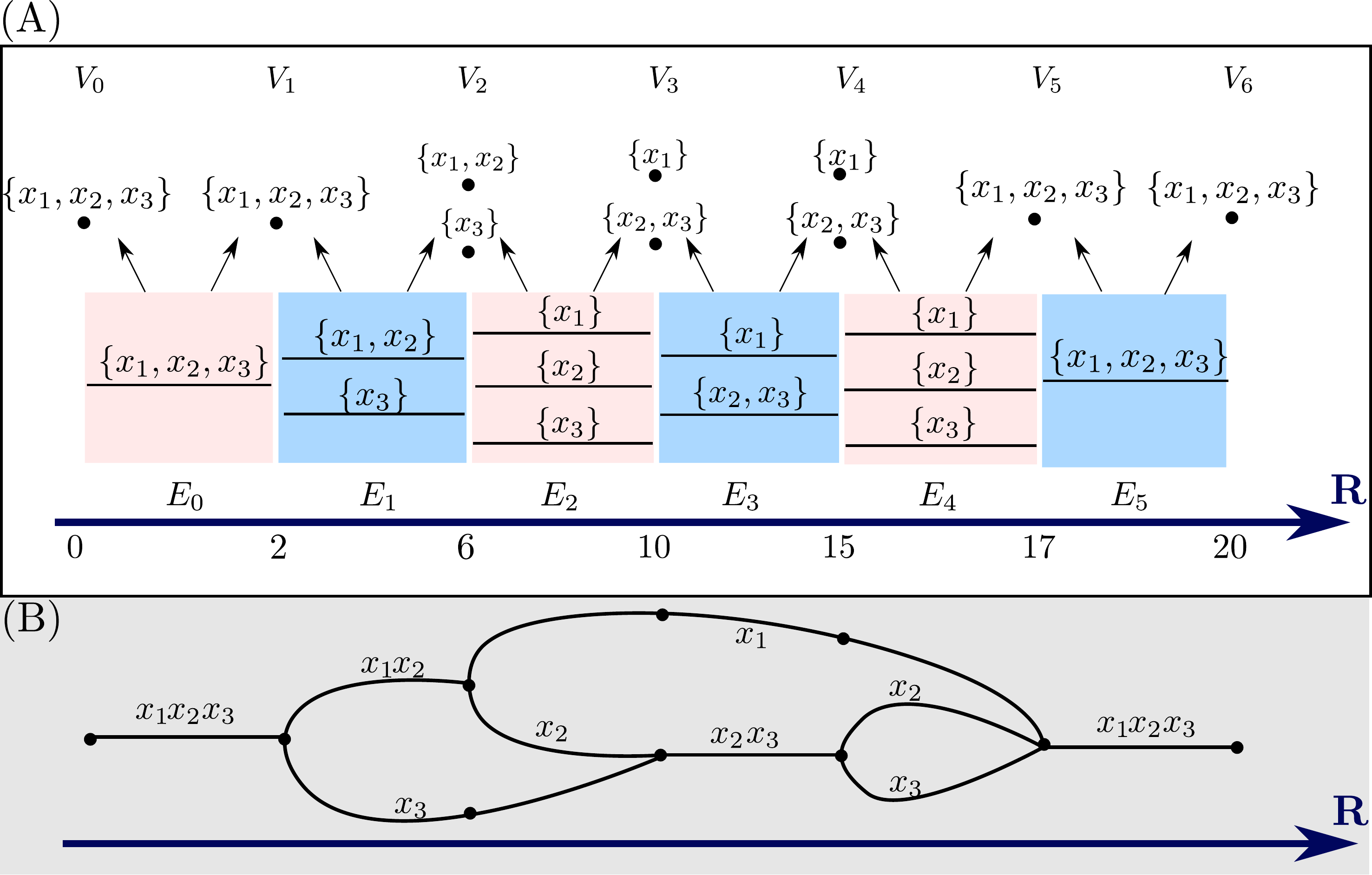}
		\caption{\label{fig:underlying} Visualization of the Reeb graph of a formigram $\theta_X:\R\rightarrow \subpart(X)$.   (A) The formigram $\theta_X$ is defined as follows over the interval $[0,20]$: Let $X:=\{x_1,x_2,x_3\}$. 
		$\theta_X(t)$ is $\{x_1x_2x_3\}$ for $t\in [0,2]\cup [17,20]$, it is $\{x_1x_2|x_3\}$ for $t\in (2,6]$, it is $\{x_1|x_2|x_3\}$ for $t\in (6,10)\cup (15,17)$, and it is $\{x_1|x_2x_3\}$ for $t\in [10,15]$. (B) depicts the Reeb graph of $\theta_X$ (with labels).}
	\end{figure}

	\ok{Figure \ref{fig:entire picture} (B)--(F) illustrates the following definition. Recall the three functors in Definition \ref{def:three functors}.}
	\begin{definition}[Summaries of formigrams]\label{def:underlying}	
	Let $\theta_X:\Int\rightarrow \subpart(X)$ be a formigram 
	\begin{enumerate}[label=(\roman*),itemsep=-1ex]
	    \item The \emph{unlabeled formigram} of $\theta_X$ is the cosheaf obtained by post-composing the unlabeling functor $\unlabel:\subpart(X)\rightarrow \Part$ to $\theta_X$.\label{item:unlabeled formigram}
	    \item  The \emph{(underlying) weighted Reeb graph} of $\theta_X$, denoted by $\omega(\theta_X)$, is the cosheaf obtained by post-composing the agglomeration functor $\A:\Part\rightarrow \wsets$ to the unlabeled formigram of $\theta_X$.\label{item:underlying weighted reeb} 
	  \item  The \emph{(underlying) Reeb graph} of $\theta_X$, denoted by $\reeb(\theta_X)$, is the cosheaf  obtained by post-composing the unweighting functor $\unweight:\wsets\rightarrow \sets$ to the weighted Reeb graph of $\theta_X$.\label{item:underlying reeb} 
	\end{enumerate}
			The above three items can sometimes be difficult to visualize for example due to possible non-planarity (cf. Example \ref{example:non-planar} below). This motivates us to further summarize formigrams into more easily visualizable invariants. One such invariant is the (zigzag) barcode \cite{zigzag}: 
	\begin{enumerate}[label=(\roman*),resume]
	  \item  The \emph{(zigzag) barcode} of $\theta_X$ is the zigzag barcode of the cosheaf $\Int\rightarrow \vect$ obtained by post-composing the free functor $\free:\sets\rightarrow \vect$ to the underlying Reeb graph of $\theta_X$. In other words, $\barc(\theta_X)$ is the $0$-th levelset barcode of the underlying Reeb graph of $\theta_X$. \label{item:zigzag barcode of formigram}
	\end{enumerate}
		\end{definition}

 \ok{We now describe how to visualize a formigram and its underlying weighted/unweighted Reeb graphs. The results of the visualization are topological graphs over the real line with or without labels;} see Figure \ref{fig:underlying} for an example. Let us fix a formigram $\theta_X:\R\rightarrow \subpart(X)$ with a set $C\subset \R$ of critical points (cf. Remark \ref{rem:join-semilattice cosheaf} \ref{item:join-semilattice cosheaf2}).

	\begin{enumerate}[label=Step \arabic*.]
				\item for each $c\in C$, we specify the vertex set $V_c:=\theta_X(c)$, which lie over $c\in \R$. 
				\item  for each pair of consecutive critical points $c_1<c_2$ in $C$, we specify the edge set $E_{c_1,c_2}:=\theta_X(t)$ for any $t\in (c_1,c_2)$, which lie over the interval $(c_1,c_2)$,
				\item for each pair of consecutive critical points $c_1<c_2$ in $C$, we define left and right attaching maps $l_{c_1,c_2}:E_{c_1,c_2}\rightarrow V_{c_1}$ and $r_{c_1,c_2}:E_{c_1,c_2}\rightarrow V_{c_2}$ by sending each block $B\in E_{c_1,c_2}$ to the blocks in $V_{c_1}$ and $V_{c_2}$ which contain $B$, respectively.
			\end{enumerate} 
As a result, we obtain a topological graph along the real line such as the one depicted in Figure \ref{fig:underlying} (B). By replacing the labeling of the elements of vertex sets and edge sets by their cardinalities, we obtain the weighted Reeb graph of $\theta_X$. Alternatively, weights can be represented by the thickness of nodes and edges as in Figure \ref{fig:entire picture} (D). By forgetting those weights, we obtain the Reeb graph of $\theta_X$.

	The following example shows that two different non-isomorphic formigrams  (Definition \ref{def:isomorphism between formigrams}) can have the same underlying weighted/unweighted Reeb graph.

	\begin{example}\label{ex:same underlying graph}  
	\ok{For the sets $X:=\{x_1,x_2,x_3\}$ and $Y:=\{y_1,y_2,y_3\}$, we consider the following two formigrams $\theta_X$ and $\theta_Y$ over $X$ and $Y$ respectively:
	\begin{align*}	
		\theta_X(t)&:=\begin{cases}\{x_1x_2|x_3\},&t\in(-3,-1)\cup(1,3)\\\{x_1x_2x_3\},\!\!&\mbox{otherwise},	\end{cases}\hspace{5mm}		
		\theta_Y(t):=\begin{cases}\{y_1y_2|y_3\},&t\in(-3,-1)\\ \{y_1|y_2y_3\},&t\in(1,3) \\\{y_1y_2y_3\},&\mbox{otherwise.}\end{cases}
		\end{align*}
		It is clear that $\theta_X$ and $\theta_Y$ are not isomorphic (cf. Definition \ref{def:isomorphism between formigrams}) whereas their weighted Reeb graphs are both isomorphic to the weighted Reeb graph depicted in Figure \ref{fig:non-planar} (A). This also implies that their \emph{unweighted} Reeb graphs are isomorphic.}
	\end{example}
Reeb graphs of formigrams are not always planar, which can make the visualization of a formigram difficult \cite{vehlow2015visualizing}:

\begin{example}\label{example:non-planar} Consider the formigram $\theta_X$ over the set $X=\{x_i\}_{i=1}^3\cup\{y_i\}_{i=1}^3\cup\{z_i\}_{i=1}^3$ given by:	
	\begin{align*}
		\theta_X(t)&:=\begin{cases}\left\{x_1x_2x_3|y_1y_2y_3| z_1z_2z_3\right\},&t\in (-\infty,1]
		\\\left\{x_1|x_2|x_3|y_1|y_2|y_3|z_1|z_2|x_3\right\},&t\in(1,2)
		\\\left\{x_1y_1z_1|x_2y_2z_2|x_3y_3 z_3\right\},&t\in[2,\infty).
		\end{cases}
	\end{align*}	
	See Figure \ref{fig:non-planar} (B) for the Reeb graph of $\theta_X$: this graph has the complete bipartite graph $K_{3,3}$ as a minor, which implies that it is not planar by Kuratowski's theorem \cite{bondy2008graph}.
\end{example}

	\begin{figure}
		\begin{center}
			\includegraphics[width=0.8\linewidth]{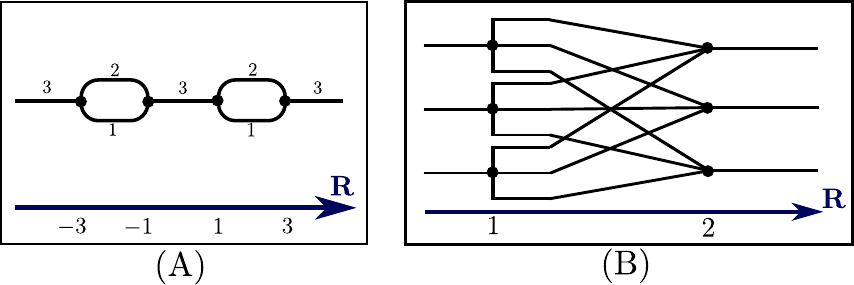}
			\caption{(A) The underlying weighted Reeb graph of the two formigrams from Example \ref{ex:same underlying graph}. (B) The underlying Reeb graph of the formigram from Example \ref{example:non-planar}. \label{fig:non-planar}}
		\end{center}
	\end{figure}
	
	\section{Interleavings between dynamic graphs and between formigrams }\label{sec:metrics-new}\label{sec:metrics}\label{sec:DG-metric}
	
	Given any nonempty finite set $X$, recall that $\graph(X)$ and $\subpart(X)$ are lattices. Therefore, we readily have the interleaving distance between two DGs over the \emph{same underlying set} $X$ or between two formigrams over the \emph{same underlying set} $X$ (cf. Remark \ref{rem:join-semilattice cosheaf} \ref{item:join-semilattice cosheaf3}). However, we often wish to quantify the difference between two given DGs (or between two formigrams) over possibly \emph{different} underlying sets. For achieving this, we blend ideas related to the Gromov-Hausdorff distance (Definition \ref{def:the GH}) with the interleaving distance. This type of idea has already appeared in the literature e.g. \cite{multi-clust,kim2020spatiotemporal,memoli2017distance,rolle2020stable}.

\paragraph{Tripods and their compositions.} We recall the notion of \emph{tripod} from \cite{memoli2017distance} as a preliminary to blending the Gromov-Hausdorff distance with the interleaving distance.

	\begin{definition}[Tripods]\label{def:tripod} Let $X$ and $Y$ be any two sets. A \emph{tripod} $R$ between $X$ and $Y$ is a pair of surjections from another set $Z$ to $X$ and $Y$, respectively. Namely, $R$ can be expressed as a diagram
		$R:\ X \xtwoheadleftarrow{\varphi_X} Z \xtwoheadrightarrow{\varphi_Y} Y. $
	\end{definition}
	
	For any sets $X,Y$ and $W$, consider any two tripods $\tripodone$ and $\tripodtwo$. Consider the set $Z:=\left\{(z_1,z_2)\in Z_1\times Z_2:\varphi_Y(z_1)=\psi_Y(z_2)\right\}$ and let $\pi_1:Z\rightarrow Z_1$ and $\pi_2:Z\rightarrow Z_2$ be the canonical projections to the first and the second coordinate, respectively. We define the composite tripod $R_2\circ R_1$ as follows:
	\begin{equation}\label{eq:comosition}
	R_2\circ R_1:X \xtwoheadleftarrow{\omega_X}\ Z \xtwoheadrightarrow{\omega_W}\ W,\ \mbox{where}\ \  \omega_X:=\varphi_X\circ\pi_1, \ \ \omega_W:=\psi_W\circ \pi_2.
	\end{equation}	
	\begin{center}	
		\begin{tikzcd}		
			&&Z\arrow[swap,two heads]{ld}{\pi_1}\arrow[two heads]{rd}{\pi_2}\\
			&Z_1\arrow[swap,two heads]{ld}{\varphi_X}\arrow[two heads]{rd}{\varphi_Y}&&Z_2\arrow[swap,two heads]{ld}{\psi_Y}\arrow[two heads]{rd}{\psi_W}
			\\X&&Y&&W
		\end{tikzcd}
	\end{center}

\begin{notation}\label{not:belongs to a tripod} Given a tripod $\tripod$, for $x\in X$ and $y\in Y$, we write $(x,y)\in R$ whenever there exists $z\in Z$ such that $\varphi_X(z)=x$ and $\varphi_Y(z)=y$.	\end{notation}

\subsection{Interleaving distance between dynamic graphs}

In this section we introduce the interleaving distance between DGs. Let $G_X=(X,E_X)$ be any graph and let $Z$ be any set. For any map $\varphi:Z\rightarrow X$, the pullback $G_Z:=\varphi^\ast G_X$ of $G_X$ via $\varphi$ is the graph on the vertex set $Z$ with the edge set $E_Z=\left\{\{z,z'\}: \left\{\varphi(z),\varphi(z')\right\}\in E_X\right\}.$ Let $\dyngx$ be a DG over $X$. The \emph{pullback} $\dynG_Z:=\varphi^*\dynG_X$ of $\dynG_X$ via $\varphi$ is a DG over $Z$ defined as follows: for all $t\in \T$, $\dynG_Z(t)$ is the graph on the vertex set $V_Z(t)=\varphi^{-1}\left(V_X(t)\right)$ with the edge set $E_Z(t)=\left \{\{z,z'\}: \left\{\varphi(z),\varphi(z')\right\}\in E_X(t) \right\}.$
	
Let $\dynG_X$ and $\dynG_Y$ be two DGs. Given a tripod $\tripod$, we write $\dynG_X\leq_R\dynG_Y$ if for all $t\in \R$ $\varphi_X^\ast\dynG_X(t)\leq \varphi_Y^\ast\dynG_Y(t)$ in the poset $\graph(Z)$. 
	
	\begin{remark}\label{rem:dg corrres} Let $\dynG_X$, $\dynG_Y$, and $\dynG_W$ be any three DGs. Let $R_1$ be any tripod between $X$ and $Y$ and let $R_2$ be any tripod between $Y$ and $W$. If $\dynG_X\leq_{R_1}\dynG_Y$ and $\dynG_Y\leq_{R_2}\dynG_W$, then it is easy to check that $\dynG_X\leq_{R_2\circ R_1}\dynG_W$ as well.
	\end{remark}
	
Recall that a DG $\dgx$ can be viewed as a constructible cosheaf $\Int\rightarrow \graph(X)$ (Definition \ref{def:dyn graphs} and Remark \ref{rem:join-semilattice cosheaf} \ref{item:join-semilattice cosheaf2}). This cosheaf sends each $I\in \Int$ to $\bigvee_I\dgx:=\bigvee\{\dgx(t):t\in I\}=\bigcup \{\dgx(t):t\in I\}$. Therefore, we can utilize the interleaving distance between cosheaves (Definition \ref{def:interleaving}) and tripods for quantifying the difference between two DGs over possibly different underlying sets.
	
	\begin{definition}\label{def:DG interleaving distance} 	Let $\dynG_X:\Int\rightarrow \graph(X)$ and $\dynG_Y:\Int\rightarrow \graph(Y)$ be two DGs. A tripod $\tripod$ is called an \emph{$\eps$-tripod between $\dynG_X$ and $\dynG_Y$} if $\dint^{\graph(Z)}(\varphi_X^\ast\dynG_X,\varphi_Y^\ast\dynG_Y)\leq \eps$. The interleaving distance between DGs $\dynG_X$ and $\dynG_Y$ is defined as
	\begin{align*}
	   \dintg(\dynG_X,\dynG_Y)=\min\{\eps\geq 0:\mbox{there exists an $\eps$-tripod between $\dynG_X$ and $\dynG_Y$}\}.
	\end{align*}
	If there is no $\eps$-tripod between $\dynG_X$ and $\dynG_Y$ for any $\eps\geq 0$, then we declare $\dintg(\dynG_X,\dynG_Y)=+\infty.$ 
	\end{definition}
	
	\begin{theorem}\label{thm:DG metric}$\dintg$ in Definition \ref{def:DG interleaving distance} is an extended pseudo metric on DGs.
	\end{theorem}
	We need the following lemma for proving Theorem \ref{thm:DG metric}.
	\begin{lemma}\label{lem:graph pullback and vee commute}Let $\varphi:Z\twoheadrightarrow X$ be a surjective map. Then, for any DG $\dgx$ and any $I\in \Int$, $\varphi^\ast\left(\bigvee_I \dgx\right)=\bigvee_I\varphi^\ast\dgx$.
	\end{lemma}
	
	\begin{proof}Fix $z,z'\in Z$ (it is possible that $z=z'$). Note that $\bigvee_I \dgx$ is the union $\cup_{t\in I} \dgx(t)$. We have that $\{z,z'\}\in \varphi^\ast\left(\bigvee_I \dgx\right)$ iff $\{\varphi(z),\varphi(z')\}\in \bigvee_I\dgx$ iff $\exists t\in I$, $\{\varphi(z),\varphi(z')\}\in \dgx(t)$ iff $\exists t\in I$, $\{z,z'\}\in \varphi^\ast\dgx(t)$ iff $\{z,z'\}\in \bigvee_I \varphi^\ast\dgx$.	\end{proof}
	
	\begin{proof}[Proof of Theorem \ref{thm:DG metric}]Reflexivity and symmetry of $\dintg$ are clear and thus we only show the triangle inequality: Let $X,Y$ and $W$ be some finite sets and let $\dgx,\dgy$  and $\dgw$ be DGs over the three sets respectively. We wish to prove that \hbox{$\dintg(\dynG_X,\dynG_W)\leq \dintg(\dynG_X,\dynG_Y)+\dintg(\dynG_Y,\dynG_W) $.}   
	Let $0<\eps_1, \eps_2<\infty$, and suppose that there are an $\eps_1$-tripod $\tripodone$ between $\dynG_X$ and $\dynG_Y$ and an $\eps_2$-tripod $\tripodtwo$ between $\dynG_Y$ and $\dynG_W$. 	It suffices to prove that $R_2\circ R_1$ is an $(\eps_1+\eps_2)$-tripod between $\dynG_X$ and $\dynG_W$. Let $I\in \Int$. Since $R_1$ is an $\eps_1$-tripod between $\dgx$ and $\dgy$, we have $\bigvee_I \varphi_X^\ast\dgx \leq \bigvee_{I^{\eps_1}}\varphi_Y^\ast \dgy$. Since $R_2$ is an $\eps_2$-tripod between $\dgy$ and $\dgw$, we have $\bigvee_{I^{\eps_1}} \psi_Y^\ast\dgy \leq \bigvee_{I^{\eps_1+\eps_2}}\psi_W^\ast \dgw$. Therefore, by Lemma \ref{lem:graph pullback and vee commute},
	\[\bigvee_I\pi_1^\ast\varphi_X^\ast \dgx \leq \bigvee_{I^{\eps_1}}\pi_1^\ast\varphi_Y^\ast\dgy=\bigvee_{I^{\eps_1}}\pi_2^\ast\psi_Y^\ast\dgy\leq \bigvee_{I^{\eps_1+\eps_2}}\pi_2^\ast \psi_W^\ast \dgw.\]
	By symmetry we also have $\bigvee_{I}\pi_2^\ast \psi_W^\ast \dgw\leq \bigvee_{I^{\eps_1+\eps_2}}\pi_1^\ast\varphi_X^\ast \dgx $. Since $I\in \Int$ is arbitrary, we have shown that $R_2\circ R_1$ is an $(\eps_1+\eps_2)$-tripod between $\dgx$ and $\dgw$, as desired.
	\end{proof}
	
	Given a DG $\dgx:\R\rightarrow \graph(X)$, let $\crit(\dgx)$ denote the set of points of discontinuity, i.e. critical points (cf. Definition \ref{def:cosheaf conditions}). 
	
	\begin{theorem}[Complexity of $\dintg$]\label{thm:DG metric-complexity} 	Fix $\rho\in(1,6)$. Then, it is not possible to compute a $\rho$-approximation to $\dintg(\dynG_X,\dynG_Y)$ between DGs in time polynomial in $|X|,|Y|,|\crit(\dynG_X)|,$ and $|\crit(\dynG_Y)|$, unless $P=NP$.
	\end{theorem}
	
	We prove Theorem \ref{thm:DG metric-complexity} in Appendix \ref{proof:DG-metric-complexity} by showing that certain NP-hard instances of the Gromov-Hausdorff distance between finite metric spaces can be reduced to the computation of the interleaving distance $\dintg$ between DGs.

	\begin{remark}[When is $\dintg\leq \eps$]\label{rem:basics for DG} 
In Definition \ref{def:DG interleaving distance}, the condition $\dint^{\graph(Z)}(\varphi_X^\ast\dynG_X,\varphi_Y^\ast\dynG_Y)\leq \eps$ with respect to the tripod $\tripod$ is equivalent to the following: For $(x,y),(x',y')\in R$, 
		
		\begin{enumerate}[label=(\roman*)]
			\item if $x\in V_X(t)$, then there exists  $s\in [t]^\eps:=[t-\eps,t+\eps]$ such that $y\in  V_Y(s).$  \label{item:basics for DG 1}
			\item if $\{x,x'\}\in E_X(t)$, then there exists  $s\in [t]^\eps$ such that $\{y,y'\} \in E_Y(s)$. \label{item:basics for DG 2} 
		\end{enumerate}
		Furthermore, the two statements obtained by exchanging the roles of $X$ and $Y$ in the above two items also hold.
	\end{remark}	
	
A sufficient condition for a pair of DGs to be $\eps$-interleaved is described in the following example.
	\begin{example}Fix $\eps\geq 0$. Let $\dgx$ and $\dgy$ be two DGs such that for every $x,x'\in X$ (resp. every $y,y'\in Y$), for any time interval $I$ of length $\eps$, there exists $t\in I$ such that the edge $\{x,x'\}$ (resp. $\{y,y'\}$) is present at $\dynG_X(t)$ (resp. $\dynG_Y(t)$). Then, $\dintg(\dynG_X,\dynG_Y)\leq \eps$. Indeed, by invoking Remark \ref{rem:join-semilattice cosheaf} \ref{item:join-semilattice cosheaf3}, it can be checked that \emph{any} tripod $R$ between $X$ and $Y$ is an  $\eps$-tripod between $\dynG_X$ and $\dynG_Y$.
	\end{example}

	\subsection{Interleaving distance between formigrams }\label{sec:formi-metric}
		In this section we introduce the interleaving distance between formigrams. This metric quantifies the structural difference between two grouping/disbanding behaviors over time. We also show that this metric dovetails with other celebrated metrics, which makes the summarization pipeline that is illustrated in Figure \ref{fig:entire picture} is entirely stable; see Remark \ref{rem:df generalizes dgh}, Proposition \ref{prop:formigram interleaving is mor discriminative than the reeb interleaving}, Corollary \ref{cor:df-db stability}, and Theorem \ref{prop:stability from dg to formi}. 
	
	Let $X$ and $Z$ be any two sets and let $P_X\in \subpart(X).$ For any map $\varphi:Z\rightarrow X$, the \emph{pullback} $P_Z:=\varphi^\ast P_X$ of $P_X$ via $\varphi$ is the subpartition of $Z$ defined as $P_Z=\{\varphi^{-1}(B): B\in P_X\}.$ Let $\theta_X$ be a formigram over $X$ and assume that $\varphi$ is surjective. Then the \emph{pullback} of $\theta_X$ via $\varphi$ is the formigram $\theta_Z:=\varphi^*\theta_X$ over $Z$ defined as $\theta_Z(t)= \varphi^*\theta_X(t)$ for all $t\in\T.$

	\begin{definition}[Comparison of  formigrams via pullbacks]\label{def:partition morphism}  Consider a tripod $\tripod$.
	\begin{enumerate}[label=(\roman*),itemsep=-1ex]
	    \item Let $P_X\in \subpart(X)$ and $P_Y\in \subpart(Y)$. we write $P_X\leq_R P_Y$ if $\varphi_X^\ast P_X\leq \varphi_Y^\ast P_Y$ in the poset $\subpart(Z)$. 
	    \item Let $\theta_X, \theta_Y$ be any two formigrams. we write $\theta_X\leq_R \theta_Y$ if for all $t\in \R$, $\varphi_X^\ast\theta_X(t)\leq \varphi_Y^\ast\theta_Y(t)$ in $\subpart(Z)$.
	\end{enumerate}
	
	\end{definition}

Let us recall that a formigram $\theta_X$ can be viewed as a constructible cosheaf $\Int\rightarrow \subpart(X)$ (Remark \ref{rem:join-semilattice cosheaf} \ref{item:join-semilattice cosheaf2}); this cosheaf sends each $I\in \Int$ to $\bigvee_I\theta_X:=\bigvee\{\theta_X(t):t\in I\}$. We now utilize both the interleaving distance between cosheaves (Definition \ref{def:interleaving}) 
and the notion of tripod for quantifying the degree of difference between two formigrams over possibly different underlying sets.

\begin{definition}	\label{def:interleaving distance2}
	Let $\theta_X$ and $\theta_Y$ be two formigrams. A tripod $\tripod$ is called an \emph{$\eps$-tripod between $\theta_X$ and $\theta_Y$} if $\dint^{\subpart(Z)}(\varphi_X^\ast\theta_X,\varphi_Y^\ast\theta_Y)\leq \eps$. The \emph{interleaving distance} between formigrams $\theta_X$ and $\theta_Y$ is defined as
	   \[\dintf(\theta_X,\theta_Y)=\min\{\eps\geq 0:\mbox{there exists an $\eps$-tripod between $\theta_X$ and $\theta_Y$}\}.\]
	If there is no $\eps$-tripod between $\theta_X$ and $\theta_Y$ for any $\eps\geq 0$, then we declare $\dintf(\theta_X,\theta_Y)=+\infty.$ 
\end{definition}

\ok{$\dintf$ is is an extended pseudo metric on formigrams, which can be proved in a similar way to $\dintg$.}

    \begin{remark}\label{rem:df generalizes dgh} \begin{enumerate}[label=(\roman*)]
        \item \ok{For unlabeled formigrams $\theta,\theta':\Int\rightarrow \Part$ (cf. Definition \ref{def:reeb graph as a cosheaf}), the interleaving distance $\dint^\Part(\theta,\theta')$ (cf. Definition \ref{def:interleaving}) can easily be infinite and thus might not be useful in practice; it is not difficult to check that $\dint^\Part(\theta,\theta')$ is finite only if the colimits of $\theta$ and $\theta'$ have the same cardinality. In this respect, by Remark \ref{rem:trivial labeling}, $\dintf$ can also be viewed as a  practical metric between unlabeled formigrams by defining \begin{equation}\label{eq:distance between unlabeled formigrams}
            \dintf(\theta,\theta'):=\dintf(\theta_X,\theta_Y)
        \end{equation} where $\theta=\unlabel\circ \theta_X$ and $\theta'=\unlabelY\circ\theta_Y$. The choices of the labeling sets $X$ and $Y$ do not affect the RHS of equation (\ref{eq:distance between unlabeled formigrams}) and thus $\dintf(\theta,\theta')$ is well-defined. }      \label{item:df generalizes dgh1}
        \item $\dintf$ between dendrograms agrees with (twice) the Gromov-Hausdorff distance between their canonically associated ultrametrics (cf. Proposition \ref{prop:ultra}).\label{item:df generalizes dgh2}
    \end{enumerate} 
	\end{remark}

	$\dintf$ is more discriminative than the interleaving distance for Reeb graphs (cf. Remark \ref{rem:interleaving between Reeb}).
	
	\begin{proposition}\label{prop:formigram interleaving is mor discriminative than the reeb interleaving}For any two formigrams $\theta_X$ and $\theta_Y$, we have
	$$\dintreeb(\reeb(\theta_X),\reeb(\theta_Y))\leq \dintf(\theta_X,\theta_Y).$$
	\end{proposition}
	\ok{In the appendix, we define a distance $\dintwreeb$ between weighted Reeb graphs which mediates between $\dintf$ and $\dintreeb$ (cf. Theorem \ref{thm:dintwreeb mediates} and Remark \ref{rem:tree of distances}).}
	
	\begin{example}\label{rem:dintf is not dint} Since $\theta_X$ and $\theta_Y$ in  Example \ref{ex:same underlying graph} have the same \ok{(un)weighted Reeb graph, they are not discriminated by the interleaving distance between (un)weighted Reeb graphs.  However, $\dint^\mathrm{F}(\theta_X,\theta_Y)= 1.$} Indeed, any tripod $R$ between $X$ and $Y$ fails to be an $\eps$-tripod between $\theta_X$ and $\theta_Y$ for $\eps<1$ and the tripod $\tripod$ which is defined as follows is a $1$-tripod: Let  $Z=\{(x_1,y_1),(x_2,y_2),(x_2,y_3)\}\subset X\times Y$ and let $\varphi_X$ and $\varphi_Y$ be the canonical projections to the first coordinate and the second coordinate, respectively.
	\end{example}
	
 Let us recall from Definition \ref{def:subpart} that the underlying set of $P\in \subpart(X)$ is defined as $\bigcup P$.  When $x\in X$ belongs to $\bigcup P$, we denote the block containing $x$ by $[x]$. The following remark \ok{describes an if-and-only-if condition for a tripod to be an $\eps$-tripod between formigrams.}
 
		\begin{remark}\label{prop:morphism for formi}Let $\theta_X, \theta_Y$ be any two formigrams and consider a tripod $\tripod$. Then the condition $\dint^{\subpart(Z)}(\varphi_X^\ast\theta_X,\varphi_Y^\ast\theta_Y)\leq \eps$ holds if and only if the following hold: for any $I\in \Int$,
		\begin{enumerate}[label=(\roman*)]
			\item If $(x,y)\in R$, $x$ belongs to the underlying set of $\bigvee_I\theta_X$ then $y$ belongs to the underlying set of $\bigvee_{I^\eps}\theta_Y$.   \label{item:morphism for formi1}
				\item If $(x,y),(x',y')\in R$, whenever $[x]=[x']$ in $\bigvee_I\theta_X$, it holds that $[y]= [y']$ in $\bigvee_{I^\eps}\theta_Y$.  \label{item:morphism for formi2}
		\end{enumerate}
			Also, the two statements obtained by exchanging the roles of $X$ and $Y$ in the above two items hold.
	\end{remark}

	\begin{proof}[Proof of Proposition \ref{prop:formigram interleaving is mor discriminative than the reeb interleaving}]Let $\tripod$ be an $\eps$-tripod between $\theta_X$ and $\theta_Y$. For $I\in \Int$, let us define the set map $f_{R,I}:\bigvee_I\theta_X\rightarrow \bigvee_{I^\eps}\theta_Y$ as $f_{R,I}([x]):=[y]\in \bigvee_{I^\eps}\theta_Y$ where $(x,y)\in R$ (this map is well-defined by Remark \ref{prop:morphism for formi}). Also, define $g_{R,I}:\bigvee_I\theta_Y\rightarrow \bigvee_{I^\eps}\theta_X$ in a similar way. The collections $f_R:=\{f_{R,I}\}_{I\in \Int}$  and $g_R:=\{g_{R,I}\}_{I\in \Int}$ form an $\eps$-interleaving pair between $\reeb(\theta_X)$ and $\reeb(\theta_Y)$, thus completing the proof. 
	\end{proof}
	
	\begin{corollary}\label{cor:df-db stability} For any two formigrams $\theta_X$ and $\theta_Y$, we have $$\bott(\barc(\theta_X),\barc(\theta_Y))\leq 2\cdot \dintf(\theta_X,\theta_Y).$$
	\end{corollary}
	\begin{proof} The claim directly follows from Theorem \ref{thm:reeb graph barcode stability} and Proposition \ref{prop:formigram interleaving is mor discriminative than the reeb interleaving}.
	\end{proof}

Summarizing DGs into formigrams via the path components functor $\pi_0$ (Definition \ref{def:from DG to formi}) is stable:
	
	\begin{theorem}\label{prop:stability from dg to formi}Let $\theta_X$ and $\theta_Y$ be the formigrams of two DGs $\dynG_X$ and $\dynG_Y$ respectively. Then,
		\[\dint^\mathrm{F}(\theta_X,\theta_Y)\leq \dintg(\dynG_X,\dynG_Y).\]
	\end{theorem}
	
		This inequality is tight (cf. Remark \ref{rem:DG and formi distance computation} \ref{item:DG and formi distance computation1}). 

\begin{proof}
Let $\eps\geq 0$ and let $\tripod$ be an $\eps$-tripod between $\dynG_X$ and $\dynG_Y$. We prove that $R$ is also an $\eps$-tripod between the formigrams $\theta_X$ and $\theta_Y$. By symmetry, it suffices to show that $\bigvee_I\theta_X\leq_R \bigvee_{I^\eps}\theta_Y$. Let $z\in Z$ and let $x:=\varphi_X(z)$, and $y:=\varphi_Y(z)$. Fix $t\in \T$. Suppose that $x$ is in the underlying set of $\bigvee_I\theta_X$. Since $\bigvee_I\dynG_X\leq_R \bigvee_{I^\eps}\dynG_Y$, we have $y\in \cup_{s\in I^\eps}V_Y(s)$, which is the underlying set of $\bigvee_{I^\eps}\theta_Y$. Pick another $z'\in Z$ and let $x':=\varphi_X(z')$ and $y':=\varphi_Y(z')$. Assume that $x,x'$ belong to the same block of $\theta_X(t)$, meaning that there is a sequence $x=x_0,x_1,\ldots,x_n=x'$ in $X$ such that $\{x_i,x_{i+1}\}\in E_X(t)$ for \hbox{$0\leq i\leq n-1$.}  For each $1\leq i\leq n-1$, pick $y_i\subset Y$ such that $(x_i,y_i)\in R$. Since $R$ is an $\eps$-tripod between $\dynG_X$ and $\dynG_Y$, we have \hbox{$\{y_i,y_{i+1}\}\in \bigcup_{s\in I^\eps}E_Y(s)$} (Remark \ref{rem:basics for DG}). Then, $y,y'$ belong to the same connected component of the graph $\bigcup_{I^\eps}\dynG_Y$ and in turn,
by Lemma \ref{lem:graph pullback and vee commute}, the same block of $\bigvee_{I^\eps}\theta_Y$. 
\end{proof}		

    \begin{remark}\label{rem:complexity of df} By Remark \ref{rem:df generalizes dgh} \ref{item:df generalizes dgh2}, constant factor approximations to  $\dintf$ cannot be obtained in polynomial time (Theorem \ref{thm:complexity-dIF}). 	\ok{The lower bound for $\dintf$ given by Proposition \ref{prop:formigram interleaving is mor discriminative than the reeb interleaving} can also lead to difficult computational problems (e.g. the graph-isomorphism problem \cite[Section 5]{de2016categorified}). Hence, computing the lower bound for $\dintf$ given by Corollary \ref{cor:df-db stability} is a realistic approach to comparing formigrams. Another tractable lower bound for $\dintf$ is introduced in the next section; Theorem \ref{thm:stability of betti}, which can sometimes be more discriminative than the lower bounds in Proposition \ref{prop:formigram interleaving is mor discriminative than the reeb interleaving} and Corollary \ref{cor:df-db stability}.}
	\end{remark}

\section{Categorical aspects of $\subpart(X)$, $\Part$ and $\sets$}
\label{sec:about partition category}

\ok{In this section we establish a few categorical results that will be useful in later sections: We show that (1) the unlabeling functor preserves limits and colimits of connected diagrams (cf. Proposition \ref{prop:unlabeling preserves limits and colimits}) and that (2)  the composition of the three functors in Definition \ref{def:three functors} preserves colimits of connected diagrams (cf. Proposition \ref{prop:colimit}). Lastly, we compute coimages in the category $\Part$ (cf. Proposition \ref{prop:image and coimage}).}
	
\subsection{(Co)limit preserving properties of unlabeling and collapsing functors}\label{sec:limits and colimits in partition category}

Let $X$ be a nonempty finite set. The unlabeling functor $\unlabel:\subpart(X)\rightarrow \Part$ (Definition \ref{def:three functors} \ref{item:unlabeling functor})  preserves limits and colimits of connected diagrams in $\subpart(X)$:

\begin{proposition}\label{prop:unlabeling preserves limits and colimits} Let $\Pb$ be a connected poset. For any $\theta_X:\Pb\rightarrow \subpart(X)$, we have $\unlabel\left(\varprojlim \theta_X\right)\cong \varprojlim \unlabel(\theta_X)$ and  $\unlabel\left(\varinjlim \theta_X\right)\cong \varinjlim \unlabel(\theta_X)$. 
\end{proposition}

\ok{The proof is rather straightforward and hence we omit it.}
	
	\begin{definition}\label{def:natural} Let $P,Q\in \subpart(X)$ such that $P\leq Q$. The \emph{canonical map} $P\rightarrow Q$ is the unique map which sends each block $B\in P$ to the \emph{unique} block $C\in Q$ such that $B\subset C$. \ok{In other words, the canonical map is the image of $P\leq Q$ via the three functors in Definition \ref{def:three functors}.}
\end{definition}
    
    \begin{definition}\label{def:forget from subpart to set} The \emph{\collapsing{}} $\emb:\subpart(X)\rightarrow \sets$ is defined as the composition of the three functors in Definition \ref{def:three functors}, i.e. $\unweight\circ \A \circ \unlabel$. Namely, $\emb$ sends each $P\in\subpart(X)$ to $P\in \ob(\sets)$ and each morphism $P\leq Q$ in $\subpart(X)$ to the canonical map $P\rightarrow Q$. 
    \end{definition}
    
    The \collapsing{} preserves the colimit of any diagram in $\subpart(X)$ indexed by a \emph{connected} poset (Definition \ref{interval} \ref{item:connectivity}). For example, let $X:=\{x_1,x_2,x_3,x_4\}$ and consider the diagram $\theta_X:\{1\leq 2\geq 3\leq 4\}\rightarrow \subpart(X)$ defined as 
    \[\{x_1|x_2|x_3|x_4\}\leq \{x_1x_2|x_3x_4\}\geq \{x_1|x_2|x_3x_4\}\leq\{x_1|x_2x_3x_4\}.\]
    Then, $\varinjlim \theta_X =\{x_1x_2x_3x_4\}\in \subpart(X)$ and $\emb \left(\varinjlim \theta_X\right) =\{\{x_1,x_2,x_3,x_4\}\}\in \ob(\sets)$, which is the colimit of the set diagram 
    \[\{\{x_1\},\{x_2\},\{x_3\},\{x_4\}\}\rightarrow \{\{x_1,x_2\},\{x_3,x_4\}\}\leftarrow \{\{x_1\},\{x_2\},\{x_3,x_4\}\}\rightarrow\{\{x_1\},\{x_2,x_3,x_4\}\}\]
    where every map in the cocone is a canonical map. We omit the proof of the following proposition.
    
    \begin{proposition}\label{prop:colimit} Let $\Pb$ a connected poset. For any $\theta_X:\Pb\rightarrow \subpart(X)$, we have $\emb \left(\varinjlim \theta_X\right) \cong \varinjlim \emb \left(\theta_X\right)$.
    \end{proposition}

\ok{We remark that the collapsing functor $\emb$ does \emph{not} preserve limits (which is actually a consequence that the agglomeration functor $\A$ does not preserve limits). For example, let $Y=\{y_1,y_2\}$ and consider the diagram $\theta_Y:\{1\leq 2\geq 3\}\rightarrow \subpart(Y)$ given by $\{y_1\}\leq \{y_1y_2\} \geq \{y_2\}.$ Then, $\varprojlim \theta_Y=\emptyset$ which is the greatest lower bound of $\{y_1\},\{y_1y_2\},\{y_2\}$ in $\subpart(Y)$. However, $\emb(\theta_X)$ is represented as the set diagram $\{\bullet\}\rightarrow \{\bullet\}\leftarrow \{\bullet\}$ and thus $\varprojlim \emb(\theta_X)$ contains a single element.}

\subsection{(Co)images in the category of partitions}\label{sec:Images and coimages in the category of partitions}

The goal of this section is to compute coimages in $\Part$ (cf. Proposition \ref{prop:image and coimage}).

\paragraph{Images and coimages in category theory \cite{mitchell1965theory}.} A morphism $f:a\rightarrow b$ is said to be a \emph{monomorphism} (\emph{mono} in short) if $f$ is left-cancellative: for any morphisms $k_1, k_2: c \rightarrow a$, if $f\circ k_{1}=f\circ k_{2}$, then $k_{1}=k_{2}$. Such $f$ is written as $f:a\hookrightarrow b$ and $a$ is called a \emph{subobject} of $b$. On the other hand, a right-cancellative morphism $g:a\rightarrow b$ is said to be an \emph{epimorphism} (\emph{epi} in short), written as $g:a\twoheadrightarrow b$, and $b$ is called a \emph{quotient object} of $a$.

The \emph{image} of a morphism $f:X\rightarrow Y$ is defined as the smallest subobject of $Y$ which $f$ factors through.

\begin{definition}[Images]\label{def:images}
Given a morphism $f:X\rightarrow Y$, an \emph{image} of $f$ (if it exists) is a mono $m:I\hookrightarrow Y$ such that there is a morphism $f_m:X\rightarrow I$ such that $f=m\circ f_m$, for any mono $z:Z\hookrightarrow Y$ and a morphism $f_z:X\rightarrow Z$ with $f=z\circ f_z$, there is a unique morphism $u:I\rightarrow Z$ such that $m=z\circ u$.

\[\begin{tikzcd}X\arrow[->]{rr}{f} \arrow{rd}{f_m} \arrow{rdd}[swap]{f_z}&&Y\\
& I \arrow[hook]{ru}{m}\arrow[dashed]{d}{u}\\
& Z \arrow[hook]{ruu}[swap]{z}
\end{tikzcd}
\]
The object $I$ will be sometimes denoted by $\im(f)$.
\end{definition}

The \emph{coimage} of a morphism is the dual notion of the image of a morphism. In many categories (such as the category of sets, the category of groups, the category of rings, or the category of vector spaces, and etc.), the coimage is canonically isomorphic to the image, often called the \emph{first isomorphism theorem}. However, in the category $\Part$, those two are turned out to be different in general.

\begin{definition}[Coimages]\label{def:coimages}
 Given a morphism $f:X\rightarrow Y$, a coimage of $f$ (if it exists) is an epi $ c:X\rightarrow C$ such that there is a morphism $f_{c}:C\rightarrow Y$ with $ f=f_{c}\circ c$,
for any epi $z:X\rightarrow Z$ for which there is a map $\displaystyle f_{z}:Z\rightarrow Y$ with $f=f_{z}\circ z$, there is a unique map $u:Z\rightarrow C$ such that $c=u\circ z$. 

\[\begin{tikzcd}X\arrow[->]{rr}{f} \arrow[two heads]{rd}{c} \arrow[swap,two heads]{rdd}{z}&& Y\\
&C \arrow{ru}{f_c}\\
& Z \arrow{ruu}[swap]{f_z} \arrow[swap,dashed]{u}{u}
\end{tikzcd}
\]
\end{definition}

\paragraph{Images and coimages in $\Part$.} It is not difficult to check that every morphism $f:(X,P_X)\rightarrow (Y,P_Y)$ in $\Part$ is a mono. On the other hand, we have the following characterization of epis in $\Part$. 

\begin{proposition}\label{prop:epi iff bijective} A morphism $f:(X,P_X)\rightarrow (Y,P_Y)$ in $\mathbf{Part}$ is an epi if and only if $f:X\hookrightarrow Y$ is surjective (hence $f$ is bijective).
\end{proposition}

We remark that the bijectivity of $f:X\rightarrow Y$ does not imply that $f$ is an isomorphism in between $(X,P_X)$ and $(Y,P_Y)$.
\begin{proof}
For the forward direction, we prove the contrapositive. Suppose that $f$ is not surjective, i.e. there exists $y\in Y$ to which no $x\in X$ is mapped via $f$. Let $Z=\{z_1,z_2\}$. Let $k_1,k_2:Y\rightarrow Z$ be any pair of maps such that they differ only at $y$, i.e. $k_1=k_2$ on $Y\setminus \{y\}$ and $k_1(y)\neq k_2(y)$. Then both $k_1$ and $k_2$ are morphisms from $(Y,P_Y)$ to $(Z,\{z_1z_2\})$. Although $k_1\circ f=k_2\circ f$, we do not have that $k_1= k_2$. Hence, $f$ is not an epi.

We prove the backward direction. Consider any two morphisms $k_1,k_2:(Y,P_Y)\rightarrow (Z,P_Z)$ such that $k_1\circ =k_2\circ f$. Let $y\in Y$. Since $f$ is surjective, there exists $x\in X$ such that $f(x)=y$. Hence, $k_1\circ f(x)=k_2\circ f(x)$ implies $k_1(y)=k_2(y)$. Since $y$ was arbitrarily chosen in $Y$, we have that $k_1=k_2$.
\end{proof}

Consider any morphism $f:(X,P_X)\rightarrow (Y,P_Y)$ in $\Part$.  Let $f^{-1}(P_Y):=\left\{f^{-1}(B):B\in P_Y\right\}$ which is a partition of $X$. Notice that the identity map $\id_X$ on $X$ is a morphism from $(X,P_X)$ to $(X,f^{-1}(P_Y))$. Now we will see that, in the category $\Part$, the image is not isomorphic to the coimage in general. 

\begin{proposition}\label{prop:image and coimage} For any morphism $f:(X,P_X)\rightarrow (Y,P_Y)$ in $\Part$,

\begin{enumerate}[label=(\roman*),itemsep=-1ex]
    \item  $f$ itself is an image of $f$.\label{item:image}
    \item The morphism $\mathrm{id}_X:(X,P_X)\rightarrow (X,f^{-1}(P_Y))$, the identity set map on $X$, is a coimage of $f$. \label{item:coimage}
\end{enumerate}
\end{proposition}

\begin{proof} Item \ref{item:image} directly follows from the fact that $f$ is a mono. We prove Item \ref{item:coimage}. Since $\id_X:X\rightarrow X$ is bijective, by Proposition \ref{prop:epi iff bijective}, $\id_X$ is an epi from $(X,P_X)$ to $(X,f^{-1}(P_Y))$. Note also that $f$ is a morphism from $(X,f^{-1}(P_Y))$ to $(Y,P_Y)$ with the obvious identity $f\circ \mathrm{id}_X=f$. Assume that there exists a pair of an epi $z:(X,P_X)\twoheadrightarrow (Z,P_Z)$ and a morphism $f_z:(Z,P_Z)\rightarrow (Y,P_Y)$ such that $f=f_z\circ z$. Since $z$ is an epi, the set map $z:X\rightarrow Z$ must be bijective. The proof ends by observing that the inverse $z^{-1}:Z\rightarrow X$ is the \emph{unique} morphism $u$  from $(Z,P_Z)$ to $(X,f^{-1}(P_Y))$ such that $\mathrm{id}_X=u\circ z$.
\end{proof}

\begin{example} Consider $i:(\{x_1,x_2\},\{x_1|x_2\})\rightarrow (\{x_1,x_2,x_3\},\{x_1x_2x_3\})$ where $i$ is the canonical inclusion $\{x_1,x_2\}\hookrightarrow \{x_1,x_2,x_3\}$. The coimage of $i$ is the morphism $\id_{\{x_1,x_2\}}:(\{x_1,x_2\},\{x_1|x_2\})\rightarrow (\{x_1,x_2\},\{x_1x_2\})$.
\end{example}


\begin{figure}
    \centering
    \includegraphics[width=\textwidth]{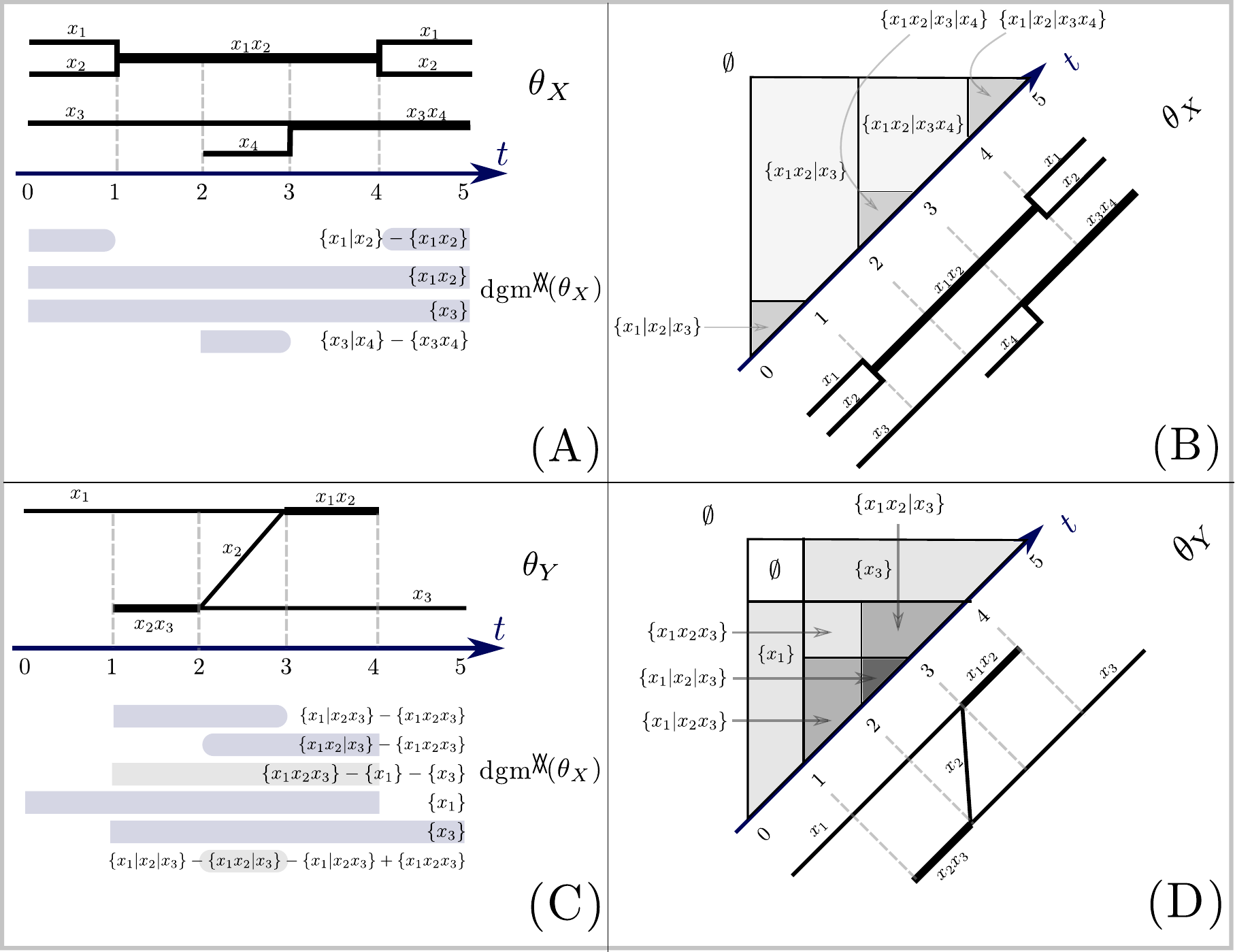}
    \caption{(A) A formigram $\theta_X$ over the set  $X:=\{x_1,x_2,x_3,x_4\}$ with support $[0,5]$ (i.e. union of the lifespans of all points) and the persistence clustergram $\dgm^{\protect\veewedge}(\theta_X)$. (B) An illustration of $\protect\veewedge \theta_X:\Int\rightarrow \subpart(X)$ for $\theta_X$ in (A). The persistence clustergram $\dgm^{\protect\veewedge}(\theta_X)$ in (A) is the M\"obius inversion of $\protect\veewedge \theta_X$. (C) A formigram $\theta_Y$ over $Y:=\{x_1,x_2,x_3\}$ and its persistence clustergram. Grey intervals carry silhouette values that are either negative or zero; $-1$ for $\{x_1x_2x_3\}-\{x_1\}-\{x_3\}$ and $0$ for $\{x_1|x_2|x_3\}-\{x_1x_2|x_3\}-\{x_1|x_2x_3\}+\{x_1x_2x_3\}$. (D) An illustration of $\protect\veewedge\theta_Y:\Int\rightarrow \subpart(X)$ for $\theta_Y$ in (C). The persistence clustergram $\dgm^{\protect\veewedge}(\theta_Y)$ in (C) is the M\"obius inversion of this $\protect\veewedge \theta_Y$. \label{fig:weak_vs_strong}}
\end{figure}

\section{Summarizing formigrams via M\"obius inversion}\label{sec:summarizing formigrams via mobius}

In this section we describe two novel summaries of formigrams: \emph{maximal group diagrams} (cf. Figure \ref{fig:maximal group diagram} (C)) and \emph{persistence clustergrams} (cf. Figure \ref{fig:entire picture} (G)); Definitions \ref{def:maximal group diagram} and \ref{def:persistence clustergram}. 
 
 \ok{In Section \ref{sec:mobius} we review the M\"obius inversion formula for a function from a locally finite poset. In Section \ref{sec:group-valued persistence diagrams and their silhouettes}, we introduce a notion of \emph{silhouette}\footnote{Silhouettes in this paper have no relation with the \emph{persistence silhouettes} in \cite{chazal2014stochastic}.} for a group-valued map from a poset; this notion allows us to have the summarization process that is illustrated as the arrow from (G) to (H) in Figure \ref{fig:entire picture}.} In Section \ref{sec:maximal group diagrams}, we show that the celebrated notion of \emph{maximal groups} by Buchin et al. \cite{buchin2013trajectory} is an instance of the generalized persistence diagrams considered in \cite{kim2018generalized,patel2018generalized}. Namely, we prove that maximal groups can be obtained by computing the M\"obius inversion of a suitably defined rank function. In Section \ref{sec:persistence clustergrams}, we introduce persistence clustergrams, which can be regarded as being ``dual" to maximal group diagrams. In Section \ref{sec:persistent cluster counting} and Section \ref{sec:stability of persistent cluster counting functor}, we consider the silhouette of a persistence clustergram and establish its stability.
 
\begin{framed}
\begin{convention}\label{conv:real intervals and zz intervals}
In Section \ref{sec:summarizing formigrams via mobius} we only consider formigrams with finite support whose critical points are contained in $\Z$ (cf. Definitions \ref{def:cosheaf conditions} and \ref{def:formigram}). We will not distinguish between the real interval $\langle a,b\rangle$ and $\langle a,b \rangle_{\ZZ}\in \Int(\ZZ)$ for any $a\leq b$ in $\Z$ (cf. Notation \ref{not:intervals in zz}). 
\end{convention}
\end{framed}

	\subsection{M\"obius inversion of a function on a poset}\label{sec:mobius} 
We review the M\"obius inversion formula for a function $f$ on a locally finite poset \cite{rota1964foundations}.  In a nutshell,  M\"obius  inversion is a discrete analogue of the derivative of a real-valued map from calculus. For example, given a map $f:\Z\rightarrow \R$ (where $\Z$ is equipped with the canonical order), its M\"obius inversion $f':\Z\rightarrow \R$ is defined as $f'(n)=f(n)-f(n-1)$ for $n\in \Z$, capturing the change of $f$ at the point $n$ relative to the value of $f$ at the point $n-1$. More generally, we can define M\"obius inversion of a function on any locally finite poset.
	\begin{definition}\label{def:locally finite} A poset $\Pb$ is said to be \emph{locally finite} if for all $p,q\in \Pb$ with $p\leq q$ the set $\{r\in \Pb: p\leq r\leq q\}$ is finite.
\end{definition}

Let $\Pb$ be a locally finite poset. The \emph{M\"obius function} $\mu_{\Pb}:\Pb\times\Pb\rightarrow \Z$ of $\Pb$ is defined recursively as

\begin{equation}\label{eq:induction}
    \mu_{\Pb}(p,q):=\begin{cases}1,& \mbox{$p=q$,}\\ -\sum_{p\leq r< q}\mu_{\Pb}(p,r),& \mbox{$p<q$,} \\ 0,& \mbox{otherwise.} \end{cases}
\end{equation}

\begin{example}\label{ex:mobius function on intervals of zz}Let $\Pb:=\Int(\ZZ)$ (Definition \ref{def:intervals in zz}). Then,
\[\mu_\Pb(I,J)=\begin{cases} 1,&I=J \mbox{ or  $I\subsetneq J$ with $\abs{J\setminus I}=2$,}\\-1,&I\subsetneq J\ \mbox{ with $\abs{J\setminus I}=1$,}\\0,&\mbox{otherwise.}
\end{cases}\]
\end{example}

We are interested in carrying out M\"obius inversion of abelian-group-valued maps defined on a given poset. Let $G$ be an abelian group. For any $m\in \Z_+$ and $x\in G$, we will interchangeably use $m\cdot x$ and $x\cdot m$ to denote $\underbrace{x+\cdots+x}_{\footnotesize{\mbox{$m$ terms}}}$.  For a negative $m\in \Z$, we will also use $m\cdot x$ or $x\cdot m$ to denote $-((-m)\cdot x)$.

\begin{theorem}[M\"obius inversion formula]\label{thm:Mobius} Let $\Pb$ be a locally finite poset and let $G$ be an abelian group. Suppose that for a function $f:\Pb \rightarrow G$,  there exists $p\in \Pb$ such that $f(q)=0$ unless $q\geq p$. 
Then, a function $g:\Pb\rightarrow G$ satisfies $f(q)=\sum_{r\leq q} g(r)$ for all $q\in \Pb$ if and only if $\displaystyle g(q)=\sum_{r\leq q} f(r)\cdot \mu_{\Pb}(r,q)$ for all $q\in \Pb$. \ok{We will often write $g=f'$.}
\end{theorem}

A difference between this theorem and \cite[Proposition 2 (p.344)]{rota1964foundations} is that the codomain of the functions $f$ and $g$ is an abelian group, whereas the codomain is a field in $\cite{rota1964foundations}$. Nevertheless, the verbatim proof in \cite{rota1964foundations} works in our setting. In \cite{patel2018generalized,kim2018generalized}, M\"obius inversion of abelian-group-valued functions is considered in order to obtain \emph{generalized persistence diagrams}. 

\begin{notation}[Neighborhood extension]\label{not:extended intervals of zz}
Let $I\in \Int(\ZZ)$. 
\begin{enumerate}[label=(\roman*),itemsep=-1ex]
    \item By $I^-$ we denote $I\cup\{(i,j)\}\in \Int(\ZZ)$ where $(i,j)\in \ZZ$ is the ``lower-left" neighborhood of $I$ in $\ZZ$. For example, for $I=(-1,1]_{\ZZ}$ (cf. Figure \ref{fig:intervals}), we have $I^-=[-1,1]_\ZZ$.
    \item  By $I^+$ we denote $I\cup\{(i',j')\}\in \Int(\ZZ)$ where $(i',j')\in \ZZ$ is the ``upper-right" neighborhood of $I$ in $\ZZ$. For example, for $I=(-1,1]_{\ZZ}$, we have $I^+=(-1,2)_\ZZ$. 
    
    \item By $I^{\pm}$ we denote $I^-\cup I^+$.
\end{enumerate}
\end{notation} 
\ok{The example below directly follows from Example \ref{ex:mobius function on intervals of zz}, Theorem \ref{thm:Mobius}, and the following fact: 
    \emph{For the opposite poset $\Pb^{\mathrm{op}}$ of $\Pb$, the M\"obius function $\mu_{\Pb^{\mathrm{op}}}$ is obtained by $\mu_{\Pb^{\mathrm{op}}}(p,q)=\mu_{\Pb}(q,p)$ for all $p,q\in \Pb$}  \cite[p.345]{rota1964foundations}.

}
\begin{example}\label{ex:mobius inversion on zz} Assume that for a function $f:\Int(\ZZ) \rightarrow G$,  there exists $J_0\in \Int(\ZZ)$ such that $f(I)=0$ unless $J_0\supset I$.  Let $g:\Int(\ZZ)\rightarrow G$ be such that $f(I)=\sum_{J\supset I}g(J)$ for all $I\in \Int(\ZZ)$. Then, 
\[f'(I)=g(I)=f(I)-f(I^-)-f(I^+)+f(I^{\pm})\hspace{2mm}\mbox{for all $I\in \Int(\ZZ)$.}\]
\ok{Note that (1) this formula is reminiscent of the original definition of \emph{persistence diagrams} \cite{cohen2007stability,ELZ02} and (2) this formula was utilized in \cite{kim2018generalized} to define generalized persistence diagrams for functors indexed by $\ZZ$.}        
\end{example}

\subsection{Silhouettes and formal sums of clusters }\label{sec:group-valued persistence diagrams and their silhouettes}

\ok{We introduce the notion of the \emph{silhouette} of a free-abelian-group valued map on a poset.} Let $A$ be a nonempty finite set. By $\Z A$ we denote the free abelian group generated by $A$. This means that any element in $\Z A$ is a formal sum of elements from $A$ with integer coefficients. 

\begin{remark}\label{rem:group homomorphism}
The map $\abs{-}:\Z A\rightarrow \Z$ that sends $\sum_{i=1}^\ell m_{i}a_i$ ($m_i\in \Z$, $a_i\in A$) to $\sum_{i=1}^\ell m_i$ is a group homomorphism.
\end{remark}

In the rest of this section let $\Pb$ be a locally finite poset.

\begin{definition}[Support and silhouette]\label{def:silhouette} Let $G$ be an abelian group. Consider a map $f:\Pb\rightarrow G$. 

\begin{enumerate}[label=(\roman*),itemsep=-1ex]
    \item The \emph{support} of $f$ is defined as $\supp(f):=\{p\in\Pb: f(p)\neq 0\in G\}$.
    \item Assume that $G=\Z A$. The \emph{silhouette} $\abs{f}:\Pb\rightarrow \Z$ of $f$ is defined as follows. For $p\in\Pb$ if $f(p)=\sum_{i=1}^\ell m_{i}a_i$ ($m_i\in \Z$, $a_i\in A$), then $\abs{f}(p):=\sum_{i=1}^\ell m_i$.\label{item:silhouette2}
\end{enumerate}
\end{definition}

Given an $f:\Pb\rightarrow \Z A$, $\supp(\abs{f})$ is contained in $\supp(f)$ but the converse is not always true. For example, if $f(p)=a_1-a_2$ with $a_1\neq a_2$, then $p\in \supp(f)$ but $p\notin\supp(\abs{f})$.

\begin{notation}\label{not:set notation}We will often represent $f:\Pb\rightarrow G$ by the set $\{(p,f(p)):p\in \supp(f)\}$. In particular, when $G=\Z$ and $\im(f)\subset \{0,1\}$, $f$ will be identified with $\supp(f)$. 

\end{notation}

The silhouette, as an operation, commutes with M\"obius inversion over any locally finite poset: 

\begin{proposition}\label{prop:silhouette commutes with mobius} Given a map $f:\Pb\rightarrow \Z A$, let $f'$ be the M\"obius inversion of $f$, i.e. for all $p\in \Pb$, $f'(p)=\sum_{q\leq p}f(q)\cdot \mu_\Pb(q,p)$. Then, the M\"obius inversion of $\abs{f}:\Pb\rightarrow \Z$ equals $\abs{f'}$, i.e. $\abs{f}'=\abs{f'}$. 
\end{proposition}

\begin{proof} Let $p\in \Pb$. We have:
\begin{align*}
    \abs{f'}(p)&=\abs{\sum_{q\leq p}f(q)\cdot \mu_\Pb(q,p)}&\mbox{by Definition \ref{def:silhouette} \ref{item:silhouette2}}\\
    &=\sum_{q\leq p}\abs{f}(q)\cdot \mu_\Pb(q,p)&\mbox{by Remark \ref{rem:group homomorphism}}\\&=\abs{f}'(p).
\end{align*}
\end{proof}

\ok{In the rest of the paper, the basis $A$ of the free abelian group $\Z A$ will be $\pow(X)$, the power set of some finite set $X$.}

\begin{example} 
\ok{Let $X:=\{x_1,x_2,x_3\}$ and assume that a given map $f:\Int(\ZZ)\rightarrow \fpow(X)$ satisfies
\begin{align}\label{eq:mobius}
   \begin{cases}
    f((1,2)_\ZZ)&=\{x_1\}+\{x_2\}+\{x_3\},\\ f([1,2)_\ZZ)&=\{x_1\}+\{x_2,x_3\},\\ f((1,2]_\ZZ)&=\{x_1,x_2\}+\{x_3\},\\ f([1,2]_\ZZ)&=\{x_1,x_2,x_3\}.\\
   \end{cases}
\end{align} Then, from Example \ref{ex:mobius inversion on zz}, we have \[f'((1,2)_{\ZZ})=\{x_1\}+\{x_2\}+\{x_3\}-(\{x_1\}+\{x_2,x_3\})-(\{x_1,x_2\}+\{x_3\})+\{x_1,x_2,x_3\},\] 
and thus $\abs{f'}((1,2)_{\ZZ})=1+1+1-(1+1)-(1+1)+1=0$. Next we consider $\abs{f}:\Int(\ZZ)\rightarrow \Z$. By equations in (\ref{eq:mobius})
\[\abs{f}((1,2)_\ZZ)=3, \hspace{5mm}\abs{f}([1,2)_\ZZ)=\abs{f}((1,2]_\ZZ)=2, \hspace{5mm}\abs{f}([1,2]_{\ZZ})=1.\]
Again from  Example \ref{ex:mobius inversion on zz}, we have $\abs{f}'((1,2)_{\ZZ})=3-2-2+1=0$ which equals $\abs{f'}((1,2)_{\ZZ})$.}
\end{example}

 \ok{We} define a natural inclusion $\iota:\subpart(X)\hookrightarrow \fpow(X)$ as follows. For any $P=\{B_i\}_{i=1}^\ell \in \subpart(X)$, let $\iota(P):=\sum_{i=1}^\ell 1\cdot B_i$. Let us also define a natural left-inverse $\Pi:\fpow(X)\rightarrow \subpart(X)$ of $\iota$ as follows. 
\begin{definition}[Left-inverse to $\iota:\subpart(X)\hookrightarrow \fpow(X)$]\label{def:left inverse} 
Let $g\in \fpow(X)$ be any nonzero element that is given as $\sum_{i=1}^\ell m_{i}B_i$ ($m_i(\neq 0)\in \Z$) where $B_i\neq B_j$ for $i\neq j$. Then, $\Pi(g)$ is defined as the finest common coarsening $\bigvee\{\{B_i\}\}_{i=1}^n$. The trivial element $0\in \fpow(X)$ is sent to $\emptyset$ via $\Pi$.
\end{definition}
 For example, if $X:=\{x_1,x_2,x_3\}$, $\Pi$ sends $\{x_1,x_2\}+\{x_2,x_3\}\in \fpow(X)$ to $\{\{x_1,x_2\}\}\vee\{\{x_2,x_3\}\}=\{\{x_1,x_2,x_3\}\}\in \subpart(X)$. 

\subsection{Maximal group diagrams}\label{sec:maximal group diagrams}

Buchin et al. defined the concept of \emph{maximal groups}\footnote{In this section \emph{groups} do not stand for groups in abstract algebra. See Remark \ref{rem:maximal groups}.}
 for a set $\mathcal{X}$ of trajectories in Euclidean space \cite{buchin2013trajectory} (cf. Remark \ref{rem:maximal groups}). We observe that maximal groups are totally determined by the formigram induced by the $\delta$-connectivity dynamic graph of the set of trajectories  (cf. Example \ref{ex:special DGs}), which enables us to prove that maximal groups are obtained by M\"obius inversion (cf. Theorem \ref{thm:maximal groups via Mobius inversion}). A corollary is that the collection of all maximal groups of $\mathcal{X}$ contains the same information as the formigram $\theta$ induced by $\mathcal{X}$ (cf. Remark \ref{rem:maximal groups}). 
 
\begin{definition}\label{def:group}
Let $\theta_X$ be a formigram over $X$. A subset $G\subset X$ is called a \emph{group} over an interval $I_G\subset \R$ if for all $t\in I$, there exists $B_t\in \theta_X(t)$ such that $G\subset B_t$. A group $H$ \emph{covers} group $G$ if $G\subset H$ and $I_G\subset I_H$. 
\end{definition}

\begin{definition}\label{def:maximal groups} Let $\theta_X$ be a formigram and let $G\subset X$ be a group over $I_G\subset \R$. Then, $G$ is said to be a \emph{maximal group} (over $I_G$) if there is no group $H\subset X$ that covers $G$.
\end{definition}

 Let us recall the natural injection $\iota:\subpart(X)\hookrightarrow \fpow(X)$ which sends $\{B_i\}_{i=1}^\ell \in \subpart(X)$ to $\sum_{i=1}^\ell 1\cdot B_i$, a formal sum of blocks of $P$. 
 In what follows, $\{B_i\}_{i=1}^\ell \in \subpart(X)$ will be identified with $\sum_{i=1}^\ell 1\cdot B_i$.

\begin{definition}\label{def:wedge rank function} Let $\theta_X$ be a formigram. The \emph{$\bigwedge$-rank function}  $\bigwedge\theta_X:\Int(\ZZ)\rightarrow \subpart(X)$  is defined as $I\mapsto \bigwedge_I\theta_X:=\bigwedge\{\theta_X(i,j):(i,j)\in I\}$. \ok{By Convention \ref{conv:real intervals and zz intervals}, when $I=\langle a,b\rangle_\ZZ$, we have $\bigwedge_I\theta_X=\bigwedge_{t\in\langle a,b \rangle} \theta_X(t)$.}
\end{definition}

 For $I\in \Int(\ZZ)$, note that $\bigwedge_I \theta_X$ is a collection of groups over $I$. In particular, $G\in \bigwedge_I \theta_X$ is a maximal group if there is no group $G'\subset X$ covering $G$ over either $I^-$ or $I^+$ (cf.  Notation \ref{not:extended intervals of zz}).  Let $\dgm^{\bigwedge}(\theta_X)$ be the M\"obius inversion of $\bigwedge \theta_X$ over $\Intzzop$ (cf. Example \ref{ex:mobius inversion on zz}).  This means that for all $I\in \Int(\ZZ)$

\begin{equation}\label{eq:wedge Mobius inversion} \dgm^{\bigwedge}(\theta_X)(I):=\bigwedge_I\theta_X-\bigwedge_{I^+}\theta_X-\bigwedge_{I^-}\theta_X+\bigwedge_{I^{\pm}}\theta_X,
\end{equation}
where these sums and subtractions should be interpreted as operations in $\fpow(X)$. Let us recall the map $\Pi: \fpow(X) \rightarrow \subpart(X)$ from Definition \ref{def:left inverse}. 

\begin{definition}\label{def:maximal group diagram} The \emph{maximal group diagram} of $\theta_X$ is defined as $\Pi\circ \dgm^{\bigwedge}(\theta_X):\Int(\ZZ)\rightarrow \subpart(X)$. \ok{We will simply write $\Pi \dgm^{\bigwedge}(\theta_X)$ for $\Pi\circ \dgm^{\bigwedge}(\theta_X)$.}
\end{definition}

\ok{We adapt Notation \ref{not:set notation} as follows. For $I\in \Int(\ZZ)$, if $\pidgm(\theta_X)(I)$ contains a block $G\subset X$, then we write $(I,G)\in \pidgm(\theta_X)$.} The following theorem says that $\pidgm(\theta_X)$ encodes all the maximal groups of $\theta_X$. See Figures \ref{fig:maximal group diagrams and persistence clustergrams} (A) and (B) for examples; \ok{any pair $(I,G)\in \pidgm(\theta_X)$ is depicted as the interval $I$ annotated by $G$.} 

\begin{figure}
    \centering
    \includegraphics[width=\textwidth]{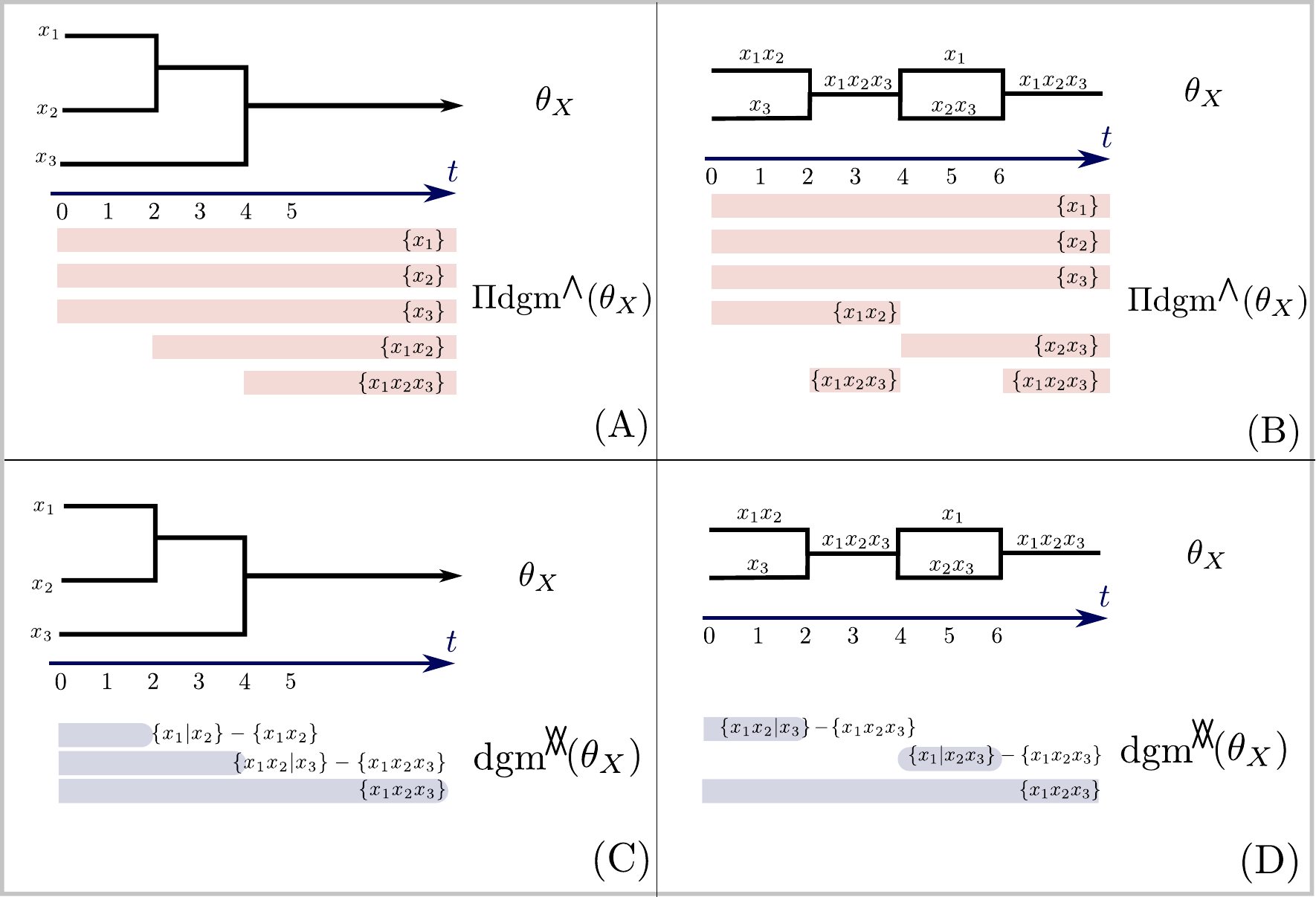}
    \caption{(A) and (B) depict maximal group diagrams (cf. Definition \ref{def:maximal group diagram}) while (C) and (D) depict persistence clustergrams (cf. Definition \ref{def:persistence clustergram}).} 
    \label{fig:maximal group diagrams and persistence clustergrams}
\end{figure}

\begin{figure}
    \centering
    \includegraphics[width=0.4\textwidth]{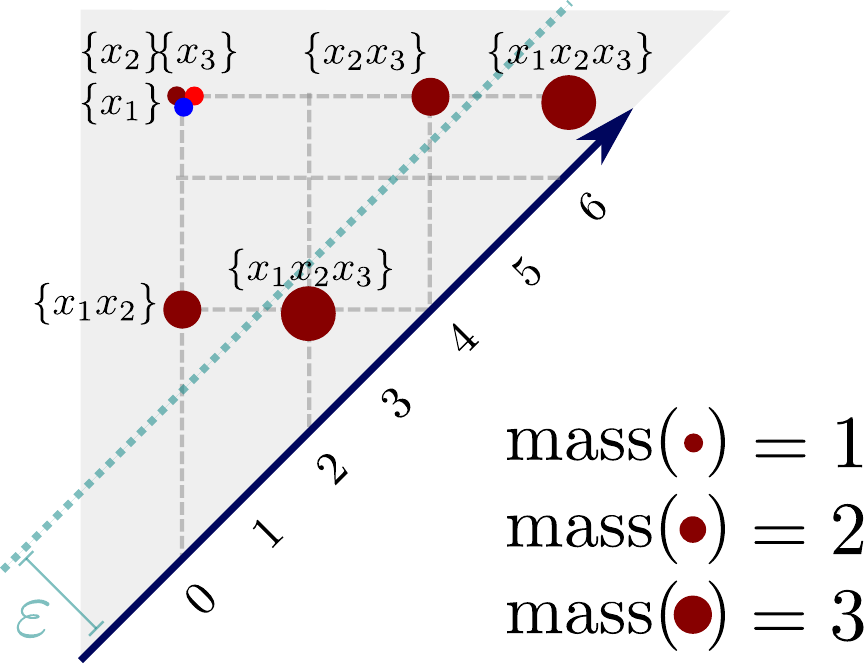}
    \caption{The maximal group diagram in Figure \ref{fig:maximal group diagrams and persistence clustergrams} (B) is represented as an annotated standard persistence diagram \cite{cohen2007stability} ($[a,b]\in \Int$ is identified with $(a,b)\in \R^\mathrm{op}\times \R$). In the 3-tuple $(m,\eps,\delta)$ of parameters for maximal groups \ok{(cf. Remark \ref{rem:maximal groups})}, $m$ corresponds to \ok{the minimal size of maximal groups (depicted as the ``mass" of points above)}, and $\eps$ corresponds to the minimal duration of maximal groups. By the monotonicity of maximal groups (cf. Remark \ref{rem:max diagram is complete}), the masses of points near the diagonal line $y=x$ tend to be relatively large.}
    \label{fig:three parameters for maximal group diagrams}
\end{figure}

\begin{theorem}\label{thm:maximal groups via Mobius inversion} Let $\theta_X$ be a formigram. A subset $G\subset X$ is a maximal group of $\theta_X$ on an interval $I\subset\R$  if and only if $(I,G)\in \pidgm(\theta_X)$. 
\end{theorem}

\begin{proof} Fix \emph{any} $G\in \bigwedge_I \theta_X$. Since $\bigwedge_{I^+} \theta_X \leq \bigwedge_I \theta_X$, either (a) $G\in \bigwedge_{I^+} \theta_X$ or (b) there exists the  subcollection $\{H_j\in \bigwedge_{I^+} \theta_X: H_j \subset G\}$ of $\bigwedge_{I^+} \theta_X$ which does not include $G$. Similarly, since $\bigwedge_{I^-} \theta_X \leq \bigwedge_I \theta_X$, there are the two analogous cases: (a') $G\in \bigwedge_{I^-} \theta_X$, or (b') there exists the subcollection $\{K_j\in \bigwedge_I\theta_X: K_j\subset G\}$ which does not include $G$. In combination, there are the four possible cases: (aa'), (ab'), (ba'), and (bb'). Among these, observe that $G$ appears in the subpartition $\pidgm(\theta_X)(I)$ only in  case (bb'), i.e. when $G$ is a maximal group over $I$. In the other three cases, $\pidgm(\theta_X)(I)$ does not contain $G$ nor any of its subsets. The claim follows.
\end{proof}

\begin{remark}\label{rem:maximal groups}
\ok{Let $X$ be a set of trajectories in  Euclidean space.}
 In \cite{buchin2013trajectory}, a \emph{group} of  $X$ is defined according to three different parameters. For any positive $m\in \Z$ and $\eps,\delta \in \R_+$,  a set $G\subset X$ is said to form an \emph{$(m,\eps,\delta)$-group} during a time interval $I$ if and only if (1) $G$ contains at least $m$ points, (2) the length of $I$ is not less than $\eps$ and (3) for any two points $x,x'\in G$ and any time $t\in I$, there is a chain $x=x_0,x_1,\ldots,x_n=x'$ of points in $X$ such that any two consecutive ones are at \hbox{distance $\leq\delta$} at time $t$. Note that (1) \emph{groups are totally determined by the formigram $\theta_X$ induced by the $\delta$-connectivity dynamic graph on $X$}, and that (2) for any positive $m\in \Z$ and any $\eps\in\R_+$ the collection of $(m,\eps,\delta)$-groups is the subcollection of $(1,0,\delta)$-groups. Therefore, we restrict our attention to maximal $(1,0,\delta)$-groups which are exactly the maximal groups of $\theta_X$ in Definition \ref{def:maximal groups}.
\end{remark}

\begin{example}\label{ex:maximal group}
\ok{Assume that the formigram $\theta_X$ in Figure \ref{fig:maximal group diagrams and persistence clustergrams} (B) is obtained from some three trajectories $x_1,x_2$ and $x_3$ in Euclidean space as described in Remark \ref{rem:maximal groups} (cf. Example \ref{ex:special DGs}  and Definition \ref{def:from DG to formi}). Figure \ref{fig:three parameters for maximal group diagrams} is another representation of the maximal group diagram of $\theta_X$.}
\end{example}

Consider any $\R$-indexed module $M$ (cf. Section \ref{sec:intervaldecomp}). Then, for any $t\in \R$, the dimension of $M(t)$ equals the total multiplicity of those intervals $I$ in the barcode of $M$ which contain $t$. We will establish an analogous result for $\theta_X:\R\rightarrow \subpart(X)$. Let us observe that for any nonempty $B\in \pow(X)$, we have that $\{B\}\in \subpart(X)$.

\begin{remark}[Monotonicity of maximal groups and completeness]\label{rem:max diagram is complete}  By Theorem \ref{thm:maximal groups via Mobius inversion} and the definition of maximal groups, $\pidgm(\theta_X)$ satisfies a \emph{monotonicity property}: if $B\subset B'$ with $(I,B),(I',B')\in \pidgm(\theta_X)$, then $I\supset I'$. This in turn implies that, for $t\in \R$, $\theta_X(t)$ is equal to 
\[\bigvee\{\{B\}\in \pow(X)\setminus\{\emptyset\}:\mbox{there exists $I\ni t$ such that $(I,\{B\})\in \pidgm(\theta_X)$}\}.\]Therefore, we can recover a formigram $\theta_X$ from its maximal group diagram $\pidgm(\theta_X)$.
\end{remark}

\begin{example}Consider the formigram depicted in Figure \ref{fig:maximal group diagrams and persistence clustergrams} (B). Note that
$\theta_X(5)=\{x_1|x_2x_3\}$ and the blocks corresponding to the four intervals containing $t=5$ are $\{x_1\},\{x_2\},\{x_3\}$ and $\{x_2x_3\}$. We have that  
\[\theta_X(5)=\{x_1|x_2x_3\}=\bigvee\{\{x_1\},\{x_2\},\{x_3\},\{x_2x_3\}\}.\]
\end{example}

Let us characterize the silhouette of the maximal group diagram of a dendrogram (cf. Remark \ref{rem:about the definition of formigrams}).

\begin{example}\label{ex:dendrogram} Let $\theta_X$ be a dendrogram, and let $G\subset X$. If there exists $t\in \R$ such that $G\in \theta_X(t)$, then define $b(G):=\min\{t\in \R: G\subset \theta_X(t)\}$. Then, the silhouette of 
$\pidgm(\theta_X)$ amounts to the set \[\{[b(G),\infty): \mbox{there exists $t\in \R$ such that $G\in \theta_X(t)$}\}\  \mbox{(cf. Notation \ref{not:set notation}}).\]
 For example, consider the dendrogram $\theta_X$ over $X:=\{x_1,x_2,x_3\}$ that is depicted in Figure \ref{fig:maximal group diagrams and persistence clustergrams} (A). Since 
 \[\bigcup_{t\in \R}\theta_X(t)=\{\{x_1\}, \{x_2\}, \{x_3\}, \{x_1,x_2\}, \{x_1,x_2,x_3\}\},\]
the silhouette of $\pidgm(\theta_X)$ consists of the five half-infinite intervals whose left endpoints are $b(\{x_1\})$, $b(\{x_2\})$, $b(\{x_3\})$, $b(\{x_1,x_2\})$, and $b(\{x_1,x_2,x_3\})$.
\end{example}

Remark \ref{rem:instability of maximal groups} below motivates us to consider a different summary of a formigram, leading to the next section. 

\begin{remark}\label{rem:instability of maximal groups} 
\begin{enumerate}[label=(\roman*)]
    \item 
The maximal group diagram can drastically change under perturbations of an input formigram \emph{with respect to $\dintf$} (for example, see Figure \ref{fig:instability of maximal groups}). This is not unexpected since the distance $\dintf$ utilizes the join operation $\bigvee$ on subpartitions to quantify the difference between formigrams whereas the maximal group diagram is extracted from the $\bigwedge$-rank function (Definition \ref{def:wedge rank function}). 

    \item The silhouette of the maximal group diagram almost always looks completely different from the barcode of $\theta_X$ (Definition \ref{def:underlying} \ref{item:zigzag barcode of formigram}); \ok{for example, observe that the silhouettes of $\Pi\dgm^\wedge(\theta_X)$ in Figure \ref{fig:maximal group diagrams and persistence clustergrams} (A),(B) are completely different from the silhouettes of $\dgm^{\veewedge}(\theta_X)$ in Figure \ref{fig:maximal group diagrams and persistence clustergrams} (C),(D). The latter silhouettes are actually the zigzag barcodes of the respective formigrams, as it will turn out in the next section (cf. Theorem \ref{thm:persistence clustergram of a saturated formigram}).}
\end{enumerate}
\end{remark}

\begin{figure}
    \centering
    \includegraphics[width=0.9\textwidth]{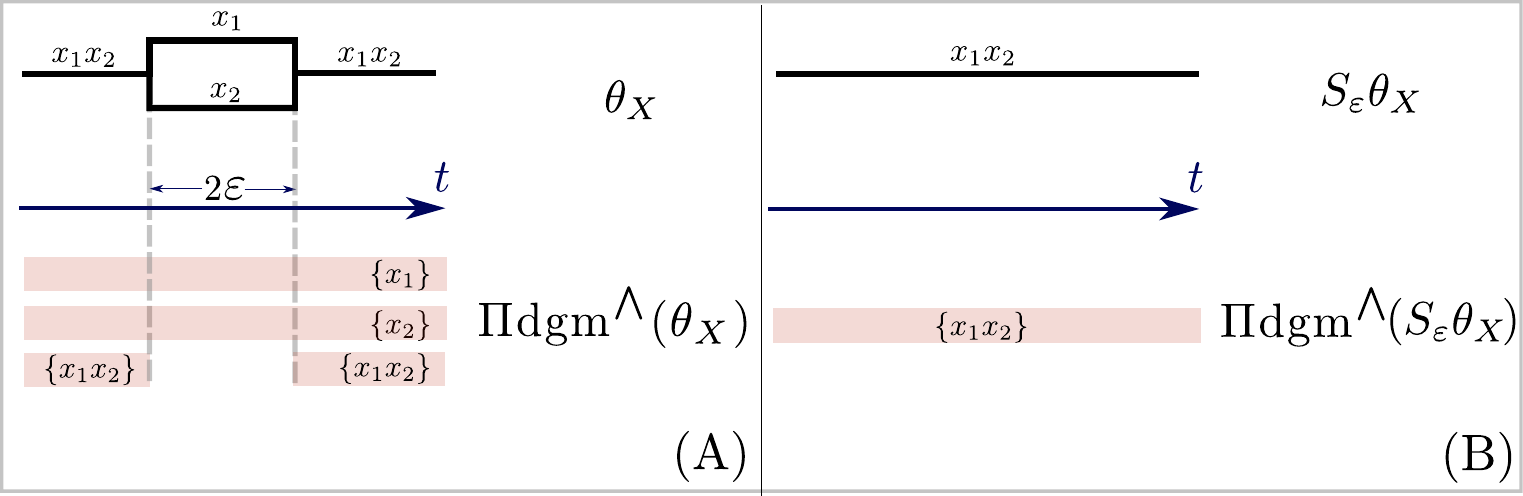}
    \caption{(A) The formigram $\theta_X$ over $X:=\{x_1,x_2\}$ such that $\theta_X(t)=\{x_1|x_2\}$ for $t\in (0,2\eps)$ and $\theta_X(t)=\{x_1x_2\}$ for $t\in \R\setminus(0,\eps)$. (B) The $\eps$-smoothing of $\theta_X$ (Definition \ref{def:cosheaf smoothing}). The bottleneck distance between the silhouettes of $\pidgm(\theta_X)$ and $\pidgm(S_\eps\theta_X)$ is $\infty$, whereas $\dintf(\theta_X,S_\eps\theta_X)=\eps$.} 
    \label{fig:instability of maximal groups}
\end{figure}

\subsection{Persistence clustergrams}\label{sec:persistence clustergrams}

In this section we define the persistence clustergram of a formigram. When a given formigram $\theta_X$ is saturated, its persistence clustergram can be regarded as an ``annotated" zigzag barcode of $\theta_X$
(cf. Figure \ref{fig:entire picture} (C)). We begin by defining the $\veewedge$-rank function of a formigram ($\veewedge$ is a fusion of $\bigwedge$ and $\bigvee$, as will be clear below). The M\"obius inversion of the $\veewedge$-rank function of $\theta_X$ will be the persistence clustergram of $\theta_X$.

Let $\theta_X$ be a formigram. For $I\in \Int$, we define $\veewedge_I \theta_X$ as the partition of the underlying set $X'(\subset X)$ of $\bigwedge_I\theta_X$ that is obtained by restricting the equivalence relation induced by $\bigvee_I\theta_X$ to $X'$. In other words,
\begin{equation}\label{eq:coimage}
    \veewedge_I \theta_X=\left\{B\cap X':B\in \bigvee_I\theta_X \ \mbox{and } B\cap X'\neq \emptyset\right\}.
\end{equation}

Note that when $\theta_X$ is saturated, $\veewedge_I\theta_X=\bigvee_I\theta_X$ for any interval $I\subset \supp(\theta_X)$.

\begin{definition}\label{def:veewedge rank function}Let $\theta_X$ be a formigram. The \emph{$\veewedge$-rank function} $\veewedge\theta_X:\Int(\ZZ)\rightarrow \subpart(X)$ of $\theta_X$ is defined by $I\mapsto \veewedge_I\theta_X$. 
\end{definition}

\begin{remark}
    \ok{In the category $\Part$ (cf. Definition \ref{def:partition category}), the pair $(\bigcup\left(\veewedge_I \theta_X\right),\veewedge_I \theta_X)$ is the \emph{coimage} of the morphism $(\bigcup\left(\bigwedge_I\theta_X\right),\bigwedge_I\theta_X)\rightarrow (\bigcup(\bigvee_I \theta_X) ,\bigvee_I \theta_X)$ which is given by the inclusion $\bigcup\left(\bigwedge_I\theta_X\right)\hookrightarrow \bigcup(\bigvee_I \theta_X)$ (cf. Proposition \ref{prop:image and coimage}). In this respect, the $\veewedge$-rank function is a rendition of the \emph{generalized rank invariant} from \cite{kim2018generalized}. Table \ref{tab:comparison} clarifies the analogy. } 
\end{remark}

\begin{table}
\begin{center}
\begin{tabular}{|c|p{4.5cm}|c|c|c|c|}
\hline \rowcolor[HTML]{C0C0C0} 
& \cellcolor[HTML]{C0C0C0} Nature of  \mbox{$F:\Pb\rightarrow \C$}  &    $\Pb$ &     $\C$ &      $\mathrm{rk}(F)(I)$   &   Representation of $\mathrm{rk}(F)(I)$                      \\ 
\hline \hline
(\rNum{1})&Sublevelset persistence \cite{cohen2007stability}               & $\R$           & \multirow{2}{*}{$\vect$}         & \multirow{6}{*}{$\coim\left(\varprojlim F|_I \rightarrow \varinjlim F|_I\right)$} & \multirow{2}{*}{An integer (dimension)}
\\ \cline{1-3}
(\rNum{2})&Levelset persistence \cite{zigzag} & $\Int$         &         &  & 
\\ \cline{1-4} \cline{6-6}
(\rNum{3})&Merge tree \cite{morozov2013interleaving}                           & $\R$         &  \multirow{2}{*}{$\sets$}         &  & \multirow{2}{*}{An integer (cardinality)}
\\  \cline{1-3}
(\rNum{4})&Reeb graph \cite{de2016categorified}                     & $\Int$           &          &  & 
\\   \cline{1-4} \cline{6-6}
(\rNum{5})&Unlabeled dendrogram \cite{clustum}      & $\R$         &  \multirow{2}{*}{$\Part$}         & &  \multirow{2}{*}{An integer partition}\\  \cline{1-3}
(\rNum{6})&Unlabeled formigram                        & $\Int$           &        & & 
\\ \hline
\end{tabular}

\caption{$\protect\mathrm{rk}(F)(I)$ denotes the \emph{rank} of the functor $F$ over an interval $I\subset\R$ up to isomorphism {\cite[Definition 3.5]{kim2018generalized}}. In each of rows (\rNum{1})-(\rNum{4}), the \emph{coimage} of $\protect\varprojlim F|_I \rightarrow \protect\varinjlim F|_I$ can be replaced by the \emph{image} of $\protect\varprojlim F|_I \rightarrow \protect\varinjlim F|_I$, since they are isomorphic. However, it is not the case in row (\rNum{5}) nor in row (\rNum{6}); see Proposition \ref{prop:image and coimage}.}\label{tab:comparison}
\end{center}
\end{table}

\begin{example}Consider the formigram depicted in Figure \ref{fig:weak_vs_strong} (A). For $I:=[0,5]$, we have $\bigwedge_I \theta_X=\{x_1|x_2|x_3\}$, $\bigvee_I \theta_X=\{x_1x_2|x_3x_4\}$, and $\veewedge_I \theta_X=\{x_1x_2|x_3\}$. Note that whereas the $\bigwedge$-rank function classifies $x_1$ and $x_3$ into different clusters over $I$, the $\veewedge$-rank function puts $x_1$ and $x_3$ into the same cluster over $I$. The $\veewedge$-rank function $\veewedge \theta_X$ is completely depicted in Figure  \ref{fig:weak_vs_strong} (B).
\end{example}

\begin{definition}\label{def:persistence clustergram}
Let $\theta_X$ be a formigram. The \emph{persistence clustergram} $\dgm^{\veewedge}(\theta_X):\Int(\ZZ)\rightarrow \fpow(X)$ is defined as the M\"obius inversion of $\veewedge\theta_X$ over the poset $\Intzzop$. Namely, for all $I\in \Int(\ZZ)$,
\begin{equation}\label{eq:verbose Mobius inversion} \dgm^{\veewedge}(\theta_X)(I):=\veewedge_I\theta_X-\veewedge_{I^+}\theta_X-\veewedge_{I^-}\theta_X+\veewedge_{I^{\pm}}\theta_X.
\end{equation}
\end{definition}

By the following remark, $\dgm^{\veewedge}$ is a complete invariant of formigrams.

\begin{remark}\label{rem:persistence clustergram is complete}
By Theorem \ref{thm:Mobius}, for any $I\in \Int(\ZZ)$, we have that $\veewedge_I\theta_X=\sum_{J\supset I}\dgm^{\veewedge}(\theta_X)(J)$. In particular, for any $t\in \R$, $\theta_X(t)=\veewedge_{[t,t]}\theta_X=\sum_{J\ni t}\dgm^{\veewedge}(\theta_X)(J)$. Thus $\theta_X$ can be recovered from $\dgm^{\veewedge}(\theta_X)$.
\end{remark}

Whereas the $\Hrm_0$ barcode or the underlying merge tree is not a complete invariant of a dendrogram, the persistence clustergram of a dendrogram is a complete invariant.

\begin{example}[Persistence clustergram of dendrograms] Consider the dendrogram $\theta_X$ over $X:=\{x_1,x_2,x_3\}$ which is depicted as in Figure \ref{fig:maximal group diagrams and persistence clustergrams} (C). Then, $\dgm^{\veewedge}(\theta_X)$ consists of the three elements (cf. Notation \ref{not:set notation}) 
\begin{equation}\label{eq:pairs}
    ([0,2),\{x_1|x_2\}-\{x_1x_2\}),\  ([0,4),\{x_1x_2|x_3\}-\{x_1x_2x_3\}),\  ([0,\infty),\{x_1x_2x_3\}),
\end{equation}
as illustrated in the figure. Invoking Remark \ref{rem:persistence clustergram is complete}, let us compute $\theta_X(1)=\sum_{I\ni 1}\dgm^{\veewedge}(\theta_X)(I)$. Since the silhouette of $\dgm^{\veewedge}(\theta_X)$ is $\{[0,2),[0,4),[0,\infty)\}$, and $1\in \R$ belongs to all of $[0,2),[0,4)$ and $[0,\infty)$ we have:
\begin{align*}
    \theta_X(1)&=\dgm^{\veewedge}(\theta_X)[0,2)+\dgm^{\veewedge}(\theta_X)[0,4)+\dgm^{\veewedge}(\theta_X)[0,\infty)\\&=\left(\{x_1|x_2\}-\{x_1x_2\}\right) +\left(\{x_1x_2|x_3\}-\{x_1x_2x_3\}\right)+\{x_1x_2x_3\}\\&=\{x_1|x_2|x_3\}.
\end{align*}
Also, the silhouette $\{[0,2),[0,4),[0,\infty)\}$ of $\dgm^{\veewedge}(\theta_X)$ coincides with the zigzag barcode of $\theta_X$. This is not merely a coincidence, as we will see in the following theorem.

\end{example}

\begin{theorem}\label{thm:persistence clustergram of a saturated formigram}If $\theta_X$ is a \emph{saturated} formigram (Definition \ref{def:special cases2}), then the barcode of $\theta_X$ is equal to the silhouette $\abs{\dgm^{\veewedge}(\theta_X)}$ of $\dgm^{\veewedge}(\theta_X)$ (cf. Figures \ref{fig:entire picture} (F),(G),(H) and \ref{fig:maximal group diagrams and persistence clustergrams} (C),(D)). 
\end{theorem}

We prove this theorem in Appendix \ref{sec:persistence clustergram of a saturated formigram} by harnessing results from \cite{kim2018generalized}.

\begin{corollary}
If $\theta_X$ is a \emph{saturated} formigram, the silhouette $\abs{\dgm^{\veewedge}(\theta_X)}:\Intzz\rightarrow \Z$ is nonnegative.
\end{corollary}

\begin{example}
When a formigram is not saturated, the silhouette of its persistence clustergram is not necessarily the same as the barcode of $\theta_X$. See Figure \ref{fig:weak_vs_strong} (C) for an illustrative example.
\end{example}

Since $\dgm^{\veewedge}$ is a complete invariant of formigrams, the distance $\dintf$ between formigrams can also be viewed as a distance between their persistence clustergrams. Since computing $\dintf$ is NP-hard, it is desirable to find many computable lower bounds for $\dintf$. One such lower bound is the bottleneck distance between the barcodes of formigrams (cf. Corollary \ref{cor:df-db stability}). One additional lower bound based on persistence clustergrams is proposed in the next section.

\subsection{Persistent cluster counting functor}\label{sec:persistent cluster counting}

We define the \emph{persistent cluster counting functor} of a formigram $\theta_X$ as the silhouette of $\veewedge\theta_X$.  We will see later that persistent cluster counting functors of formigrams can be used for discriminating formigrams which cannot be discriminated by the underlying Reeb graphs of formigrams. Throughout this section, $X$ will denote nonempty finite sets. \ok{Convention \ref{conv:real intervals and zz intervals} still applies in this section.} 

For $P\in \subpart(X)$, $\abs{P}$ denotes the number of blocks in $P$. By viewing $P$ as the formal sum $\sum_{B\in P}1\cdot B \in \fpow(X)$, this notation is consistent with the notation $\abs{-}$ which is defined in Remark \ref{rem:group homomorphism}.

\begin{definition}
\label{def:persistent cluster counting functor} The \emph{persistent cluster counting functor} of a formigram $\theta_X$ is the map $\betti:\Intzz\rightarrow \Z_+$ defined by $I\mapsto \abs{\veewedge_{I} \theta_X}$.
\end{definition}

For $t\in \R$,  $\betti(t,t)=\left|\theta_X(t)\right|,$ the number of the blocks in $\theta_X(t)$. The value $\betti(I)$ in Definition \ref{def:persistent cluster counting functor} is the number of \emph{independent} groups over $I$ (cf. Definition \ref{def:group}); \ok{Two groups $G_1$ and $G_2$ over an interval $I$ are called independent if there is no pair of $x_1\in G_1$ and $x_2\in G_2$ such that $x_1$ and $x_2$ belong to the same block of $\theta_X(t)$ for some $t\in I$.  The independence of $G_1$ and $G_2$ is stronger condition than $G_1\cap G_2=\emptyset$. If $G_1$ and $G_2$ are not independent, then they are called \emph{dependent}. For the equivalence relation generated by dependence, the value $\betti(I)=\veewedge_I \theta_X$ is the number of equivalence classes of groups over $I$. Note that $\betti$ is analogous to the rank invariant of $\R$-indexed and $\ZZ$-indexed modules  \cite{carlsson2009theory,cohen2007stability,kim2018generalized}.} 

\ok{We have that \begin{equation}\label{eq:persistent cluster counting Mobius inversion}
    \betti(I)=\sum_{J\supset I} \abs{\dgm^{\veewedge}(\theta_X)}(J)
\end{equation} by Theorem \ref{thm:Mobius} and Proposition \ref{prop:silhouette commutes with mobius}. In Appendix \ref{sec:from weigted reeb to persistent cluster counting} we also prove that $\abs{\veewedge \theta_X}$ can be obtained from the unlabeled formigram of $\theta_X$ (cf. Definition \ref{def:underlying} \ref{item:unlabeled formigram}); see Proposition \ref{prop:from weigted reeb to persistent cluster counting}. This establishes the arrow (C)$\rightarrow$(H) in Figure \ref{fig:entire picture}, invoking that $\betti$ and  $\abs{\dgm^{\veewedge}(\theta_X)}$ can recover each other.} 
\subsection{Stability of the persistent cluster counting functor}\label{sec:stability of persistent cluster counting functor}

 \ok{Convention \ref{conv:real intervals and zz intervals} still applies in this section.}  \ok{For any two intervals $I\subset J$, the number of independent maximal groups over $I$ is greater than equal to that of $J$, i.e. $\abs{\veewedge_J \theta_X}\leq \abs{\veewedge_I \theta_X}$. }
 This monotonicity allows us to quantify the difference between two persistent cluster counting functors via the so-called \emph{erosion distance}.

\begin{definition}[{\cite{patel2018generalized,puuska2017erosion}}]\label{def:erosion}Let $Y_1,Y_2:\U \rightarrow \Z_+$ be any two  order-reversing maps. The \emph{erosion distance} between $Y_1$ and $Y_2$ is defined as 
\[\dero(Y_1,Y_2):=\inf\left\{\eps\in[0,\infty): \mbox{for all $I\in \Int$, } Y_i(I)\leq Y_j(I^\eps),\ \mbox{for}\ i,j\in\{1,2\}\right\}.\] 
\end{definition}

Throughout this section, $X$ and $Y$ will denote nonempty finite sets. Persistent cluster counting functor enjoys stability:

\begin{theorem}\label{thm:stability of betti}For any two formigrams $\theta_X$ and $\theta_Y$, we have
	\[\dero\left(\betti, \bettiy \right)\leq \dintf(\theta_X,\theta_Y).\] 
\end{theorem}

We prove this theorem at the end of this section. Invoking Proposition \ref{prop:stability from dg to formi}, this theorem implies that DGs can be \ok{summarized} into persistent cluster counting functors with a guarantee of stability. Example \ref{ex:last} below shows that  persistent cluster counting functors can sometimes be more discriminative than the  Reeb graph of formigrams (cf. Definition \ref{def:underlying}). 

\begin{example}\label{ex:last}Consider the two DGs $\dyngx$ and $\dyngy$ over $X:=\{x_1,x_2\}$ and $Y:=\{y_1\}$ respectively given as follows (cf. Figure \ref{fig:summarizing} (A)):
\[V_X(t)\cup E_X(t)=\begin{cases}
\{\{x_1\}\},& t\in [0,1)\cup(2,3]
\\
\{\{x_1\},\{x_2\},\{x_1,x_2\}\},& t\in [1,2]\\
\emptyset,& \mbox{otherwise,}
\end{cases}
\hspace{8mm} V_Y(t)\cup E_Y(t)=\begin{cases}
\{\{y_1\}\},& t\in [0,3]
\\
\emptyset,& \mbox{otherwise.}
\end{cases}
\]

Let $\theta_X$ and $\theta_Y$ be the formigrams of $\dynG_X$ and $\dynG_Y$, respectively (cf. Definition \ref{def:from DG to formi} and Figure \ref{fig:summarizing} (B)).  
Then $\betti,\bettiy:\U\rightarrow \Z_+$ are described in  Figure \ref{fig:summarizing} (C). By Definition \ref{def:erosion}, we have 
$\dero\left(\betti,\bettiy\right)=1$  (More details are provided below). Note that these two DGs cannot be discriminated by computing their Reeb graphs nor their zigzag barcodes.
\end{example}

\begin{remark}\label{rem:DG and formi distance computation} 
\begin{enumerate}[label=(\roman*)]
    \item The inequalities in Theorems \ref{prop:stability from dg to formi}  and \ref{thm:stability of betti} are \emph{tight}. For the DGs $\dynG_X,\dynG_Y$ and formigrams $\theta_X,\theta_Y$ in Example \ref{ex:last}, we have that $\dintg(\dynG_X,\dynG_Y)=\dintf(\theta_X,\theta_Y)=1$. Indeed, $X \xtwoheadleftarrow{\pi_1} \{(x_1,y_1),(x_2,y_1)\} \xtwoheadrightarrow{\pi_2} Y$ (where $\pi_1$ and $\pi_2$ are the canonical projections) is a 1-tripod between $\dynG_X$ and $\dynG_Y$, and also between $\theta_X$ and $\theta_Y$.  \label{item:DG and formi distance computation1}
    \item Assume that $\theta_X,\theta_Y$ are two formigrams with the same $n$ critical points. Once $\betti$ and $\bettiy$ are computed, computing $\dero(\betti,\bettiy)$  via \emph{ordinary binary search} requires the (expected) cost $O(n^2\log n)$, see \cite[Section 5]{kim2020spatiotemporal}. \label{item:DG and formi distance computation2}
    \end{enumerate}
\end{remark}

\begin{figure}
    \centering
    \includegraphics[width=\textwidth]{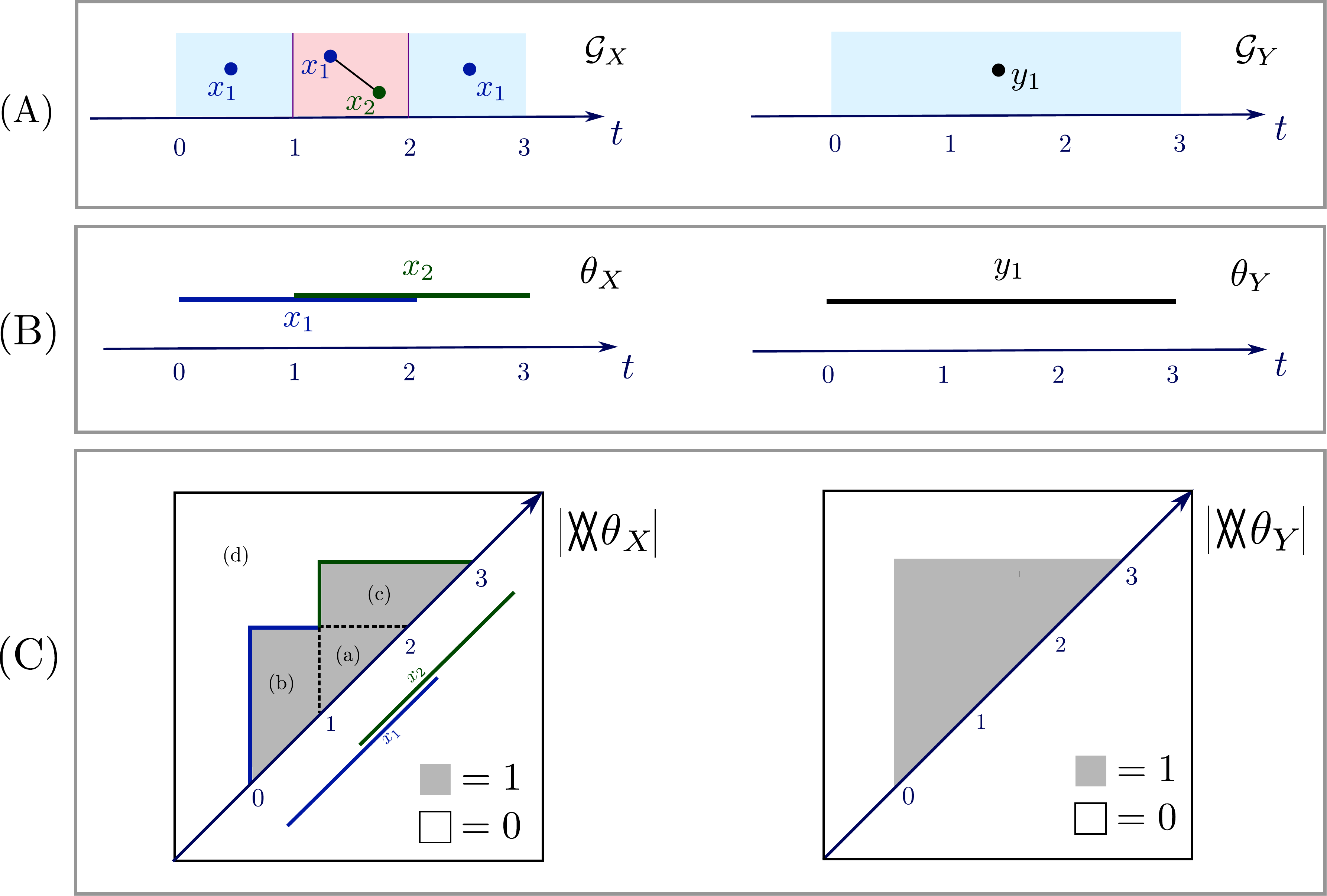}
    \caption{The summarization process of the two DGs $\dynG_X$ and $\dynG_Y$ from Example \ref{ex:last}.}
    \label{fig:summarizing}
\end{figure}

\begin{proof}[Details for Example \ref{ex:last}]\label{ex:betti0}\label{proof:last}Let us compute the formigram $\theta_X$ induced from $\dynG_X$ (cf. Definition \ref{def:from DG to formi}):
\begin{align*}
\theta_X(t)=\begin{cases}
\{\{x_1\} \},& t\in [0,1)\\
\{\{x_1,x_2\}\}, & t\in [1,2]\\
\{\{x_2\}\},& t\in (2,3]\\
\emptyset,& t\in \R\setminus [1,3]
\end{cases}\hspace{5mm}\mbox{(cf. Figure \ref{fig:summarizing} (B)).}
\end{align*}

Next we compute $\betti:\U\rightarrow\Z_+$ (cf. Figure \ref{fig:summarizing} (C)). Note that 
	\begin{equation*}
		\bigwedge_{[s,t]}\theta_X=
		\begin{cases}\begin{aligned}		
		\{\{x_1,x_2\}\},&&\mbox{$1\leq s\leq t\leq 2$,}&& \mbox{Region (a)}\\
		\{\{x_1\}\},&&\mbox{$s\in [0,1)$ and $s\leq t\leq 2$,} &&\mbox{Region (b)}\\
		\{\{x_2\}\},&&\mbox{$t\in (2,3]$ and $1< s\leq t$,}&& \mbox{Region (c)}\\
		\emptyset,&&\mbox{otherwise,}&& \mbox{Region (d),}\\	
		\end{aligned}\end{cases}
	\end{equation*}
	where regions (a),(b),(c), and (d) are marked in Figure \ref{fig:summarizing} (C). Notice that the cardinality of $\bigwedge_{[s,t]}\theta_X$ is $1$ for $[s,t]$ in regions (a),(b) and (c). Thus, we directly have that $\betti[s,t]=1$ for $[s,t]\in\U$ in these regions. Also, since $\bigwedge_{[s,t]}\theta_X$ is empty for $[s,t]$ in region (d), we also have that $\betti[s,t]=0$ for $[s,t]$ in region (d). Similarly, one can compute $\theta_Y:\R\rightarrow \subpart(Y)$ and $\bettiy:\U\rightarrow \Z_+$ as
	\[\theta_Y(t)=\begin{cases} \{y_1\},& t\in[0,3] \\ \emptyset,&\mbox{otherwise,}
	\end{cases}\hspace{3mm}	\mbox{and}\hspace{3mm} \bettiy[s,t]=\begin{cases} 1,& 0\leq s\leq t\leq 3 \\ 0,&\mbox{otherwise.}
	\end{cases}\hspace{3mm}\mbox{(see Figure \ref{fig:summarizing} (B),(C))}\]
	Observe from Figure \ref{fig:summarizing} (C) that $\dero(\betti,\bettiy)=1$. 
\end{proof}

\begin{remark}\label{rem:tree of distances}In general, we have the following ``poset of distances":
\begin{center}
\begin{tikzpicture}
    \node (top) at (0,0) {$\dintf(\theta_X,\theta_Y)$};
    \node (middle left top) at (-3,-1) {$\dintwreeb(\omega(\theta_X),\omega(\theta_Y))$};
    \node (middle left) at (-3,-2)  { $\dintreeb(\reeb(\theta_X),\reeb(\theta_Y))$};
    \node (middle right) at (3,-2) {$\dero(\betti,\bettiy)$};
    \node (bottom left) at (-3,-3) {$\bott (\barc(\theta_X),\barc(\theta_Y))$};
    \draw [thick, shorten <=-2pt, shorten >=-2pt] (top) -- (middle left top);
    \draw [thick, shorten <=-2pt, shorten >=-2pt] (middle left top) -- (middle left);
    \draw [thick, shorten <=-2pt, shorten >=-2pt] (top) -- (middle right);
     \draw [thick, shorten <=-2pt, shorten >=-2pt] (middle left) -- (bottom left);
\end{tikzpicture}
\end{center}
where the distance $\dintwreeb$ between weighted Reeb graphs is defined in Appendix \ref{sec:distance between weighted reeb graphs}.
When $\theta_X$ and $\theta_Y$ are \emph{saturated formigrams}, we have the following totally ordered hierarchy: \[\dintf(\theta_X,\theta_Y)\geq \dintwreeb(\omega(\theta_X),\omega(\theta_Y))\geq \dintreeb(\reeb(\theta_X),\reeb(\theta_Y)) \geq  \bott (\barc(\theta_X),\barc(\theta_Y))\geq \dero(\betti,\bettiy).
\]
The last inequality follows from Theorem \ref{thm:persistence clustergram of a saturated formigram} and the well-known inequality $\bott\geq \dero$ for (standard) persistence diagrams \cite{mccleary2018bottleneck,patel2018generalized}.
\end{remark}

\begin{proof}[Proof of Theorem \ref{thm:stability of betti}]\label{proof:stability of betti} If $\dintf(\theta_X,\theta_Y)=+\infty$, there is nothing to prove. For some $\eps\in \R_+$, assume that $\tripod$ is an $\eps$-tripod between $\theta_X$ and $\theta_X$ (cf. Definition \ref{def:interleaving distance2}). Fix $I\in \U$ and we will only prove that $\betti(I^\eps)\leq \bettiy(I).$ If $\bigwedge_{I^\eps}\theta_X$ is empty, then $\betti(I^\eps)=0$ by definition and hence there is nothing to prove. Suppose that $\bigwedge_{I^\eps}\theta_X$ is nonempty. We show that there are two set maps $\psi_1:\bigwedge_{I^\eps}\theta_X\rightarrow \bigwedge_{I}\theta_Y$ and $\psi_2:\bigvee_{I}\theta_Y\rightarrow \bigvee_{I^\eps}\theta_X$ which make the following diagram commutes.
\begin{center}
	\begin{tikzcd}[column sep=huge]
		\bigwedge_{I^\eps}\theta_X\arrow{r}{\phi^X(I^\eps)}\arrow{d}[swap]{\psi_1}&\bigvee_{I^\eps}\theta_X\\ \bigwedge_{I}\theta_Y\arrow{r}{\phi^Y(I)}&\bigvee_{I}\theta_Y\arrow{u}[swap]{\psi_2}.
	\end{tikzcd}
\end{center}
Here the maps $\phi^X(I^\eps)$ and $\phi^Y(I)$ are the canonical maps in $\subpart(X)$ and $\subpart(Y)$, respectively (cf. Definition \ref{def:natural}). Indeed, if we have such maps $\psi_1$ and $\psi_2$, then we have:
\begin{align*}
\betti(I^\eps)=\abs{\im\left(\phi^X(I^\eps)\right)}\leq \abs{\im\left( \phi^Y(I)\right)}=\bettiy(I),
\end{align*}as desired.

Let us construct two such maps $\psi_1$ and $\psi_2$. First, define set maps $f:X\rightarrow Y$ and $g:Y\rightarrow X$ such that
\begin{equation}\label{eq:the map f}
	\{(x,f(x)):x\in X\}\cup \{(g(y),y):y\in Y\}\subset R,
\end{equation} 
which is possible since $\varphi_X$ and $\varphi_Y$ are surjective (cf. Notation \ref{not:belongs to a tripod}).
We now construct the map $\psi_1:\bigwedge_{I^\eps}\theta_X\rightarrow \bigwedge_{I}\theta_Y$. Let $B\in \bigwedge_{I^\eps}\theta_X$ and pick any $x\in B$. By definition of $\bigwedge_{I^\eps}\theta_X$, this implies that $x$ is in the underlying set of $\theta_X(t)$ for \emph{all} $t\in I^\eps$. 

We claim that the lifespan $I_{f(x)}$ of $f(x)\in Y$ in the formigram $\theta_Y$ contains the interval $I=:[s,t]$. Since $I_{f(x)}$ is an interval in $\R$ (cf. Definition \ref{def:formigram}), it suffices to prove that $I_{f(x)}$ contains some $\alpha\in (-\infty,s]$ and $\beta\in [t,\infty)$. Invoking that $R$ is an $\eps$-tripod between $\theta_X$ and $\theta_Y$, we have that
\begin{equation}\label{eq:R refinement}
    \theta_X(s-\eps)\leq_R\bigvee_{[s-\eps]^\eps}\theta_Y=\bigvee_{[s-2\eps,s]}\theta_Y.
\end{equation}
Also, since $x$ is in the underlying set of $\theta_X(s-\eps)$ and $(x,f(x))\in R$ (from inclusion (\ref{eq:the map f})), relation (\ref{eq:R refinement}) implies that the element $f(x)\in Y$ must be in the underlying set of $\bigvee_{[s-2\eps,s]}\theta_Y$ (cf. Remark \ref{prop:morphism for formi} \ref{item:morphism for formi1}). This implies that there exists $\alpha\in [s-2\eps,s]$ such that $y$ is in the underlying set of $\theta_Y(\alpha)$. Hence, $\alpha$ belongs to $I_{f(x)}$. Similarly, one can also check that there exists $\beta\in [t,t+2\eps]$ such that $\beta\in I_{f(x)}$. Therefore, we have that $I\subset I_{f(x)}$.

This inclusion implies that $\bigwedge_{I}\theta_Y$ contains a block $C$ which in turn contains $f(x)$. We define $\psi_1$ as the map sending $B$ to $C$. Note that the definition of $\psi_1$ depends on the choice of a representative element $x$ for each block $B$ in $\bigwedge_{I^\eps}\theta_X$. 

Next, by invoking that $R$ is an $\eps$-tripod between $\theta_X$ and $\theta_Y$, inclusion (\ref{eq:the map f}), and Remark \ref{prop:morphism for formi} \ref{item:morphism for formi2}, we define the map $\psi_2:\bigvee_{I}\theta_Y \rightarrow \bigvee_{I^\eps}\theta_X$ as follows: for each $C\in \bigvee_{I}\theta_Y$, let $\psi_2(C)$ be the unique block $B\in \bigvee_{I^\eps}\theta_X$ containing the image of $C$ via the map $g:Y\rightarrow X$. 

It remains to verify that $\phi^X(I^\eps)=\psi_2\circ \phi^Y(I)\circ \psi_1.$ Pick any block $B \in \bigwedge_{I^\eps}\theta_X$.  Since $\phi^X(I^\eps)$ is the canonical map in $\subpart(X)$, $B$ is sent to the unique block $B'\in \bigvee_{I^\eps}\theta_X$ containing the whole block $B$. 
Now, let $x\in B$ be the representative element which was used for defining $\psi_1$. Via $\psi_1$, the block $B$ is sent to the block $C\in \bigwedge_{I}\theta_Y$ containing $f(x)$. Then, the map $\phi^Y(I)$ sends $C$ to the unique block $C'\in \bigvee_{I}\theta_Y$ containing $f(x)$ (and the whole block $C$). From Remark \ref{prop:morphism for formi} \ref{item:morphism for formi2}, we have that $\bigvee_{I}\theta_Y\leq_{R}\bigvee_{I^\eps}\theta_X$. Thus, by inclusion (\ref{eq:the map f}), we conclude that $\psi_2$ sends $C'(\ni f(x))$ to the block $B' (\ni x)$, completing the proof. \end{proof}

	\section{Analysis of dynamic metric spaces (DMSs)}\label{sec:DMSs}

		\ok{Our work is motivated by the desire to construct a well-defined summarization tool of clustering behavior of time-varying metric data, which is modeled as \emph{dynamic metric spaces (DMSs).} In Section \ref{sec:dms definition}, we define DMSs. In Section \ref{sec:from DMSs to DGs}, we establish a sufficient condition for DMSs to be converted into DGs via the Rips graph functor (cf. Proposition \ref{prop:from dms to filt2}). This enables us to produce summaries of DMSs such as those which are illustrated in Figure \ref{fig:entire picture}.\footnote{In \cite[Section 5]{kim2020analysis}, DMSs generated by \emph{Boid} \cite{boids} were successfully classified by the bottleneck distance on their zigzag barcodes (cf. Figure \ref{fig:entire picture} (F)).} In Section \ref{sec:distance between DMSs}, we define the \emph{$\lambda$-slack interleaving distance} $\dintl$ between DMSs. Although this metric was utilized in \cite{kim2020spatiotemporal}, all properties of $\dintl$ mentioned therein are proved in this paper, including the fact that $\dintl$ is a metric. }

	\subsection{DMSs}\label{sec:dms definition}
	
	In this section we introduce definitions pertaining to our model for dynamic metric spaces (DMSs). In particular, \emph{tameness} (cf. Definition \ref{def:tameDMS}) is a crucial requirement on DMSs, which permits transforming DMSs into DGs via the Rips graph functor (cf. Proposition \ref{prop:from DMS to DG}), thus subsequently into formigrams, persistence clustergrams, Reeb graphs, and barcodes.

	\begin{definition}[\cite{kim2020spatiotemporal}]\label{dynamic} A \emph{dynamic metric space}  is a  pair $\gamma_X = (X,d_X(\cdot))$ where $X$ is a nonempty finite set and $d_X:\T\times X\times X\rightarrow \R_+$ satisfies:\begin{enumerate}[label=(\roman*),itemsep=-1ex]
			\item For every $t\in\T$, $\gamma_X(t)=(X,d_X(t))$ is a pseudo-metric space.\label{item:dynamic1}		
			\item For any $x, x'\in  X$ with $x\neq x'$ the function $d_X(\cdot)(x, x'): \T \rightarrow \R_+$ is not
identically zero.			\label{item:dynamic2}	
			\item For fixed $x,x'\in X$,	$d_X(\cdot)(x,x'):\T\rightarrow \R_+$  is continuous. \label{item:dynamic3}	
		\end{enumerate}
		We refer to $t$ as the \emph{time} parameter.
	\end{definition}
	
	We remark that a DMS $\gamma_X$ is not just a continuous curve in the Gromov-Hausdorff space \cite{chowdhury2016explicit}, but it also keeps track of the identities of points in $X$.  Item \ref{item:dynamic2} is assumed to avoid redundancy; otherwise there could be two points which are forever identified.
	
	\begin{example}\label{ex:dmss}                                                                        An  example is given by $n$ particles/animals  moving continuously inside an environment $\Omega\subset \R^d$ where particles are allowed to coalesce. 
			If the $n$ trajectories are $p_1(t),\ldots,p_n(t)\in \R^d$, then let $P:=\{1,\ldots,n\}$ and define a DMS $\gamma_P := (P,d_P(\cdot))$ as follows: for $t\in \T$ and $i,j\in\{1,\ldots,n\}$, let $d_P(t)(i,j):=\|p_i(t)-p_j(t)\|,$ where $\|\cdot\|$ denotes the Euclidean norm. Buchin et al. considered this setting \cite{buchin2013trajectory}.    
	\end{example}                                                           
	
	We now introduce a notion of \emph{equality} between two DMSs.                                                                                                                                                                      
	\begin{definition}\label{def:isomorphism} Let $\gamma_X = (X,d_X(\cdot))$ and $ \gamma_Y=(Y,d_Y(\cdot))$ be two DMSs. We say that $\gamma_X$ and $\gamma_Y$ are  \emph{isomorphic} if there exists a bijection $\varphi:X\rightarrow Y$ such that $\varphi$ is an isometry between $\gamma_X(t)$ and $\gamma_Y(t)$  across all $t\in \T$.
	\end{definition}

	\subsection{From DMSs to DGs}\label{sec:from DMSs to DGs}
	
	We introduce a notion of \emph{tame DMS} which will ultimately ensure that the Rips graph functor induces DGs satisfying our definition (cf. Definition \ref{def:dyn graphs} and Example \ref{ex:special DGs}). 
	A continuous function $f:\R\rightarrow\R$ is called \emph{tame}, if for any $c\in\R$ and any finite interval  $I\subset\R,$  the set $f^{-1}(c)\cap I\subset\R$ has only finitely many connected components (and is possibly empty). Elementary functions including polynomial functions (in particular, constant functions), piecewise linear functions with locally finitely many critical points are tame.
	
	\begin{definition}
	\label{def:tameDMS}
		A DMS $\gamma_X=(X,d_X(\cdot))$ is said to be \emph{tame} if for any $x, x'\in X$ the function $d_X(\cdot)(x,x'):\T\rightarrow \R_+$ is tame. 
	\end{definition}
	
For $\delta\geq 0$ and for any finite metric space $(X,d_X)$, let $\rips^1(X,d_X)$ be the 1-skeleton of the $\delta$-Rips complex of $X$, i.e. a simple graph over the vertex set $X$ with the edge set $E_X=\left\{\{x,x'\}\subset X: d_X(x,x')\leq \delta  \ \mbox{and} \ x\neq x'\right\}$.

    \begin{proposition}\label{prop:from dms to filt2}  Let $\gamma_X$ be a \emph{tame} DMS over $X$ and let $\delta\geq 0$. Then, by defining $\rips^1(\gamma_X)(t):=\rips^1(\gamma_X(t))$ for $t\in\T$, $\rips^1(\gamma_X):\R\rightarrow \graph(X)$ is a saturated DG over $X$. 
	\end{proposition}	
	The proof of Proposition \ref{prop:from dms to filt2} is given in Appendix \ref{proof:from dms to filt2}. Let $\gamma_X$ be a tame DMS over $X$. By Proposition \ref{prop:from dms to filt2} and Definition \ref{def:from DG to formi}, one can obtain a formigram $\theta_X:=\pi_0\left(\rips^1(\gamma_X)\right)$. 

	\subsection{The $\lambda$-slack interleaving distance between DMSs}\label{sec:distance between DMSs}
	
	The main goal of this section is to introduce a $[0,\infty)$-parametrized family $\left\{\dintl\right\}_{\lambda\in[0,\infty)}$ of extended metrics for DMSs. Each  metric in this family is a hybrid between the Gromov-Hausdorff distance \cite{buragobook} and the interleaving distance \cite{CCG09,de2016categorified}. We obtain a stability result with respect to the most stringent metric (the metric corresponding to $\lambda=0$) in the family  (cf. Theorem \ref{thm:main thm2}). 
	
	\begin{definition}Let $\eps\geq 0$. Given any map $d:X\times X\rightarrow \R$, by $d+\eps$ we denote the map from $X\times X$ to $\R$ defined by $(d+\eps)(x,x')=d(x,x')+\eps$ for all $(x,x')\in X\times X.$
	\end{definition}
	
	\begin{definition}\label{def:bigvee}Given any DMS $\gamma_X=(X,d_X(\cdot))$ and any interval $I\subset \T$, define the map  $\mbox{$\bigvee_{I}d_X:X\times X\rightarrow \R_+$}$ by\\ $\mbox{$\left(\bigvee_{I}d_X\right)(x,x'):= \min_{t\in I}d_X(t)(x,x')$}$ for all $(x,x')\in X\times X.$ 
	\end{definition}
	
	Given any map $d:X\times X \rightarrow \R$, let $Z$ be any set and let $\varphi:Z\rightarrow X$ be any map. Then, we define the pullback $\varphi^*d:Z\times Z\rightarrow \R$ of $d$ under $\varphi$ by
	$$\varphi^*d(z,z'):= d\left(\varphi(z),\varphi(z')\right)$$for all $(z,z')\in Z\times Z.$	Consider any two functions $d_1:X \times X\rightarrow \R$ and $d_2:Y\times Y \rightarrow \R$. Given a tripod \\$\tripod$ between $X$ and $Y$, by $d_1\leq_R d_2$ we mean $\varphi_X^*d_1(z,z')\leq \varphi_Y^*d_2(z,z')$ for all $(z,z')\in Z\times Z$. 

	Recall that for any $t\in \T$, $[t]^\eps:=[t-\eps,t+\eps]$. 
	
	\begin{definition}[$\lambda$-distortion of a tripod]\label{def:distortion}Fix $\lambda\geq 0$. Let $\gammax$ and \hbox{$\gammay$} be any two DMSs. Let $\tripod$ be a tripod between $X$ and $Y$ such that 
		\begin{equation}\label{eq:distor}
		\mbox{for all  $t\subset \T$},\ \ \bigvee_{[t]^\eps}d_X\leq_R d_Y(t)+\lambda\eps\ \mbox{and}\  \bigvee_{[t]^\eps}d_Y\leq_R d_X(t)+\lambda\eps.	
		\end{equation}
		We call any such $R$ a \emph{$(\lambda,\eps)$-tripod} between $\gamma_X$ and $\gamma_Y$. Define the \emph{$\lambda$-distortion} $\dyndis_\lambda(R)$ of $R$ to be the infimum of those $\eps\geq 0$ for which $R$ is a $(\lambda,\eps)-$tripod.
	\end{definition}
	
	\begin{example} For $\lambda>0$, $\dyndis_\lambda(R)$ takes into account both spatial and temporal distortion of the tripod $R$ between $\gamma_X$ and $\gamma_Y$:	
		\begin{enumerate}[label=(\roman*)]
			\item (Spatial distortion) Let $a,b\geq 0$. For the two metric spaces $(X,d_{X,a})=\left(\{x,x'\},\begin{psmallmatrix}0 & a\\a & 0\end{psmallmatrix}\right)$ and $(X,d_{X,b})=\left(\{x,x'\},\begin{psmallmatrix}0 & b\\b & 0\end{psmallmatrix}\right),$ consider the two \emph{constant} DMSs $\gamma_{X,a}\equiv (X,d_{X,a})$ and $\gamma_{X,b}\equiv (X,d_{X,b})$. Take the tripod $R: X\xtwoheadleftarrow[]{\mathrm{id}_X}X \xtwoheadrightarrow[]{\mathrm{id}_X} X.$ Then, for $\lambda>0$, it is easy to check that $\dyndis_\lambda(R)=\frac{\abs{a-b}}{\lambda}.$
			\item (Temporal distortion) Fix $\tau\geq 0$. Let $\gammax$ be any DMS and define any continuous map $\alpha:\T\rightarrow \T$ such that $\norm{\alpha-\mathrm{id}_\T}_\infty\leq \tau.$ Define the DMS $\gamma_X\circ \alpha:=(X,d_X(\alpha(\cdot)))$, i.e. for $t\in \T$, $\gamma_X\circ\alpha(t)=(X,d_X(\alpha(t)))$. Take the tripod $R: X\xtwoheadleftarrow[]{\mathrm{id}_X}X \xtwoheadrightarrow[]{\mathrm{id}_X} X.$ Then, for any $\lambda\geq0$, it is easy to check that $\dyndis_\lambda(R)\leq \tau.$
		\end{enumerate} 
	\end{example}

	\begin{remark}
		In Definition \ref{def:distortion}, if $R$ is a $(\lambda,\eps)$-tripod, then $R$ is also a $(\lambda,\eps')$-tripod for any $\eps'>\eps$: Fix any $t\subset \T$. If for some $\eps\geq 0$, $\bigvee_{[t]^\eps}d_X\leq_R d_Y(t)+\lambda\eps$, then for any $\eps'>\eps$, 
		$$ \bigvee_{[t]^{\eps'}}d_X \leq \bigvee_{[t]^\eps}d_X\leq_R d_Y(t)+\lambda\eps<d_Y(t)+\lambda\eps'.$$ 
	\end{remark}
	
	Now we introduce a family of metrics for DMSs. 
	\begin{definition}[The $\lambda$-slack interleaving distance between DMSs]\label{def:lambda dist} For each $\lambda\geq0$, we define the \emph{$\lambda$-slack interleaving distance} between any two DMSs $\gammax$ and $\gammay$ as
		$$\dintl(\gamma_X,\gamma_Y):=\min_R\dyndis_\lambda(R) $$
		where the minimum ranges over all tripods between $X$ and $Y$. For simplicity, when $\lambda=0$, we write $\dintm$ instead of $d_{\mathrm{I},0}^\mathrm{dynM}$. If $\dintm(\gamma_X,\gamma_Y)\leq\eps$ for some $\eps\geq 0$, then we say that $\gamma_X$ and $\gamma_Y$ are \emph{$\eps$-interleaved} or simply \emph{interleaved}.  By definition, it is clear that for all $\lambda>0$, $\dintl\leq \dintm$.
	\end{definition}
	
 For $r>0$, we call any DMS $\gamma_X=(X,d_X(\cdot))$ $r-$\emph{bounded} if the distance between any pair of points in $X$ does not exceed $r$ across all $t\in \T$. If $\gamma_X$ is $r-$bounded for some $r>0$, then $\gamma_X$ is said to be simply \emph{bounded}.
	\begin{theorem}\label{thm:lambda metric}For each $\lambda\geq 0,$ $\dintl$ is an extended metric between DMSs modulo isomorphism. In particular, for $\lambda>0$, $\dintl$ is a metric between bounded DMSs modulo isomorphism.
	\end{theorem}

	The proof of Theorem \ref{thm:lambda metric} together with details pertaining to the following facts are deferred to Appendix \ref{sec:details on distance between DMSs}: 
	\begin{enumerate}[label=(\roman*),itemsep=-1ex]
		\item For $\lambda>0$, $\dintl$ generalizes the Gromov-Hausdorff distance (cf. Remark \ref{rem:gromov}).
		\item The metrics $\dintl$, for different $\lambda>0$, are bilipschitz-equivalent (cf. Proposition \ref{prop:equivalence}).	
		\item In Proposition \ref{prop:ultra2} we will elucidate a link between $\dintm$ and the Gromov-Hausdorff distance. This link will be useful for determining the computational complexity of $\dintm$ (cf. Theorem \ref{thm:complex}).
	\end{enumerate}

	In the rest of this section, we discuss several properties of $\dintm$.

	\medskip\noindent\textbf{Stability results.} The following proposition provides a gateway for extending the stability results that are illustrated in Figure \ref{fig:entire picture} to the setting of DMSs. 
	\begin{proposition}[Stability of summarizing DMSs into DGs]
	\label{prop:from DMS to DG}Let $\gammax$ and $\gammay$ be any tame DMSs. Fix any $\delta\geq 0$. Consider saturated DGs $\dynG_X:=\rips^1(\gamma_X), \dynG_Y:=\rips^1(\gamma_Y)$, as in Proposition \ref{prop:from dms to filt2}. Then, $$\dint^\mathrm{dynG}(\dynG_X,\dynG_Y)\leq \dint^\mathrm{dynM}(\gamma_X,\gamma_Y).$$  
	\end{proposition}
	
	\begin{proof} The proof follows from the fact that for any $\eps\geq 0$, any $(0,\eps)$-tripod $R$ between $\gamma_X$ and $\gamma_Y$ (Definition \ref{def:lambda dist}) is also an $\eps$-tripod between $\dynG_X$ and $\dynG_Y$ (cf. Definition \ref{def:DG interleaving distance}). 
\end{proof}

	A priori Proposition \ref{prop:from DMS to DG} does not seem to be a satisfactory stability theorem in that the RHS can be infinity in many cases in comparison with the LHS. Nevertheless, this `weak' stability seems to be the most we can expect for the mapping [DMSs] $\mapsto$ [DGs] via the Rips graph functor because DMSs change continuously over time (cf. Definition \ref{dynamic} \ref{item:dynamic3}) whereas DGs are discontinuous in the sense that edges appear and disappear over time.	
	
	Proposition \ref{prop:from DMS to DG} together with Theorem \ref{thm:reeb graph barcode stability}, Proposition \ref{prop:formigram interleaving is mor discriminative than the reeb interleaving} and Theorem \ref{prop:stability from dg to formi} directly imply:
	
	\begin{theorem}\label{thm:main thm2}Let $\gammax$ and $\gammay$ be any two tame DMSs. Fix any $\delta\geq 0$. Consider the saturated formigrams $\theta_X:=\pi_0\left(\rips^1 (\gamma_X)\right)$ and $\theta_Y:=\pi_0\left(\rips^1 (\gamma_Y)\right)$. \footnote{These are formigrams by Proposition \ref{prop:from dms to filt2} and Definition \ref{def:from DG to formi}.} Then,
		$$\bott\left(\barc(\theta_X),\barc(\theta_Y)\right)\leq 2\ \dint^\mathrm{dynM}(\gamma_X,\gamma_Y).$$
	\end{theorem}
	
	We discuss the generalization of Theorem \ref{thm:main thm2} to higher dimensional homology barcodes in Appendix \ref{sec:higher dimensional barcodes}.

	\paragraph{Computational complexity $\dintm$.}  A DMS $\gamma_X=(X,d_X(\cdot))$ is said to be \emph{piecewise linear} if for all $x,x'\in X$, the function $d_X(\cdot)(x,x):\T\rightarrow\R_+$ is piecewise linear. We denote by $S_X$ the set of all breakpoints of all distance functions $d_X(\cdot)(x,x')$, $x,x'\in X$.
	
	\begin{theorem}[Complexity of  $\dintm$]\label{thm:complex}
		Fix $\rho\in(1,6)$ and let $\gamma_X$ and $\gamma_Y$ be piecewise linear DMSs. Then, it is not possible to compute a $\rho$-approximation to $\dintm(\gamma_X,\gamma_Y)$ in time depending polynomially on $|X|,|Y|,|S_X|,$ and $|S_Y|$, unless $P=NP$. 
	\end{theorem}
	
	Theorem \ref{thm:complex} will be proved in Appendix \ref{proof:complex}. More examples about $\dintm$ are provided in Appendix \ref{sec:0-slack DMS intereleaving} as well.
	The theorem above indicates that computing the lower bound for $\dintm$ given by Theorem \ref{thm:main thm2} is a realistic approach to comparing DMSs, especially when one is interested in a specific spatial scale.

	\section{Conclusion}\label{sec:discussion}
	
	We have established a stable mathematical framework for summarizing dynamic graphs, which often arise as discrete representations of spatiotemporal data. The evolution of connected components of a dynamic graph is fully encoded in a \emph{formigram}, a constructible cosheaf over $\R$ valued in the lattice of subpartitions.\footnote{In the preprint of this work, we also considered formigrams derived from \emph{directed} dynamic graphs \cite{kim2017stable}.} The lattice structure of subpartitions of a given set has been indispensable for obtaining:
	\begin{enumerate}[label=(\roman*),itemsep=-1ex]
	\item  the formigram interleaving distance $\dintf$ which is more discriminative than the well-known Reeb graph interleaving distance.
	\item the maximal group diagram and persistence clustergram which are complete and visualizable invariants of formigrams (and thus more discriminative invariants than the Reeb graph of a formigram). 
	\end{enumerate}
	Note that we can directly extend the definitions of the maximal group diagram and the persistence clustergram to \emph{any} cellular cosheaf over a topological space other than $\R$ \cite{curry2014sheaves}. Indeed, our two main ingredients, M\"obius inversion and the lattice structure of subpartitions, do not depend on any specific property of $\R$. Exploring this extension in relation to \emph{Reeb spaces} \cite{patel2010reeb} is left for future work. Extending the functorial pipeline in the left column of Figure \ref{fig:entire picture} to the entire diagram is also left for the future  \cite{mccleary2020edit}. \ok{Devising a  bottleneck-type distance which can directly measure the difference between persistence clustergrams or between maximal group diagrams is of independent interest.}

	\appendix
	
	\section{Details and Proofs}\label{sec:details}
	
		\subsection{Bottleneck distance}\label{sec:interleaving and bottleneck}

	Recall that injective partial functions are referred to as \emph{matchings}. We use $\sigma:A\nrightarrow B$ to denote a matching $\sigma\subset A\times B$ between sets $A$ and $B$. The canonical projections of $\sigma$ onto $A$ and $B$ are denoted by $\mathrm{coim} (\sigma)$ and $\mathrm{im} (\sigma)$, respectively.\\ \indent Many equivalent expressions for the \emph{bottleneck distance} have been given in the TDA literature. We adopt the following form from \cite{bauer2015induced}: Recall Notation \ref{not:intervals in zz}.	Letting $\mathcal{A}$ be a multiset of intervals in $\R$ and $\eps\geq 0$, $$\mathcal{A}^\eps:=\lmulti\langle b,d\rangle \in \mathcal{A}:b+\eps<d \rmulti=\lmulti I\in\mathcal{A}:[t,t+\eps]\subset I\ \mbox{for some}\ t\in\R\rmulti.$$ Note that $\mathcal{A}^0=\mathcal{A}$.  
	\begin{definition}[\cite{bauer2015induced}]\label{def:bottleneck} Let $\mathcal{A}$ and $\mathcal{B}$ be multisets of intervals in $\R$. We define a $\delta$-matching between $\mathcal{A}$ and $\mathcal{B}$ to be a matching $\sigma:\mathcal{A}\nrightarrow\mathcal{B}$ such that 
		$\mathcal{A}^{2\delta}\subset  \mathrm{coim} (\sigma)$, $\mathcal{B}^{2\delta}\subset \mathrm{im} (\sigma)$, and if $\sigma\langle b,d\rangle=\langle b',d'\rangle$, then $$\langle b,d\rangle \subset \langle b'-\delta, d'+\delta \rangle,\hspace{5mm} \langle b',d'\rangle \subset \langle b-\delta, d+\delta \rangle.$$
		with the convention $+\infty+\delta=+\infty$ and $-\infty-\delta=-\infty$. We define the bottleneck distance $\bott$ by
		$$\bott(\mathcal{A},\mathcal{B}):=\inf\{\delta\in[0,\infty): \exists \mbox{$\delta$-matching between $\mathcal{A}$ and $\mathcal{B}$}\}.$$We declare $\bott(\mathcal{A},\mathcal{B})=+\infty$ when there is no $\delta$-matching between $\mathcal{A}$ and $\mathcal{B}$ for any $\delta\in[0,\infty)$.
	\end{definition} 

	\subsection{Proof of Theorem \ref{thm:DG metric-complexity}}
	
	We recall the \emph{Gromov-Hausdorff distance} between metric spaces. Let $(X,d_X)$ and $(Y,d_Y)$ be any two metric spaces and let $\tripod$ be a tripod between $X$ and $Y$. Then, the \emph{distortion} of $R$ is defined as $$\displaystyle\dis(R):=\sup_{\substack{z,z'\in Z}}\abs{d_X\left(\varphi_X(z),\varphi_X(z')\right)-d_Y\left(\varphi_Y(z),\varphi_Y(z')\right)}.$$  
	\begin{definition}[Gromov-Hausdorff distance {\cite[Section 7.3]{buragobook}}] \label{def:the GH} Let $(X,d_X)$ and $(Y,d_Y)$ be any two compact metric spaces. Then, 
		$$\dgh\left((X,d_X),(Y,d_Y)\right)=\frac{1}{2}\inf_R\ \dis(R)$$
		where the infimum is taken over all tripods $R$ between $X$ and $Y$. In particular, any tripod $R$ between $X$ and $Y$ is said to be an \emph{$\eps$-tripod} between $(X,d_X)$ and $(Y,d_Y)$ if $\dis(R)\leq \eps$.
	\end{definition}
	
	\begin{proposition}\label{prop:computation}Let $(X,d_X)$ and $(Y,d_Y)$ be any two finite metric spaces. Then, there exist two DGs $\dyngx$ and $\dyngy$ corresponding to $(X,d_X)$ and $(Y,d_Y)$ respectively such that $$\dintg(\dynG_X,\dynG_Y)=2\cdot \dgh\left((X,d_X),(Y,d_Y)\right).$$
	\end{proposition}
	
	\begin{proof}
Let $T$ be the diameter of $(X,d_X)$.		
		For $t\in \R$, we define:
	$$V_X(t):=\begin{cases}
		\emptyset,&t\notin[0,T]\\ X,&t\in[0,T],
		\end{cases}\hspace{2mm}E_X(t):=\begin{cases}
		\emptyset,&t\notin[0,T]\\ \{\{x,x'\}\subset X:  x\neq x' \mbox{ and } d_X(x,x')\leq t\},&t\in[0,T].
		\end{cases}$$ 
		We define $\dynG_X$ by $t\mapsto (V_X(t),E_X(t))$. Define $\dynG_Y$ similarly. We show that $\dintg(\dynG_X,\dynG_Y)\geq 2\cdot \dgh\left((X,d_X),(Y,d_Y)\right).$ To this end, suppose that for some $\eps\geq 0$, $\tripod$ is any $\eps$-tripod between $\dynG_X$ and $\dynG_Y$ (Definition \ref{def:DG interleaving distance}). Then, by the construction of $\dynG_X,\dynG_Y$, it must hold that $\abs{d_X\left(\varphi_X(z),\varphi_X(z')\right)-d_Y\left(\varphi_Y(z),\varphi_Y(z')\right)}\leq \eps$ for all $z,z'\in Z.$ The other inequality $\dintg(\dynG_X,\dynG_Y)\leq 2\cdot \dgh\left((X,d_X),(Y,d_Y)\right)$ can be similarly proved. 
	 \end{proof}
	 
	 \begin{definition}\label{def:ultrametric space}
	 \ok{An \emph{ultrametric space} is a metric space $(X,d)$ in which the following ultra-triangle inequality holds: for all $x,y,z\in X$, \[d(x,z)\leq \max \left\{d(x,y),d(y,z)\right\}.\]}
	\ok{If $(X,d)$ were a pseudometric, then $d$ is called an ultra-pseudometric.}
	 \end{definition}
	\paragraph{Proof of Theorem \ref{thm:DG metric-complexity}}\label{proof:DG-metric-complexity}
		Pick any two ultrametric spaces $(X,u_X)$ and $(Y,u_Y)$. Then, by Proposition \ref{prop:computation}, there exist DGs $\dyngx$ and $\dyngy$ such that the interleaving distance between $\dynG_X$ and $\dynG_Y$ is identical to twice the Gromov-Hausdorff distance $\Delta:=d_\mathrm{GH}((X,u_X),(Y,u_Y))$ between  $(X,u_X)$ and $(Y,u_Y)$. However, according to \cite[Corollary 3.8]{schmiedl2017computational}, $\Delta$ cannot be approximated within any factor less than $3$ in polynomial time, unless $P=NP$. The author shows this by observing that any instance of the $3$-partition problem can be reduced to an instance of the bottleneck $\infty$-Gromov-Hausdorff distance ($\infty$-BGHD) problem between ultrametric spaces (cf. \cite[p.865]{schmiedl2017computational}). The proof follows.	\qed
	
	\subsection{Details about Remark \ref{rem:complexity of df}}
	
	\begin{remark}[Interleaving between dendrograms]\label{rem:df motive} When $\theta_X,\theta_Y$ are dendrograms over sets $X$ and $Y$ respectively, let $\tripod$ be an $\eps$-tripod between $\theta_X$ and $\theta_Y$. Since both $\theta_X$ and $\theta_Y$ get coarser as $t\in\R$ increases,  the interleaving condition in Definition \ref{def:interleaving distance2} can be rewritten as follows: for all $t\in\R$ it holds that
		$\theta_X(t) \leq_{R} \theta_Y(t+\eps)$ and $\theta_Y(t) \leq_{R} \theta_X(t+\eps)$ (cf. Definition \ref{def:partition morphism}).	
	\end{remark}
	Let $X$ be a finite set and let $\theta_X$ be a dendrogram over $X$ (cf. Remark \ref{rem:about the definition of formigrams}). Recall from \cite{clustum} that this $\theta_X$ induces a canonical ultra-pseudometric $u_X:X\times X\rightarrow \R_+$ on $X$ (cf. Definition \ref{def:ultrametric space}) defined by 
	\begin{equation}\label{eq:ultra}
	    	u_X(x,x'):=\inf\{\eps\geq 0: \mbox{$x,x'$ belong to the same block of $\theta_X(\eps)$}\}
	\end{equation}

	\begin{proposition}\label{prop:ultra} Given any two dendrograms $\theta_X,\theta_Y$ over sets $X,Y$, respectively, let $u_X,u_Y$ be the canonical ultra-pseudometrics on $X$ and $Y$, respectively. Then,  $\dint^\mathrm{F}(\theta_X,\theta_Y)=2\ d_{\mathrm{GH}}((X,u_X), (Y,u_Y)).$
	\end{proposition}

	\begin{proof} We first show that the LHS $\geq$ the RHS. 
	Let $\eps\geq 0$ and let $\tripod$ be any $\eps$-tripod between the two dendrograms $\theta_X$ and $\theta_Y$.	Let $(x,y),(x',y')\in R$ and let $t:=u_X(x,x')$. This implies that $x,x'$ belong to the same block of the partition $\theta_X(t).$ Since $\theta_X(t)\leq_R \bigvee_{[t]^\eps}\theta_Y=\theta_Y(t+\eps)$,  $y$ and $y'$ must belong to the same block of $\theta_Y(t+\eps)$, and in turn this implies that $u_Y(y,y')\leq t+\eps=u_X(x,x')+\eps$. By symmetry, we also have $u_Y(y,y')\leq u_X(x,x')+\eps$ and in turn $\abs{u_X(x,x')-u_Y(y,y')}\leq \eps$. By Definition \ref{def:the GH}, this implies that $\dgh((X,u_X),(Y,u_Y))\leq \eps/2.$

	Next, we prove the opposite inequality. Let $\tripod$ be a tripod between $X$ and $Y$ such that $\mathrm{dis}(R)=\eps.$ it suffices to show that for all $t\in\R$, $\theta_X(t)\leq_R\theta_Y(t+\eps)$ and $\theta_Y(t)\leq_{R}\theta_X(t+\eps)$. By symmetry, we only prove that  $\theta_X(t)\leq_R\theta_Y(t+\eps)$ for all $t\in \T$.  For $t<0$, since $\theta_X(t)=\emptyset$,  $\theta_X(t)\leq_R\theta_Y(t+\eps)$ trivially holds. Now pick any $t\geq 0$ and pick any $(x,y),(x',y')\in R$.  Assume that $x,x'$ belong to the same block of $\theta_X(t),$ implying that $u_X(x,x')\leq t.$ Since $\abs{u_X(x,x')-u_Y(y,y')}\leq\eps$, we know $u_Y(y,y')\leq t+\eps,$ and hence $y,y'$ belong to the same block of $\theta_Y(t+\eps)$. Therefore, $\theta_X(t)\leq_R\theta_Y(t+\eps)$ for all $t\in\R$. \end{proof}

	\begin{theorem}[Complexity of computing $\dintf$]\label{thm:complexity-dIF} Fix $\rho\in(1,6)$. It is not possible to obtain a $\rho$ approximation to the distance $\dintf(\theta_X,\theta_Y)$ between formigrams in time polynomially depending on $|X|,|Y|,|\mathrm{crit}(\theta_X)|$, $|\mathrm{crit}(\theta_Y)|$ unless $P=NP$. 
	\end{theorem}

	\begin{proof}
		Pick any two dendrograms and invoke Proposition \ref{prop:ultra} to reduce the problem to the computation of the Gromov-Hausdorff distance between the ultra-pseudometric spaces associated to the dendrograms. The rest of the proof follows along the same lines as that of Theorem \ref{thm:DG metric-complexity}.
 \end{proof}
	
\subsection{Proof of Theorem \ref{thm:persistence clustergram of a saturated formigram}} \label{sec:persistence clustergram of a saturated formigram}
	Theorem \ref{thm:persistence clustergram of a saturated formigram} will directly follow from Theorem \ref{thm:zigzag barcode via Mobius inversion of full component function} below. 
	
 We explicitly represent the colimit of $M:\ZZ\rightarrow \sets$ as follows. For $k,l\in \ZZ$, assume that $x\in M{(k)}$ and $y\in M{(l)}$. We write $x\sim y$ if $k$ and $l$ are comparable and one of $x$ and $y$ is mapped to the other via the internal map between $M(k)$ and $M(l)$.  The colimit of $M$ is the pair $\left(C,(i_k)_{k\in\ZZ}\right)$ described as follows:
    \begin{equation}\label{eq:set colimit construction}
        C:=\left(\coprod_{k\in\ZZ}M(k)\right)\big/\approx,
    \end{equation}
    where $\approx$ is the equivalence relation generated by the relations $x_k\sim x_{l}$ for $x_k\in M(k)$ and $x_{l}\in M(l)$ with $k,l$ being comparable. For the quotient map $q:\coprod_{k\in\ZZ}M(k)\rightarrow C$, each $i_k$ is the composition $M_k\hookrightarrow \coprod_{k\in\ZZ}M(k) \stackrel{q}{\rightarrow} C$. \label{item:set-colimit construction}

Let $I\in \Int(\ZZ)$. For any functor $N:I\rightarrow \sets$, we can construct the limit and colimit of $N$ in the same way; namely, in the above description, replace $M$ and $\ZZ$ by $N$ and $I$, respectively. \textbf{In what follows, we use this explicit construction  whenever considering colimits of (interval restrictions of) $\ZZ$-indexed $\sets$-diagrams.} 

\begin{definition}\label{def:full component} Let $I\in \Int(\ZZ)$ and let $N:I\rightarrow \sets$ by any functor. Let $c\in \varinjlim N$. We define the \emph{support} of $c$ as
\[\supp(c):=\{k\in I:  \exists x_k\in N_k, \ i_k(x_k)=c\}.\]
In particular, if $\supp(c)=I$, we call $c$ a \emph{full component} of the functor $N$. 
\end{definition}

Given $M:\ZZ\rightarrow \sets$ and $I\in \Int(\ZZ)$, we denote the number of \emph{full components} of $M|_I$ by $\full(M|_I)$. 	Recall Notation \ref{not:extended intervals of zz}.

	\begin{theorem}[{\cite[Corollary 4.10]{kim2018generalized}}]\label{thm:zigzag barcode via Mobius inversion of full component function}
	For any functor $M:\ZZ\rightarrow \sets$, the multiplicity of $I$ in $\B(\free\circ M)$ is\[\full(M|_I)-\full(M|_{I^+})-\full(M|_{I^-})+\full(M|_{I^\pm}).\]
	\end{theorem}

\begin{proof}[Proof of Theorem \ref{thm:persistence clustergram of a saturated formigram}]
By Proposition \ref{prop:silhouette commutes with mobius}, for every $I\in \Int(\ZZ)$, \[\abs{\dgm^{\veewedge}(\theta_X)}(I)=\abs{\veewedge_I\theta_X}-\abs{\veewedge_{I^+}\theta_X}-\abs{\veewedge_{I^-}\theta_X}+\abs{\veewedge_{I^{\pm}}\theta_X}.\] 
Therefore, by Theorem \ref{thm:zigzag barcode via Mobius inversion of full component function}, it suffices to show that $\abs{\veewedge_J\theta_X}=\full(\reeb(\theta_X)|_J)$ for all $J\in \Int(\ZZ)$. If $J\in \Int(\ZZ)$ is not a subset of $\supp(\theta_X)$, then clearly $0=\abs{\veewedge_J\theta_X}=\full(\reeb(\theta_X)|_J)$. Now assume that $J\in\Int(\ZZ)$ is contained in $\supp(\theta_X)$. Then, since $\theta_X$ is saturated, $\veewedge_J \theta_X=\bigvee_J \theta_X$. Also, $\abs{\bigvee_J\theta_X}$ is equal to the number of full components of $\reeb(\theta_X)|_J$, completing the proof.  
\end{proof}


\subsection{From unlabeled formigrams to persistent cluster counting functors}\label{sec:from weigted reeb to persistent cluster counting}

\ok{Let $\theta_X$ be a formigram. We shall prove that the persistent counting functor $\abs{\veewedge \theta_X}$ (cf. Definition \ref{def:persistent cluster counting functor}) can be obtained from the unlabeled formigram of $\theta_X$ (cf. Definition \ref{def:underlying} \ref{item:unlabeled formigram}). 

\begin{proposition}\label{prop:from weigted reeb to persistent cluster counting} Let $\theta$ be the unlabeled formigram of $\theta_X$. For any $I\in \Int(\ZZ)$, consider the canonical limit-to-colimit morphism $\varphi_I:\varprojlim \theta|_I\rightarrow \varinjlim \theta|_I$ in the category $\Part$. Then, $\coim(\varphi_I)\cong (\bigcap_{t\in I} \theta_X(t), \veewedge_I \theta_X)$.
\end{proposition}
\begin{proof} By Proposition \ref{prop:unlabeling preserves limits and colimits}, $\varprojlim \theta|_I\cong \bigwedge_I\theta_X$ and  $\varinjlim \theta|_I\cong \bigvee_I\theta_X$, and the morphism $\varphi_I$ in $\Part$ is the inclusion $\bigcap_{t\in I}\bigcup \theta_X(t)\hookrightarrow \bigcup_{t\in I}\bigcup\theta_X(t)$. Now by Proposition \ref{prop:image and coimage} \ref{item:coimage} the desired isomorphism follows.
\end{proof}
Proposition \ref{prop:from weigted reeb to persistent cluster counting} implies that we can extract $\abs{\veewedge \theta_X}$ from $\theta$: Namely,  $\abs{\veewedge \theta_X}(I)=\abs{\veewedge_I \theta_X}$ equals the number of blocks in the second entry of $\coim(\varphi_I)$. 
Reciprocally, one may wonder whether $\abs{\veewedge \theta_X}$ contains enough information to reconstruct $\theta$. That is not true; there exists a pair of formigrams which have identical persistent cluster functor, whereas their underlying weighted/unweighted Reeb graphs are different. This implies that their unlabeled formigrams are also different.
}

	\subsection{Details from Section \ref{sec:DMSs}}\label{sec:details on DMSs}
	
	\subsubsection{Details from Section \ref{sec:from DMSs to DGs}}
	
	\begin{proof}[Proof of Proposition \ref{prop:from dms to filt2}]\label{proof:from dms to filt2}

 	Clearly, $\rips^1(\gamma_X)$ is a function $\R\rightarrow \graph(X)$.  
 	We show that $\rips^1(\gamma_X)$ is cosheaf-inducing (Definition \ref{def:cosheaf conditions}). First we prove that locally $\rips^1(\gamma_X)$ admits only finitely many points of discontinuity (those points are called critical points). Let $I\subset\T$ be any nonempty finite interval.
		For $i,j\in X:=\{1,\ldots,n\}$, let $f_{i,j}:=d_X(\cdot)(i,j):\T\rightarrow\R_+$. Note that  discontinuity points of $\rips^1(\gamma_X)$ can  occur only at  endpoints of  connected components of the set ${f_{i,j}}^{-1}(\delta)$ for some $i,j\in X$. Fix any $i,j\in X$. Then, by Definition \ref{def:tameDMS}, the set ${f_{i,j}}^{-1}(\delta)\cap I$ has only finitely many connected components and thus there are only finitely many endpoints arising from those components. Since the set $X$ is finite, this implies that $\rips^1(\gamma_X)$ can have only finitely many points of discontinuity in $I$.   
 Fix any point $c\in \T$ on which $\rips^1(\gamma_X)$ is discontinuous. Consider the following two subsets of $X\times X$: $$A(c,\delta):=\{(i,j):\ i<j\in X,\ d_X(c)(i,j)\leq \delta \},$$ $$B(c,\delta):=\{(i,j):\ i<j\in X,\ d_X(c)(i,j)>\delta\}.$$
		The continuity of $d_X(\cdot)(i,j)$ for each $(i,j)\in X\times X$ guarantees that there exists $\eps>0$ such that $$B(t,\delta)\supset B(c,\delta)\hspace{3mm} \mbox{for all}\ t\in (c-\eps,c+\eps)$$ and in turn $$A(t,\delta)\subset A(c,\delta)\hspace{3mm} \mbox{for all}\ t\in (c-\eps,c+\eps)$$since $A(t,\delta)\cup B(t,\delta)=\{(i,j):i<j\in X\}$ for all $t\in \T.$
		This implies that the graph $\rips^1(\gamma_X(c))$ contains $\rips^1(\gamma_X(t))$ as a subgraph for each $t\in (c-\eps,c+\eps)$.
		
	\end{proof}
	
	\subsubsection{Details from Section \ref{sec:distance between DMSs}}\label{sec:details on distance between DMSs}

	\medskip\noindent\textbf{Details about $\dintl$.} We investigate further properties of the metrics in the family $\left\{\dintl\right\}_{\lambda\in[0,\infty)}$. In particular, a discussion about stable invariants of DMSs with respect to the metrics $\dintl$ for $\lambda>0$ can be found in \cite{kim2020spatiotemporal}.

	\begin{remark}\label{rem:bounded} Let $\lambda>0$. The distance $\dintl$ between any two bounded DMSs is finite. More specifically, for any $r$-bounded DMSs $\gamma_X=(X,d_X(\cdot))$ and $\gamma_Y=(Y,d_Y(\cdot))$ for some $r>0$, any tripod $R$ between $X$ and $Y$ is a $(\lambda, \frac{r}{\lambda})$-tripod between $\gamma_X$ and $\gamma_Y$. This implies that  $$\dintl(\gamma_X,\gamma_Y)\leq \frac{r}{\lambda}.$$\label{item:bounded2}	
	\end{remark}

	\begin{definition}[Equivalent tripods]\label{def:equivalent tripods} Let $X,Y$ be any two sets. For any two tripods $\tripod$ and $S:\ X \xtwoheadleftarrow{\psi_X} Z \xtwoheadrightarrow{\psi_Y} Y$ between $X$ and $Y$, we say that $R$ and $S$ are \emph{equivalent} if $(x,y)\in R$ if and only if $(x,y)\in S$.
	\end{definition}

	\begin{remark}\label{rem:equivalent tripod}Let $\gammax$ and $\gammay$ be any two DMSs. Suppose that $R$ and $S$ are equivalent tripods between $X$ and $Y$ (Definition \ref{def:equivalent tripods}). Then, it is not difficult to check that for any $\lambda,\eps\geq 0$, $R$ is a $(\lambda,\eps)$-tripod between $\gamma_X$ and $\gamma_Y$ if and only if $S$ is a $(\lambda,\eps)$-tripod  between $\gamma_X$ and $\gamma_Y$.
	\end{remark}

	\begin{proof}[Proof of Theorem \ref{thm:lambda metric}]
	We prove the triangle inequality. 
		Take any DMSs $\gamma_X,\gamma_Y$ and $\gamma_W$ over $X$,$Y$ and $W$, respectively. For some $\eps,\eps'>0$, let $\tripodone$ and $\tripodtwo$ be any $(\lambda,\eps)$-tripod between $\gamma_X$ and $\gamma_Y$ and $(\lambda,\eps')$-tripod between $\gamma_Y$ and $\gamma_W$  (Definition \ref{def:distortion}), respectively. Consider the set $Z:=\left\{(z_1,z_2)\in Z_1\times Z_2:\varphi_Y(z_1)=\psi_Y(z_2)\right\}$ and let $\pi_1:Z\rightarrow Z_1$ and $\pi_2:Z\rightarrow Z_2$ be the canonical projections to the first and the second coordinate, respectively. Define the tripod $R_2\circ R_1$ between $X$ and $W$ as in equation (\ref{eq:comosition}). It is not difficult to check that $R_2\circ R_1$ is a $(\lambda,\eps+\eps')$-tripod between $\gamma_X$ and $\gamma_W$ and thus we have $\dintl(\gamma_X,\gamma_W)\leq \dintl(\gamma_X,\gamma_Y)+ \dintl(\gamma_Y,\gamma_W).$ 
		
		Next assume that $\dintl(\gamma_X,\gamma_Y)=0$. We outline the proof of the fact that $\gamma_X$ and $\gamma_Y$ are isomorphic  (Definition \ref{def:isomorphism}). Because there are only finitely many tripods between $X$ and $Y$ up to equivalence (Definition \ref{def:equivalent tripods}), $\dintl(\gamma_X,\gamma_Y)=0$ implies that there must be a certain tripod $\tripod$ between $X$ and $Y$ such that $R$ becomes an $(\lambda,\eps)$-tripod between $\gamma_X$ and $\gamma_Y$ for \emph{any} $\eps>0$.  In order to show that $\gamma_X$ and $\gamma_Y$ are isomorphic, one needs to prove that that $R$ is in fact $(\lambda,0)$-tripod. After that, invoke Definition \ref{dynamic}, \ref{item:dynamic2} and \ref{item:dynamic3} to verify that the multivalued map $\varphi_Y\circ\varphi_X^{-1}:X\rightrightarrows Y$ is in fact a bijection from $X$ to $Y$. 
		
		\ok{Lastly}, by Remark \ref{rem:bounded}, for $\lambda>0$, $\dintl$ is finite between bounded DMSs. 
	\end{proof}
	
	\begin{remark}[For $\lambda>0$, $\dintl$ generalizes the Gromov-Hausdorff distance]\label{rem:gromov} Let $\lambda>0 $. Given any two constant DMSs $\gamma_X\equiv (X,d_X)$ and $\gamma_Y\equiv (Y,d_Y)$, the value $\dintl(\gamma_X,\gamma_Y)$ equals the Gromov-Hausdorff distance between $(X,d_X)$ and $(Y,d_Y)$ up to multiplicative constant $\frac{\lambda}{2}$. Indeed, for any tripod $\tripod$ between $X$ and $Y$, condition  (\ref{eq:distor}) reduces to 
		$$\abs{d_X\left(\varphi_X(z),\varphi_X(z')\right)-d_Y\left(\varphi_Y(z),\varphi_Y(z')\right)}\leq \lambda\eps \ \ \mbox{for all $z,z'\in Z$}.$$ Therefore, $$\dgh((X,d_X),(Y,d_Y))=\frac{\lambda}{2}\cdot\dintl(\gamma_X,\gamma_Y).$$	
	\end{remark}

	We have the following bilipschitz-equivalence relation between the metrics $\dintl$ for different $\lambda>0$.
	
	\begin{proposition}[Bilipschitz-equivalence]\label{prop:equivalence} For all $0<\lambda<\lambda'$,
		$$d_{\mathrm{I},\lambda'}^\mathrm{dyn}\leq \dintl\leq \frac{\lambda'}{\lambda}\cdot d_{\mathrm{I},\lambda'}^\mathrm{dyn}.$$
	\end{proposition}

	\begin{proof}\label{proof:equivalence} Fix any two DMSs $\gamma_X$ and $\gamma_Y$ over $X$ and $Y$. That  $\dintll(\gamma_X,\gamma_Y)\leq \dintl(\gamma_X,\gamma_Y)$ follows from the observation that any $(\lambda,\eps)$-tripod $R$ between $\gamma_X$ and $\gamma_Y$ is also a $(\lambda',\eps)$-tripod (Definition \ref{def:distortion}).
		We next prove $\dintl(\gamma_X,\gamma_Y)\leq \frac{\lambda'}{\lambda}\cdot d_{\mathrm{I},\lambda'}^\mathrm{dyn}(\gamma_X,\gamma_Y).$ For some $\eps\geq 0$ let $R$ be any $(\lambda',\eps)$-tripod between $\gamma_X$ and $\gamma_Y$. It suffices to show that $R$ is also a $(\lambda,\frac{\lambda'}{\lambda}\eps)$-tripod.
		Fix any $t\in T.$ Then,
		$$\bigvee_{[t]^{\left(\frac{\lambda'}{\lambda}\eps\right)}}d_X\leq \bigvee_{[t]^{\eps}}d_X\leq_R d_Y(t)+\lambda'\eps=d_Y(t)+\lambda\left(\frac{\lambda'}{\lambda}\eps\right).$$
		By symmetry, we also have$\bigvee_{[t]^{\left(\frac{\lambda'}{\lambda}\eps\right)}}d_Y\leq_{R}d_X(t)+\lambda\left(\frac{\lambda'}{\lambda}\eps\right),$	as desired.
 \end{proof}
	
\section{Distance between weighted Reeb graphs}\label{sec:distance between weighted reeb graphs}

\ok{In this section we introduce a distance between weighted Reeb graphs which mediates between $\dintf$ and $\dintreeb$ (cf. Definition \ref{def:reeb graph as a cosheaf} and Table \ref{table:names of cosheaves}).}

\begin{proposition}[Realization as an unlabeled formigram] \label{prop:realization}
For any weighted Reeb graph $F:\Int\rightarrow \wsets$, there exists an unlabeled formigram $\theta:\Int\rightarrow \Partin$ such that $F\cong \A\circ \theta$.\label{item:realization 1}
    \label{item:realization 2}
\end{proposition}

 \ok{The proof is rather trivial and thus we omit it. In general, a realization of a weighted Reeb graph as an unlabeled formigram is not unique; see Example \ref{ex:same underlying graph}.} \ok{ Proposition \ref{prop:realization} allows us to define the following dissimilarity measure on weighted Reeb graphs. From equation (\ref{eq:distance between unlabeled formigrams}), recall how to define $\dintf$ between unlabeled formigrams. For all weighted Reeb graphs $F,G:\Int\rightarrow \wsets$, we define:}
\begin{equation}\label{eq:W}
    W(F,G):=\inf\{\eps\in[0,\infty]: \mbox{there exist $\theta,\theta':\Int\rightarrow \Part$ s.t.  $F\cong \A\circ\theta$, and  $G\cong \A\circ \theta',\ \dintf(\theta,\theta')=\eps$}\}. 
\end{equation}

\ok{Since a realization of a weighted Reeb graph as an unlabeled formigram is not necessarily unique, we have to possibly take into account multiple realizations of $F$ and $G$ to compute $W(F,G)$.} \ok{This leads to the fact that $W$ does not necessarily satisfy the triangle inequality and thus we consider the \emph{maximal sub-dominant metric} of $W$ \cite{carlsson2013classifying} as a metric on weighted Reeb graphs: }

\begin{definition}[Metric on weighted Reeb graphs]\label{def:distance between weighted Reeb graphs} For any two weighted Reeb graphs $F,G:\Int\rightarrow \wsets$,
\[\dintwreeb(F,G):=\inf \left\{\sum_{i=0}^{m-1}W(F_i,F_{i+1}):\ F=F_1,\ldots,F_m=G \ \mbox{is a sequence in $\Int^{\wsets}$}\right\}.\]
\end{definition}

 $\dintwreeb$ is the greatest metric on weighted Reeb graphs among those upper bounded by $W$.  $\dintwreeb$ mediates between $\dintf$ and $\dintreeb$:

\begin{theorem}\label{thm:dintwreeb mediates} 
For any two formigrams $\theta_X$ and $\theta_Y$, let $\omega(\theta_X)$ and $\omega(\theta_Y)$ be their weighted Reeb graphs. Then,
\begin{equation}\label{eq:mediates}
    \dintreeb(\reeb(\theta_X),\reeb(\theta_Y))\leq \dintwreeb(\omega(\theta_X),\omega(\theta_Y)) \leq \dintf(\theta_X,\theta_Y).
\end{equation}
\end{theorem}

\begin{proof}
From the definition of $\dintreeb$ and Proposition \ref{prop:formigram interleaving is mor discriminative than the reeb interleaving}, we have that \[\dintreeb(\reeb(\theta_X),\reeb(\theta_Y))\leq W(\omega(\theta_X),\omega(\theta_Y)).\] Since $\dintwreeb$ is the greatest metric on weighted Reeb graphs among those upper bounded by $W$, the left inequality in (\ref{eq:mediates}) follows.  The right inequality in (\ref{eq:mediates}) directly follows from the definition of $\dintwreeb$.
\end{proof}

$\dintwreeb$ is strictly less discriminative than $\dintf$ whereas strictly more discriminative than $\dintreeb$: 
\begin{example}\label{ex:dintwreeb} 
\begin{enumerate}[label=(\roman*)]
    \item Consider $\theta_X$ and $\theta_Y$ in Example \ref{ex:same underlying graph}. Since the underlying weighted Reeb graphs of $\theta_X$ and $\theta_Y$ are isomorphic, their distance in $\dintwreeb$ is zero.  However, by Remark \ref{rem:df generalizes dgh} \ref{item:df generalizes dgh1}, we have that $\dintf(\unlabel\circ\theta_X,\unlabelY\circ\theta_Y)=\dintf(\theta_X,\theta_Y)>0$.\label{item:dintwreeb1}
    \item Let $F$ and $G$ be weighted Reeb graphs depicted in Figure \ref{fig:two weighted reebs}. Their unweighted Reeb graphs $\A\circ F$ and $\A\circ G$ are clearly isomorphic and thus $\dintreeb(\A\circ F,\A\circ G)=0$. On the other hand, $\dintwreeb(F,G)= 1/2;$ this follows from the observation that both $F$ and $G$ are uniquely realized (up to natural isomorphism) by unlabeled formigrams $\theta$ and $\theta'$, which leads to $\dintwreeb(F,G)=\dintf(\theta,\theta')$. Also, it is not difficult to check that $\dintf(\theta,\theta')=1/2$. \label{item:dintwreeb2}      
\end{enumerate} 
\end{example}

\begin{figure}
    \centering
    \includegraphics[width=0.45\textwidth]{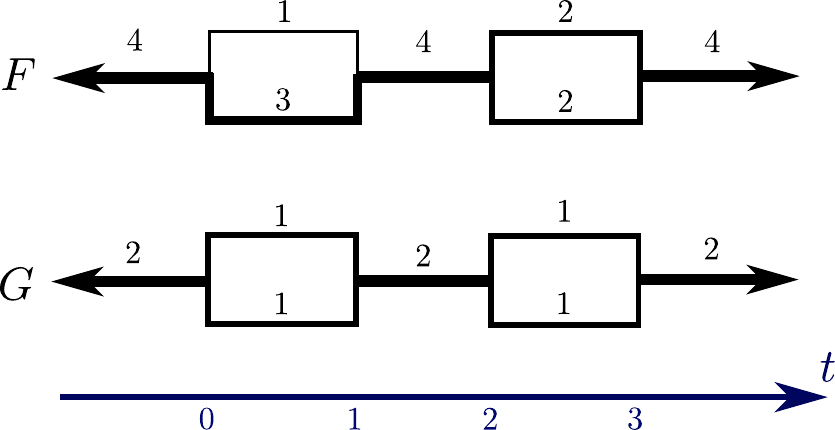}
    \caption{Two weighted Reeb graphs in Example \ref{ex:dintwreeb} \ref{item:dintwreeb2}.}
    \label{fig:two weighted reebs}
\end{figure}

	\section{Smoothing formigrams}\label{sec:smoothing}
\ok{The goal of this section is to establish a few basic properties of smoothing of formigrams. In particular, we reveal its effect on the zigzag barcodes of formigrams and its compatibility with smoothing of Reeb graphs in \cite{de2016categorified}; see Propositions \ref{prop:barcode} and \ref{prop:compatibility}.}

Recall that a formigram $\theta_X$ will be regarded as either a \cosheaf{} function $\R\rightarrow\subpart(X)$ or a constructible cosheaf $\Int\rightarrow \subpart(X)$ (cf. Definition \ref{def:formigram}, Remark \ref{rem:join-semilattice cosheaf} \ref{item:join-semilattice cosheaf1} and \ref{item:join-semilattice cosheaf2}). 
 By Definition \ref{def:cosheaf smoothing} and Proposition \ref{prop:smooting costructible yields constructible},  a smoothing operation on formigrams can be induced  via the join operation on subpartitions.
 Namely, $S_\eps\theta_X$ sends each $I\in\Int$ to $\bigvee_{I^\eps}\theta_X:=\bigvee\{\theta_X(t):t\in I^\eps\}$.

 \begin{remark}[Comparison with robust grouping structure]\label{rem:difference from Buchin's smoothing}
 The use of the join operation is an important element that distinguishes our notion of smoothing from the \emph{robust grouping structure} in \cite{buchin2013trajectory}. In particular, given a dynamic metric space (DMS), the induced formigram of this DMS (which is obtained by combining Definition \ref{def:from DG to formi} and Proposition \ref{prop:from dms to filt2}) can be smoothed out using the join operation. We emphasize that this smoothing operation is \emph{intrinsic} in contrast to the robust grouping structure from \cite{buchin2013trajectory}. Namely, our smoothing operation can be carried out  \emph{without} constructing any topological space in the spatiotemporal ambient space of the DMS as illustrated in \cite[Figure 11]{buchin2013trajectory}. 
 One consequence of this `intrinsicality' is that, when a dynamic \emph{graph} is the input data (as opposed to a dynamic \emph{metric space}), we can smooth out its induced formigram (cf. Definition \ref{def:from DG to formi}), while \cite{buchin2013trajectory} does not propose such a method. Since the coordinates of entities are not always available in applications (e.g. sensor networks \cite{de2006coordinate,de2007homological}, low-cost swarm robots \cite{rubenstein2012kilobot}, etc.),  \ok{this intrinsicality is a desirable property.}

 \end{remark}

Given a formigram $\theta_X$, Figure \ref{fig:smoothing} illustrates both the relationship between $\reeb(\theta_X)$ and $\reeb(S_\eps\theta_X)$ and the relationship between their zigzag barcodes. The following proposition precisely describes the relationship between $\barc(\theta_X)$ and $\barc(S_\eps\theta_X)$. 
For any $r\in \R$, we define $-\infty+r$ to be $-\infty$.

	\begin{figure}
		\begin{center}
			\includegraphics[width=\textwidth]{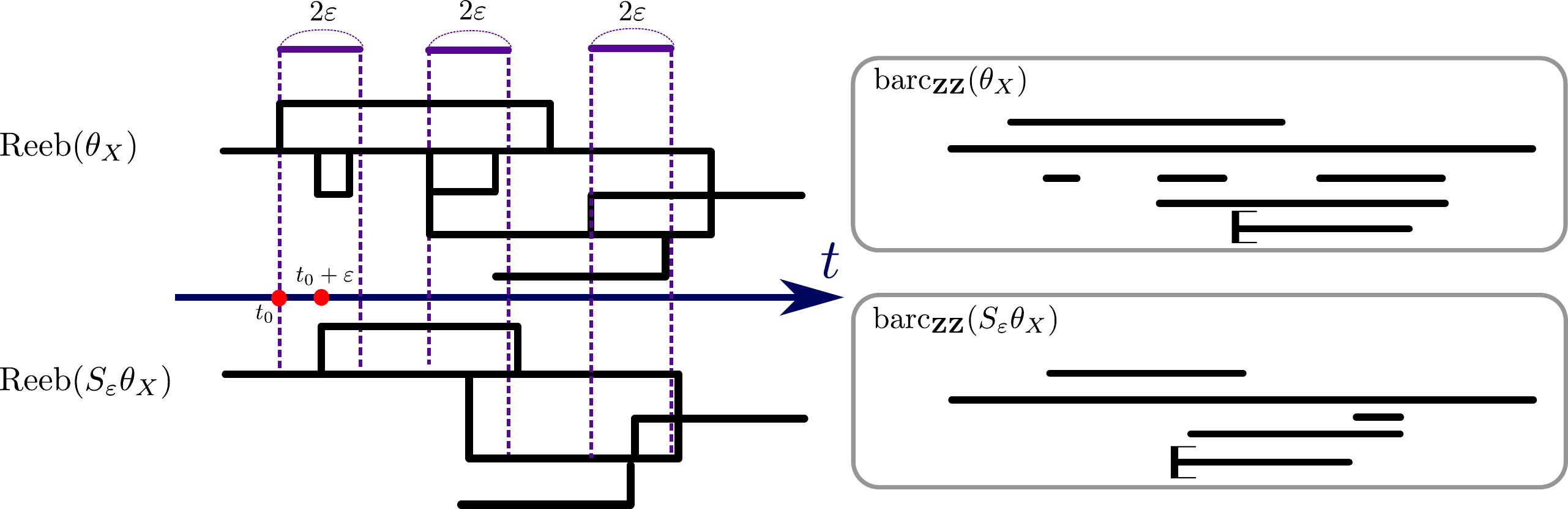}
		\end{center}
		\caption{\label{fig:smoothing}An illustration for Proposition \ref{prop:barcode}. \textbf{Top}: The Reeb graph of a formigram $\theta_X$ and its barcode. 
			\textbf{Bottom}: The Reeb graph of the formigram $S_\eps\theta_X$ and its barcode. Small loops in $\reeb(\theta_X)$ disappear in $\reeb(S_\eps\theta_X)$. In the barcodes, bars with ``$[$" on the left stand for half-closed intervals of the form $[a,b)$. Open intervals in $\barc(\theta_X)$ that are shorter than $2\eps$ do not have corresponding intervals in $\barc(S_\eps\theta_X)$.     
			Also, disbanding and merging events in $\theta_X$ \emph{which do not correspond to vertices on small loops in $\reeb(\theta_X)$} are replicated in $S_\eps\theta_X$: disbanding events in $\theta_X$ are reflected in $S_\eps\theta_X$ but with delay  $\eps$, whereas  merging events in $\theta_X$ are advanced by $\eps$. For example, observe from the graphs $\reeb(\theta_X)$ and $\reeb(S_\eps\theta_X)$ that the disbanding event in $\theta_X$ at $t=t_0$ is delayed to $t=t_0+\eps$ in $S_\eps\theta_X$.} 
	\end{figure}

	\begin{proposition}\label{prop:barcode} Let $\theta_X$ be a formigram over $X$ and let $\eps\geq 0$. Then, we have the following bijection between $\barc(\theta_X)$ and $\barc(S_\eps\theta_X)$ (cf. Figure \ref{fig:smoothing}):
		\begin{equation}\label{eq:effect}
		\begin{array}{cccl}	\barc(\theta_X)&&\barc(S_\eps\theta_X)&\\
		\cline{1-3}
		
		{(b,d)}&\leftrightarrow&(b+\eps,d-\eps)&\mbox{for $-\infty\leq b\leq b+\eps< d-\eps\leq +\infty$}\\
		
		{(b,d)}&\leftrightarrow&\mathrm{Nothing}&\mbox{for $b< d<b+2\eps$}.
		\\	
		{[b,d)}&\leftrightarrow&[b-\eps,d-\eps)&
		\\	
		{(b,d]}&\leftrightarrow&(b+\eps,d+\eps]&
		\\	
		
		{[b,d]}&\leftrightarrow&[b-\eps,d+\eps]&
		\end{array}
		\end{equation}
		Recall the free functor $\free:\sets\rightarrow \vect$ (cf. Definition \ref{def:free}) and the fact that any constructible cosheaf $\Int\rightarrow \vect$ is interval decomposable (cf. Proposition \ref{prop:interval decomposability of constructible cosheaf}).
\begin{lemma}\label{lem:no weird block}
        Given any constructible cosheaf $F:\Int\rightarrow \sets$, the barcode of $\free\circ F:\Int\rightarrow \vect$ cannot include any interval of the form $[b,a]_{\mathrm{BL}}:=\{(x,y)\in \U: x\leq a <b\leq y\}$ for $a<b$ in $\R$ (cf. Figure \ref{fig:block}).
\end{lemma}

\begin{proof}
Since $F$ is constructible is defined via colimits over restrictions of a zigzag diagram over $\R$ (cf. Definition \ref{def:constructible cosheaf}), for any $J\in \Int$ and for any $x\in F(J)$, there exist $t\in J$ and $y\in F([t,t])$ such that $F([t,t]\subset J)(y)=x$. This property directly implies that the interval module $I^{[b,a]_{\mathrm{BL}}}$ cannot be a summand of $\free\circ F$. 
\end{proof}

\begin{proof}[Proof of Proposition \ref{prop:barcode}] By Definitions \ref{def:zigzag barcode} and \ref{def:cosheaf smoothing}, $\barc(S_\eps\theta_X)$ is equal to the multiset
\[\lmulti B\cap \R_{y=x+2\eps}: B\in \mathrm{barc}(\free\circ F)\rmulti. \]
where $\R_{y=x+2\eps}$ is the line $y=x+2\eps$ identified with the real line via the bijection $(r-\eps,r+\eps)\leftrightarrow r$. The table (\ref{eq:effect}) is directly obtained by Lemma \ref{lem:no weird block} and  the \emph{block decomposability} of $\free\circ \emb \circ \theta_X$ \cite[Section 3]{botnan2018algebraic}.
\end{proof}

\begin{figure}
    \centering
    	\includegraphics[width=0.25\textwidth]{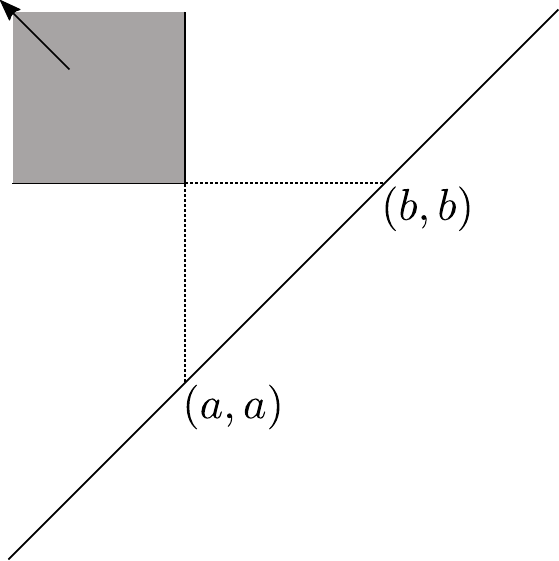}
    \caption{An illustration of $[b,a]_{\ZZ}$ for $a<b$ in $\Z$.}
    \label{fig:block}
\end{figure}

	\end{proposition}

	The bijective correspondence of barcodes given in Proposition \ref{prop:barcode} directly implies the following: 
	
	\begin{corollary}\label{cor:bott-eps} Let $\theta_X$ be any formigram over $X$. Then, for $\eps\geq 0$, $$\bott\left(\barc\left(S_\eps\theta_X\right),\ \barc\left(\theta_X\right)\right)\leq \eps.$$
	\end{corollary}
	
The smoothing operations defined for formigrams and Reeb graphs  (cf. Definition \ref{def:cosheaf smoothing}) are compatible in the following sense: 
	
	\begin{proposition}\label{prop:compatibility} 
	Let $\theta_X$ be a formigram over $X$. Then, for any  $\eps\geq 0$, $$\reeb(S_\eps\theta_X)=S_\eps\reeb(\theta_X).$$ 
	\end{proposition} 
	
	\begin{proof}
	Let $I\in \Int$. We have:
	\begin{align*}
	    	\reeb(S_\eps\theta_X)(I)&=\left(\emb\circ (S_\eps\theta_X) \right)(I)&\mbox{by Definitions \ref{def:three functors} and \ref{def:forget from subpart to set}}\\&=\emb\circ \left(S_\eps\theta_X\right) (I)&\mbox{by Proposition \ref{prop:colimit}}\\&=\emb\circ \theta_X (I^\eps)&\mbox{by Definition \ref{def:cosheaf smoothing}}\\&= \left(\emb\circ \theta_X \right) (I^\eps)&\mbox{by Proposition \ref{prop:colimit}}
	    	\\&=S_\eps \reeb(\theta_X)(I).
	\end{align*}
	\end{proof}
	
Formigrams change in a continuous manner under $\eps$-smoothing:	
	
	\begin{proposition}\label{prop:continuity} For any $\eps\geq 0$ and any formigram $\theta_X$,
			$$\dintf\left(S_\eps\theta_X,\theta_X\right)\leq \eps.$$ 
	\end{proposition}
	
	\begin{proof} Consider the tripod $R: X\xtwoheadleftarrow[]{\mathrm{id}_X}X \xtwoheadrightarrow[]{\mathrm{id}_X} X$ and check that $R$ is an $\eps$-tripod between $S_\eps\theta_X$ and $\theta_X$. \end{proof}
	
	The following proposition is analogous to \cite[Proposition 4.14]{de2016categorified}:
	
	\begin{proposition}\label{prop:contrations}For any $\eps\geq 0$,  $S_\eps$ is a contraction on formigrams, i.e. for any formigrams $\theta_X$ and $\theta_Y$
			$$\dintf\left(S_\eps\theta_X,S_\eps\theta_Y\right)\leq \dintf \left(\theta_X,\theta_Y\right).$$
	\end{proposition}
	
	\begin{proof}
	For $\delta\geq 0$, assume that $\tripod$ is a $\delta$-tripod between $\theta_X$ and $\theta_Y$. We claim that $R$ is also a $\delta$-tripod between $S_\eps\theta_X$ and $S_\eps\theta_Y$. First, we remark that $\varphi_X^\ast S_\eps\theta_X=S_\eps\varphi_X^\ast  \theta_X$. Indeed, for any $I\in \Int$, $(\varphi_X^\ast S_\eps\theta_X)(I)=\varphi_X^\ast (S_\eps\theta_Y(I))=\varphi_X^\ast (\theta_X(I^\eps))=(S_\eps\varphi_X^\ast\theta_X)(I)$. Therefore,
	\begin{align*}
	    S_\delta \left(\varphi_X^\ast S_\eps\theta_X\right)&= S_\delta \left(S_\eps\varphi_X^\ast  \theta_X\right)\\&= S_{\eps}\left(S_\delta \varphi_X^\ast\theta_X\right)&\mbox{by Remark \ref{rem:smoothing forms semigroup}}\\&\geq S_\eps \varphi_Y^\ast\theta_Y&\mbox{by the choice of $R$}\\&=\varphi_Y^\ast S_\eps\theta_Y
	\end{align*}
	and by symmetry we have $S_\delta \left(\varphi_Y^\ast S_\eps\theta_Y\right)\geq \varphi_X^\ast S_\eps\theta_X$, completing the proof.
	\end{proof}

\section{About the $0$-slack interleaving distance between DMSs}\label{sec:0-slack DMS intereleaving}

\ok{We clarify the computational complexity of $\dintm$ (cf. Theorem \ref{thm:complex}) and provide a few examples of computing $\dintm$.}
	
	\paragraph{Computational complexity of $\dintm$.} We relate the Gromov-Hausdorff distance between two given ultrametric spaces to the interleaving distance $\dintm$ between certain DMSs induced by those ultrametric spaces. Then, invoking results from F. Schmiedl's PhD thesis \cite{schmiedl2015shape,schmiedl2017computational} we obtain the claim of Theorem \ref{thm:complex}. 
	
	Given a  ultrametric space $(X,u_X)$, define a DMS $\mathcal{D}(X,u_X):=(X,d_X(\cdot))$ where for all $x,x'\in X$ and for all $t\in \T$, $ d_X(t)(x, x'):=\max(0,u_X(x,x')-t)$. It is noteworthy that for any $x,x'\in X$, $d_X(\cdot)(x,x'):\T\rightarrow \R_+$ is decreasing down to zero and that $d_X(0)=u_X$, a legitimate metric (i.e. not just pseudo-metric), satisfying the second item of Definition \ref{dynamic}. Furthermore, note that $\mathcal{D}(X,u_X)$ is clearly piecewise linear and that the set of breakpoints is $S_{\mathcal{D}(X,u_X)} = \{u_X(x,x'),\,x,x'\in X\}.$ 
	Recall Definition \ref{def:the GH}.

	\begin{proposition}\label{prop:ultra2} For any two ultrametric spaces $(X,u_X)$ and $(Y,u_Y)$ we have $$\dintm(\mathcal{D}(X,u_X),\mathcal{D}(X,u_Y))=2\ d_\mathrm{GH}((X,u_X),(Y,u_Y)).$$ 
	\end{proposition}
	
	\begin{proof}Let $\mathcal{D}(X,u_X)= (X,d_X(\cdot))$ and $\mathcal{D}(Y,u_Y)=(Y,d_Y(\cdot))$. Observe that for any $x,x'\in X$, any $t\in \T$, and any $\eps\geq 0$, $\min_{s\in[t]^\eps}d_X(s)(x,x') = d_X(t+\eps)(x,x')$ since $d_X$ is decreasing over time. Thus, for some $\eps\geq0$, a tripod $\tripod$ is an $\eps$-tripod between $(X,d_X), (Y,d_Y)$ (Definition \ref{def:the GH}) if and only if for all $z,z'\in Z$ and for all $t\in \T$,  $d_X(t+\eps)\left(\varphi_X(z),\varphi_X(z')\right) \leq d_Y(t)\left(\varphi_Y(z),\varphi_Y(z')\right)$ and $d_Y(t+\eps)\left(\varphi_Y(z),\varphi_Y(z')\right)\leq d_X(t)\left(\varphi_X(z),\varphi_X(z')\right)$, if and only if for all $z,z'\in Z$ and for all $t\in \T$, $\max\left(0,u_X\left(\varphi_X(z),\varphi_X(z')\right)-t-\eps\right)\leq \max\left(0,u_Y\left(\varphi_Y(z),\varphi_Y(z')\right)-t\right)$ and $\max\left(0,u_Y\left(\varphi_Y(z),\varphi_Y(z')\right)-t-\eps\right)\leq \max\left(0,u_X\left(\varphi_X(z),\varphi_X(z')\right)-t\right)$ if and only if for all $z,z'\in Z$, \[\abs{u_X\left(\varphi_X(z),\varphi_X(z')\right)-u_Y\left(\varphi_Y(z),\varphi_Y(z')\right)}\leq \eps,\] completing the proof. 
	\end{proof}
	\paragraph{Proof of Theorem \ref{thm:complex}}\label{proof:complex}Pick any two ultrametric spaces $(X,u_X)$ and $(Y,u_Y)$. Then, by Proposition \ref{prop:ultra2}, the interleaving distance between $\mathcal{D}(X,u_X)$ and $\mathcal{D}(Y,u_Y)$ is identical to twice the Gromov-Hausdorff distance $\Delta:=d_\mathrm{GH}((X,u_X),(Y,u_Y))$ between  $(X,u_X)$ and $(Y,u_Y)$. The rest of the proof follows along the same lines as that of Theorem \ref{thm:DG metric-complexity}.
	\qed 
	
Next we discuss a few computational examples of $\dintm$. Let $\psi:\T\rightarrow \R_+$ be any non identically zero continuous function. Then,  for any finite metric space $(X,d_X')$ we have the DMS $\gamma_X^\psi = (X,d_X^\psi(\cdot))$ where for $t\in\T$,  $d_X^\psi(t):=\psi(t)\cdot d_X'.$

\begin{figure}
	\includegraphics[width=\textwidth]{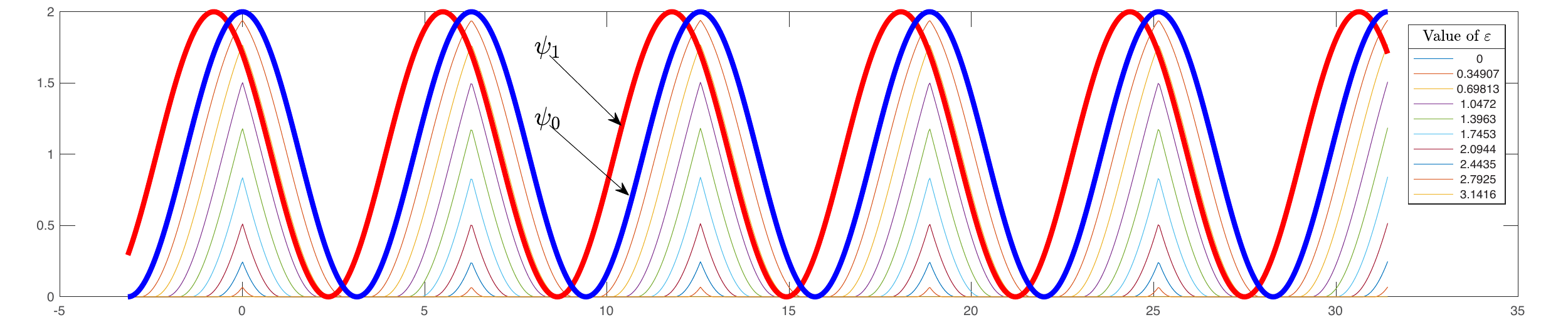}
	\caption{\textbf{The interleaving condition.} The thick blue curve and the thick red curve represent the graphs of $\psi_0(t)=1+\cos(t)$ and $\psi_1(t)=1+\cos(t+\pi/4),$ respectively. Fixing $\eps\geq 0$, define a function $S_\eps(\psi_0):\R\rightarrow \R$ by $S_\eps(\psi_0)(t):=\min_{s\in[t]^\eps}\psi_0(s)$.  The thin curves below the thick blue curve illustrate the graphs of $S_\eps(\psi_0)$ for several different choices of $\eps$.  Note that for $\eps\geq\pi/4\simeq 0.785$, it holds that $S_\eps(\psi_0)\leq \psi_1.$}\label{fig:cosine}
\end{figure}

\begin{example}[An interleaved pair of DMSs \uppercase\expandafter{\romannumeral1\relax}]\label{example0}This example refers to Figure \ref{fig:cosine}. Fix the two-point metric space $(X,d_X)=\left(\{x,x'\},\begin{psmallmatrix}0 & 1\\1 & 0\end{psmallmatrix}\right)$ and consider two DMSs $\gamma_X^{\psi_0}=(X,d_X^{\psi_0})$ and $\gamma_X^{\psi_1}=(X,d_X^{\psi_1})$ where, for $t\in \R$, $\psi_0(t)=1+\cos(t)$, $\psi_1(t)=1+\cos(t+\pi/4)$. Then, $\gamma_X^{\psi_0}$ and $\gamma_X^{\psi_1}$ are $\eps$-interleaved if and only if for $i,j\in\{0,1\}$, $i\neq j$, and for all $t\in \R$, $S_\eps(\psi_i)(t):=\min_{s\in[t]^\eps}\psi_i(s)=\left(\bigvee_{[t]^\eps}d_X^{\psi_i}\right)(x,x')\leq d_X^{\psi_j}(t)(x,x')=\psi_j(t)$. In fact, this inequality holds if and only if $\eps\geq\pi/4$, and hence $\dintm\left(\gamma_X^{\psi_0},\gamma_X^{\psi_1}\right)=\pi/4$ (cf. Figure \ref{fig:cosine}).
\end{example}

The following example generalizes the previous one.

\begin{example}[An interleaved pair of DMSs \uppercase\expandafter{\romannumeral2\relax}]\label{example}
	Fix the two-point metric space $(X,d_X)=\left(\{x,x'\},\begin{psmallmatrix}0 & 1\\1 & 0\end{psmallmatrix}\right)$ and consider two DMSs $\gamma_X^{\psi_0}=(X,d_X^{\psi_0})$ and $\gamma_X^{\psi_1}=(X,d_X^{\psi_1})$  where, for $t\in \R$, $\psi_0(t)=1+\cos(\omega t)$, $\psi_1(t)=1+\cos(\omega(t+\tau))$,  for fixed $\omega>0$ and $0<\tau<\frac{2\pi}{\omega}$. Since in this case $\psi_1(t) = \psi_0(t+\tau)$ for all $t$, one would expect that  the interleaving distance between $\gamma_X^{\psi_0}$ and $\gamma_X^{\psi_1}$ is able to uncover the precise the value of $\tau$. In this respect, we have:
	$
	\dintm(\gamma_X^{\psi_0},\gamma_X^{\psi_1})=\min \Big(\tau,\ \frac{2\pi}{\omega}-\tau\Big)=:\eta(\omega,\tau).
	$
	
\end{example}

	\section{Higher dimensional persistent homology barcodes of dynamic metric spaces.}\label{sec:higher dimensional barcodes}
	
	In this section we discuss extendibility of Theorem \ref{thm:main thm2}. The zigzag barcodes $\barc(\theta_X)$ and $\barc(\theta_Y)$ in Theorem \ref{thm:main thm2} encodes the clustering behaviors of the given DMSs for a fixed scale $\delta\geq 0$. 
	However, we do not need to restrict ourselves to clustering features of DMSs. Imagine that a flock of birds flies while keeping a circular arrangement from time $t=0$ to $t=1$. Regarding this flock as a DMS (trajectory data in $\R^3$), we may want to have an interval containing $[0,1]$ in its \emph{1-dimensional homology barcode}. This idea can actually be implemented as follows. 
	
	For a fixed $\delta\geq 0$, we substitute the Rips complex functor $\rips$ for the Rips graph functor $\rips^1$ in Proposition \ref{prop:from DMS to DG}. What we obtain is a \emph{dynamic simplicial complex or zigzag simplicial filtration}, a generalization of Definition \ref{def:dyn graphs}, induced from any tame DMS $\gamma_X$.  We then can apply the $k$-th homology functor to this zigzag simplicial filtration for each $k\geq 0$ in order to obtain a $\vect$-valued constructible cosheaf over $\R$. This zigzag module will be a signature summarizing the time evolution of $k$-dimensional homological features of $\gamma_X$. By virtue of Proposition \ref{prop:interval decomposability of constructible cosheaf} we eventually obtain the \emph{$k$-th homology barcode} $\barc\left(\Hrm_k\left(\rips(\gamma_X)\right) \right)$ of $\gamma_X$ with respect to the fixed scale $\delta\geq 0$; see also \cite{dey2021updating} for the computation of $\barc\left(\Hrm_k\left(\rips(\gamma_X)\right) \right)$ for various $\delta$. In particular, the $0$-th homology barcode of the resulting zigzag module coincides with $\barc\left(\pi_0\left(\rips^1(\gamma_X)\right)\right)$ as defined in Theorem \ref{thm:main thm2}. 
	
	A natural question is then to ask whether  our stability theorem (Theorem \ref{thm:main thm2}) can be extended to higher dimensional homology barcodes:
	
	\begin{question}For any pair of tame DMSs $\gammax$ and $\gammay$, is it true that for any $\delta\geq 0$ and for any $k\geq 1$,
		$$\bott\left(\barc\left(\Hrm_k\left(\rips(\gamma_X)\right) \right), \barc\left(\Hrm_k\left(\rips(\gamma_Y)\right) \right)\right) \leq 2\ \dintm\left(\gamma_X,\gamma_Y\right)\ ?$$
	\end{question}
	
	Interestingly, we found a family of counter-examples that indicates that stability, as expressed by Theorem \ref{thm:main thm2}, is a phenomenon which seems to be essentially tied to clustering (i.e. $\Hrm_0$) information.
	
	\begin{theorem}\label{thm:counter-example} For each integer $k\geq 1$ there exist two different tame DMSs  $\gamma_{X_k}$ and $\gamma_{Y_k}$, and $\delta_k\geq 0$ such that $ \dintm\left(\gamma_{X_k},\gamma_{Y_k}\right)<\infty$ but such that the bottleneck distance between the barcodes of $\Hrm_k\left(\mathcal{R}_{\delta_k}\left(\gamma_{X_k}\right)\right)$ and $\Hrm_k\left(\mathcal{R}_{\delta_k}\left(\gamma_{Y_k}\right)\right)$ is unbounded. 
	\end{theorem}

	\begin{figure}
		\begin{center}
			\hspace{-2mm}\begin{tikzpicture}[scale=0.6]
			\draw (0,1.8) node {$\gamma_{X_1}$};\fill (0,1) circle (1mm);\fill (1,0) circle (1mm);\fill (-1,0) circle (1mm);\fill (0,-1) circle (1mm);
			\draw (-1,0)--(0,-1)--(1,0);
			\draw (-1,0)--(0,1)--(1,0);;
			
			\draw (1.2,0) node[anchor=north] {$e_1$};
			\draw (0.35,1) node {$e_2$};
			\draw (-1.3,0) node[anchor=north]{$-e_1$};
			\draw (0.3,-1) node[anchor=north] {$-e_2$}; 
			\end{tikzpicture}\hspace{-2mm}\begin{tikzpicture}[scale=0.6]
			\draw (-2,0) node {$\empty$};
			\draw (0,1.8) node {$\gamma_{X_1}'$}; 
			\draw[red,<->,>=stealth', line width=1pt] (1.2,0)--(2,0);			\draw (1,0) node[fill, red, star, inner sep=2pt,minimum size=2mm]{} ;
			\draw (0,1)--(-1,0)--(0,-1);
			\draw[dashed, red, line width=1.5pt] (0,1)--(1,0)--(0,-1);;
			
			\draw (2.4,0) node[anchor=north] {$(1+\sin^2(t))e_1$};
			\draw (0.35,1) node {$e_2$};
			\draw (-1.3,0) node[anchor=north]{$-e_1$};
			\draw (0.3,-1) node[anchor=north] {$-e_2$}; 
			\fill (0,1) circle (1mm);\fill (-1,0) circle (1mm);\fill (0,-1) circle (1mm); 			
			\end{tikzpicture}\hspace{-2mm}\begin{tikzpicture}[ {z=(60:0.5cm)}][scale=0.1]
			\draw (1.05,-0.1,0) node[anchor=north] {$e_1$};
			\draw (0,1.3,0) node {$e_3$};
			\draw (0,2) node {$\gamma_{X_2}$};
			\draw (-1,0,0)--(0,-1,0);
			\draw (0,-1,0)--(1,0,0);\draw (0,0,1)--(1,0,0)--(0,0,-1);
			\draw (1,0,0)--(0,1,0);
			\draw (0,0,1)--(0,-1,0);
			\draw (0,-1,0)--(0,0,-1);
			\draw (0,0,-1)--(-1,0,0)--(0,0,1);
			\draw (-1,0,0)--(0,1,0);
			\draw (0,0,1)--(0,1,0);
			\draw (0,0,-1)--(0,1,0);
			\fill[gray,opacity=0.3] (-1,0,0)--(0,-1,0)--(0,0,1)--cycle;
			\draw[fill,gray,opacity=0.3] (0,-1,0)--(0,0,-1)--(-1,0,0)--cycle;
			\fill[gray, opacity=0.3] (-1,0,0)--(0,1,0) --(0,0,1)--cycle;
			\fill[gray,opacity=0.3] (-1,0,0)--(0,1,0)--(0,0,-1)--cycle;
			\fill[gray, opacity=0.3] (1,0,0)--(0,1,0)--(0,0,1)--cycle;
			\fill[gray, opacity=0.3]
			(1,0,0)--(0,-1,0)--(0,0,1)--cycle;
			\fill[gray, opacity=0.3]
			(1,0,0)--(0,1,0)--(0,0,-1)--cycle; 
			\fill[gray, opacity=0.3]
			(1,0,0)--(0,-1,0)--(0,0,-1)--cycle; 
			\fill (0,0,1) circle (1mm);
			\fill (0,0,-1) circle (1mm); 
			\fill (0,1,0) circle (1mm); 
			\fill (0,-1,0) circle (1mm); 
			\fill (-1,0,0) circle (1mm); 
			\fill (1,0,0) circle (1mm); 
			\end{tikzpicture}\hspace{-0.5mm}\begin{tikzpicture}[{z=(60:0.5cm)}][scale=0.1]
			\draw (0,2) node {$\gamma_{X_2}'$};
			\draw (2,0,0) node[anchor=north] {$(1+\sin^2(t))e_1$};
			\draw (0,1.3,0) node {$e_3$};
			
			\draw (0,0,1)--(0,-1,0);
			\draw[fill,gray,opacity=0.3] (0,-1,0)--(0,0,-1)--(-1,0,0)--cycle;
			\fill[gray, opacity=0.3] (-1,0,0)--(0,1,0) --(0,0,1)--cycle;
			\fill[gray,opacity=0.3] (-1,0,0)--(0,1,0)--(0,0,-1)--cycle;
			\fill[gray,opacity=0.3] (-1,0,0)--(0,-1,0)--(0,0,1)--cycle;
			\fill[pink, opacity=0.3] (1,0,0)--(0,1,0)--(0,0,1)--cycle;
			\fill[pink, opacity=0.3]
			(1,0,0)--(0,-1,0)--(0,0,1)--cycle;
			\fill[pink, opacity=0.3]
			(1,0,0)--(0,1,0)--(0,0,-1)--cycle; 
			\fill[pink, opacity=0.3]
			(1,0,0)--(0,-1,0)--(0,0,-1)--cycle; 
			
			\draw (0,-1,0)--(0,0,-1);
			\draw (0,0,-1)--(-1,0,0)--(0,0,1);
			\draw (-1,0,0)--(0,1,0);
			\draw (0,0,1)--(0,1,0);
			\draw (0,0,-1)--(0,1,0);
			
			\draw[red,<->,>=stealth', line width=1.5pt] (1.2,0,0)--(2,0,0);
			\draw (-1,0,0)--(0,-1,0);
			\draw[dashed, red, line width=1.2pt] (0,-1,0)--(1,0,0);\draw[dashed, red, line width=1.2pt] (0,0,1)--(1,0,0)--(0,0,-1);\draw[dashed, red, line width=1.2pt] (1,0,0)--(0,1,0);
			\fill (0,0,1) circle (1mm);
			\fill (0,0,-1) circle (1mm); 
			\fill (0,1,0) circle (1mm); 
			\fill (0,-1,0) circle (1mm); 
			\fill (-1,0,0) circle (1mm); 
			\draw (1,0,0) node[fill, red, star, inner sep=2pt,minimum size=1mm]{}; 			\end{tikzpicture}			
		\end{center}
		\caption{\label{fig:counter-example}Pairs of DMSs $(\gamma_{X_i},\gamma_{X_i}')$ for $i=1,2$ such that $\dintm\left(\gamma_{X_i},\gamma_{X_i}'\right)\leq \pi/2$. In contrast, for $k=1$ (or $k=2$), the bottleneck distance between their $k$-dimensional zigzag-persistence  barcodes is infinite for $\delta\in [\sqrt{2},2)$. DMS $\gamma_{X_1}$, described as the left-most figure, ($\gamma_{X_2}$, the third figure from the left) consists of four (eight) static points located at $\pm e_1=(\pm 1,0,0)$ and $\pm e_2=(0,\pm 1,0)$ (and $\pm e_3=(0,0,\pm 1)$), respectively. On the other hand, DMS $\gamma_{X_1}'$, illustrated at the second from the left ($\gamma_{X_2}'$, at the right-most), contains a single oscillating point, denoted by a star shape, with  trace $(1+\sin^2(t))e_1$ for $t\in \R$ along with three (five) static points located at $-e_1, +e_2$ and $-e_2$, (and $\pm e_3$), respectively. Then, the 1-dimensional (2-dimensional) zigzag-persistent homology barcode for $\gamma_{X_1}$ (for $\gamma_{X_2}$) consists of exactly one interval $(-\infty,\infty)$, indicating the presence of a loop (a void) for all time. However, the barcode of $\gamma_{X_1}'$ ($\gamma_{X_2}'$) consists of an infinite number of ephemeral intervals $[n\pi,n\pi]$, $n\in \Z$, indicating the on-and-off presence of a loop (a void) that exists only at $t=n\pi$ for $n\in \Z$ in its configuration.} 
	\end{figure}
	
	\begin{proof}
		Fix any $k\geq 1$. We will illustrate DMSs  $\gamma_{X_k}$ and $\gamma_{Y_k}$ as collections of trajectories of points in $\R^{k+1}$, with the metric inherited from the Euclidean metric of $\R^{k+1}$ across all $t\in \T$. For $k=1$ or $k=2$, see Figure \ref{fig:counter-example}.
		
		Define $\gamma_{X_k}$ to be the constant DMS consisting of $2(k+1)$ points $\pm e_i=(0,\ldots,0,\pm 1,0,\ldots,0)\in \R^{k+1}$  for $i=1,2,\ldots,k+1$. On the other hand,  define $\gamma_{Y_k}$ to be obtained from $\gamma_{X_k}$ by substituting the still point $+e_1$ of $\gamma_{X_k}$ by the oscillating point $(1+\sin^2(t))e_1=(1+\sin^2(t),0,\ldots,0)$ for $t\in \T$. 
		
		It is not difficult to check that $\dintm\left(\gamma_{X_k},\gamma_{Y_k}\right)\leq \pi/2$. However, with the connectivity parameter $\delta=\sqrt{2}$, their barcodes of the $k$-th zigzag persistent homology are $\barc\left(\Hrm_k\left(\rips(X_k)\right)\right)=\{(-\infty,\infty)\}$ and $\barc\left(\Hrm_k\left(\rips(Y_k)\right)\right)=\{[n\pi,n\pi]:n\in \Z\}$, respectively.   
		Therefore, $\bott\left(\barc\left(\Hrm_k\left(\rips(X_k)\right)\right), \barc\left(\Hrm_k\left(\rips(Y_k)\right)\right) \right)=+\infty.$
 \end{proof}
	
	\bibliographystyle{abbrv}
	\bibliography{biblio}

\begin{thebibliography}{10}

\bibitem{adams2015evasion}
H.~Adams and G.~Carlsson.
\newblock Evasion paths in mobile sensor networks.
\newblock {\em The International Journal of Robotics Research}, 34(1):90--104,
  2015.

\bibitem{adams2021efficient}
H.~Adams, D.~Ghosh, C.~Mask, W.~Ott, and K.~Williams.
\newblock Efficient evader detection in mobile sensor networks.
\newblock {\em arXiv preprint arXiv:2101.09813}, 2021.

\bibitem{azumaya1950corrections}
G.~Azumaya et~al.
\newblock Corrections and supplementaries to my paper concerning
  {K}rull-{R}emak-{S}chmidt's theorem.
\newblock {\em Nagoya Mathematical Journal}, 1:117--124, 1950.

\bibitem{bauer2014measuring}
U.~Bauer, X.~Ge, and Y.~Wang.
\newblock Measuring distance between {R}eeb graphs.
\newblock In {\em Proceedings of the thirtieth annual symposium on
  Computational geometry}, pages 464--473, 2014.

\bibitem{bauer2015induced}
U.~Bauer and M.~Lesnick.
\newblock Induced matchings and the algebraic stability of persistence
  barcodes.
\newblock {\em Journal of Computational Geometry}, 6(2):162--191, 2015.

\bibitem{bauer2015strong}
U.~Bauer, E.~Munch, and Y.~Wang.
\newblock {Strong Equivalence of the Interleaving and Functional Distortion
  Metrics for Reeb Graphs}.
\newblock In L.~Arge and J.~Pach, editors, {\em 31st International Symposium on
  Computational Geometry (SoCG 2015)}, volume~34 of {\em Leibniz International
  Proceedings in Informatics (LIPIcs)}, pages 461--475, Dagstuhl, Germany,
  2015. Schloss Dagstuhl--Leibniz-Zentrum fuer Informatik.

\bibitem{benkert2008reporting}
M.~Benkert, J.~Gudmundsson, F.~H{\"u}bner, and T.~Wolle.
\newblock Reporting flock patterns.
\newblock {\em Computational Geometry}, 41(3):111--125, 2008.

\bibitem{birkhoff1948lattice}
G.~Birkhoff.
\newblock {\em Lattice theory}, volume~25.
\newblock American Mathematical Society New York, 1948.

\bibitem{bjerkevik2016stability}
H.~B. Bjerkevik.
\newblock On the stability of interval decomposable persistence modules.
\newblock {\em Discrete \& Computational Geometry}, 66(1):92--121, 2021.

\bibitem{bondy2008graph}
J.~Bondy and U.~Murty.
\newblock {\em Graph theory (graduate texts in mathematics)}.
\newblock Springer New York, 2008.

\bibitem{botnan2018algebraic}
M.~Botnan and M.~Lesnick.
\newblock Algebraic stability of zigzag persistence modules.
\newblock {\em Algebraic \& geometric topology}, 18(6):3133--3204, 2018.

\bibitem{botnan2015interval}
M.~B. Botnan.
\newblock Interval decomposition of infinite zigzag persistence modules.
\newblock {\em Proceedings Of The American Mathematical Society},
  145(8):3571--3577, 2017.

\bibitem{bredon2012sheaf}
G.~E. Bredon.
\newblock {\em Sheaf theory}, volume 170.
\newblock Springer Science \& Business Media, 2012.

\bibitem{bubenik2014categorification}
P.~Bubenik and J.~A. Scott.
\newblock Categorification of persistent homology.
\newblock {\em Discrete \& Computational Geometry}, 51(3):600--627, 2014.

\bibitem{buchin2013trajectory}
K.~Buchin, M.~Buchin, M.~J. van Kreveld, B.~Speckmann, and F.~Staals.
\newblock Trajectory grouping structure.
\newblock {\em JoCG}, 6(1):75--98, 2015.

\bibitem{buragobook}
D.~Burago, Y.~Burago, and S.~Ivanov.
\newblock {\em A Course in Metric Geometry}, volume~33 of {\em AMS Graduate
  Studies in Math.}
\newblock American Mathematical Society, 2001.

\bibitem{carlsson2009topology}
G.~Carlsson.
\newblock Topology and data.
\newblock {\em Bulletin of the American Mathematical Society}, 46(2):255--308,
  2009.

\bibitem{zigzag}
G.~Carlsson and V.~De~Silva.
\newblock Zigzag persistence.
\newblock {\em Foundations of computational mathematics}, 10(4):367--405, 2010.

\bibitem{clustum}
G.~Carlsson and F.~M{\'e}moli.
\newblock Characterization, stability and convergence of hierarchical
  clustering methods.
\newblock {\em Journal of Machine Learning Research}, 11:1425--1470, 2010.

\bibitem{multi-clust}
G.~Carlsson and F.~M{\'e}moli.
\newblock Multiparameter hierarchical clustering methods.
\newblock In {\em Classification as a Tool for Research}, pages 63--70.
  Springer, 2010.

\bibitem{carlsson2013classifying}
G.~Carlsson and F.~M{\'e}moli.
\newblock Classifying clustering schemes.
\newblock {\em Foundations of Computational Mathematics}, 13(2):221--252, 2013.

\bibitem{carlsson2009theory}
G.~Carlsson and A.~Zomorodian.
\newblock The theory of multidimensional persistence.
\newblock {\em Discrete \& Computational Geometry}, 42(1):71--93, 2009.

\bibitem{CCG09}
F.~Chazal, D.~Cohen-Steiner, M.~Glisse, L.~J. Guibas, and S.~Oudot.
\newblock Proximity of persistence modules and their diagrams.
\newblock In {\em Proc. 25th ACM Sympos. on Comput. Geom.}, pages 237--246,
  2009.

\bibitem{chazal2014stochastic}
F.~Chazal, B.~T. Fasy, F.~Lecci, A.~Rinaldo, and L.~Wasserman.
\newblock Stochastic convergence of persistence landscapes and silhouettes.
\newblock In {\em Proceedings of the thirtieth annual symposium on
  Computational geometry}, pages 474--483, 2014.

\bibitem{chowdhury2016explicit}
S.~Chowdhury and F.~M{\'e}moli.
\newblock Explicit geodesics in {G}romov-{H}ausdorff space.
\newblock {\em Electronic Research Announcements}, 25:48--59, 2018.

\bibitem{clause2021zigzag}
N.~Clause.
\newblock Zigzag persistent homology and dynamic networks.
\newblock \url{https://github.com/ndag/DynGraphZZ}, 2021.

\bibitem{clause2020spatiotemporal}
N.~Clause and W.~Kim.
\newblock Spatiotemporal persistent homology computation tool.
\newblock \url{https://github.com/ndag/PHoDMSs}, 2020.

\bibitem{cohen2007stability}
D.~Cohen-Steiner, H.~Edelsbrunner, and J.~Harer.
\newblock Stability of persistence diagrams.
\newblock {\em Discrete \& Computational Geometry}, 37(1):103--120, 2007.

\bibitem{curry2020classification}
J.~Curry and A.~Patel.
\newblock Classification of constructible cosheaves.
\newblock {\em Theory and Applications of Categories}, 35(27):1012--1047, 2020.

\bibitem{curry2014sheaves}
J.~M. Curry.
\newblock {\em Sheaves, cosheaves and applications}.
\newblock University of Pennsylvania, 2014.

\bibitem{de2006coordinate}
V.~De~Silva and R.~Ghrist.
\newblock Coordinate-free coverage in sensor networks with controlled
  boundaries via homology.
\newblock {\em The International Journal of Robotics Research},
  25(12):1205--1222, 2006.

\bibitem{de2007homological}
V.~De~Silva, R.~Ghrist, et~al.
\newblock Homological sensor networks.
\newblock {\em Notices of the American mathematical society}, 54(1), 2007.

\bibitem{de2016categorified}
V.~De~Silva, E.~Munch, and A.~Patel.
\newblock Categorified {R}eeb graphs.
\newblock {\em Discrete \& Computational Geometry}, 55(4):854--906, 2016.

\bibitem{dey2021computing}
T.~K. Dey and T.~Hou.
\newblock Computing zigzag persistence on graphs in near-linear time.
\newblock In {\em 37th International Symposium on Computational Geometry (SoCG
  2021)}, volume 189 of {\em Leibniz International Proceedings in Informatics
  (LIPIcs)}, pages 30:1--30:15. Schloss Dagstuhl -- Leibniz-Zentrum f{\"u}r
  Informatik, 2021.

\bibitem{dey2021updating}
T.~K. Dey and T.~Hou.
\newblock Updating zigzag persistence and maintaining representatives over
  changing filtrations.
\newblock {\em arXiv preprint arXiv:2112.02352}, 2021.

\bibitem{di2016edit}
B.~Di~Fabio and C.~Landi.
\newblock The edit distance for {R}eeb graphs of surfaces.
\newblock {\em Discrete \& Computational Geometry}, 55(2):423--461, 2016.

\bibitem{edelsbrunner2008persistent}
H.~Edelsbrunner, J.~Harer, et~al.
\newblock Persistent homology-a survey.
\newblock {\em Contemporary mathematics}, 453:257--282, 2008.

\bibitem{ELZ02}
H.~Edelsbrunner, D.~Letscher, and A.~Zomorodian.
\newblock Topological persistence and simplification.
\newblock {\em Discrete Comput. Geom.}, 28:511--533, 2002.

\bibitem{gabriel1972unzerlegbare}
P.~Gabriel.
\newblock Unzerlegbare darstellungen i.
\newblock {\em Manuscripta mathematica}, 6(1):71--103, 1972.

\bibitem{gamble2012applied}
J.~Gamble, H.~Chintakunta, and H.~Krim.
\newblock Applied topology in static and dynamic sensor networks.
\newblock In {\em 2012 International Conference on Signal Processing and
  Communications (SPCOM)}, pages 1--5. IEEE, 2012.

\bibitem{ghrist2020cellular}
R.~Ghrist and H.~Riess.
\newblock Cellular sheaves of lattices and the {T}arski {L}aplacian.
\newblock {\em arXiv preprint arXiv:2007.04099, To appear in Homology, Homotopy
  and Applications}, 2022.

\bibitem{gonzalez2015spatiotemporal}
R.~Gonzalez-Diaz, M.-J. Jimenez, and B.~Medrano.
\newblock Spatiotemporal barcodes for image sequence analysis.
\newblock In {\em International Workshop on Combinatorial Image Analysis},
  pages 61--70. Springer, 2015.

\bibitem{gudmundsson2006computing}
J.~Gudmundsson and M.~van Kreveld.
\newblock Computing longest duration flocks in trajectory data.
\newblock In {\em Proceedings of the 14th annual ACM international symposium on
  Advances in geographic information systems}, pages 35--42. ACM, 2006.

\bibitem{gudmundsson2007efficient}
J.~Gudmundsson, M.~van Kreveld, and B.~Speckmann.
\newblock Efficient detection of patterns in 2d trajectories of moving points.
\newblock {\em Geoinformatica}, 11(2):195--215, 2007.

\bibitem{hajij2017visual}
M.~Hajij, B.~Wang, C.~Scheidegger, and P.~Rosen.
\newblock Visual detection of structural changes in time-varying graphs using
  persistent homology. arxiv preprint.
\newblock {\em arXiv preprint arXiv:1707.06683}, 3, 2017.

\bibitem{huang2008modeling}
Y.~Huang, C.~Chen, and P.~Dong.
\newblock Modeling herds and their evolvements from trajectory data.
\newblock In {\em International Conference on Geographic Information Science},
  pages 90--105. Springer, 2008.

\bibitem{hwang2005mining}
S.-Y. Hwang, Y.-H. Liu, J.-K. Chiu, and E.-P. Lim.
\newblock Mining mobile group patterns: A trajectory-based approach.
\newblock In {\em PAKDD}, volume 3518, pages 713--718. Springer, 2005.

\bibitem{jardine-sibson}
N.~Jardine and R.~Sibson.
\newblock {\em Mathematical taxonomy}.
\newblock John Wiley \& Sons Ltd., London, 1971.
\newblock Wiley Series in Probability and Mathematical Statistics.

\bibitem{jeung2008discovery}
H.~Jeung, M.~L. Yiu, X.~Zhou, C.~S. Jensen, and H.~T. Shen.
\newblock Discovery of convoys in trajectory databases.
\newblock {\em Proceedings of the VLDB Endowment}, 1(1):1068--1080, 2008.

\bibitem{kalnis2005discovering}
P.~Kalnis, N.~Mamoulis, and S.~Bakiras.
\newblock On discovering moving clusters in spatio-temporal data.
\newblock In {\em SSTD}, volume 3633, pages 364--381. Springer, 2005.

\bibitem{kim2017stable}
W.~Kim and F.~Memoli.
\newblock Stable signatures for dynamic graphs and dynamic metric spaces via
  zigzag persistence.
\newblock {\em arXiv preprint arXiv:1712.04064v4}, 2017.

\bibitem{kim2018CCCG}
W.~Kim and F.~M{\'e}moli.
\newblock Formigrams: Clustering summaries of dynamic data.
\newblock In {\em Proceedings of 30th Canadian Conference on Computational
  Geometry (CCCG18)}, 2018.

\bibitem{kim2018generalized}
W.~Kim and F.~M{\'e}moli.
\newblock Generalized persistence diagrams for persistence modules over posets.
\newblock {\em Journal of Applied and Computational Topology}, 5(4):533--581,
  2021.

\bibitem{kim2020spatiotemporal}
W.~Kim and F.~M{\'e}moli.
\newblock Spatiotemporal persistent homology for dynamic metric spaces.
\newblock {\em Discrete \& Computational Geometry}, 66(3):831--875, 2021.

\bibitem{kim2020analysis}
W.~Kim, F.~M{\'e}moli, and Z.~Smith.
\newblock Analysis of dynamic graphs and dynamic metric spaces via zigzag
  persistence.
\newblock In {\em Topological data analysis}, pages 371--389. Springer, 2020.

\bibitem{kim2019interleaving}
W.~Kim, F.~M{\'e}moli, and A.~Stefanou.
\newblock Interleaving by parts for persistence in a poset.
\newblock {\em arXiv preprint arXiv:1912.04366}, 2019.

\bibitem{li2010swarm}
Z.~Li, B.~Ding, J.~Han, and R.~Kays.
\newblock Swarm: Mining relaxed temporal moving object clusters.
\newblock {\em Proceedings of the VLDB Endowment}, 3(1-2):723--734, 2010.

\bibitem{mac2013categories}
S.~Mac~Lane.
\newblock {\em Categories for the working mathematician}, volume~5.
\newblock Springer Science \& Business Media, 2013.

\bibitem{mccleary2018bottleneck}
A.~McCleary and A.~Patel.
\newblock Bottleneck stability for generalized persistence diagrams.
\newblock {\em Proceedings of the American Mathematical Society},
  148(7):3149--3161, 2020.

\bibitem{mccleary2020edit}
A.~McCleary and A.~Patel.
\newblock Edit distance and persistence diagrams over lattices.
\newblock {\em arXiv preprint arXiv:2010.07337}, 2020.

\bibitem{memoli2017distance}
F.~M{\'e}moli.
\newblock A distance between filtered spaces via tripods.
\newblock {\em arXiv preprint arXiv:1704.03965}, 2017.

\bibitem{mitchell1965theory}
B.~Mitchell.
\newblock {\em Theory of categories}, volume~17.
\newblock Academic Press, 1965.

\bibitem{morozov2013interleaving}
D.~Morozov, K.~Beketayev, and G.~Weber.
\newblock Interleaving distance between merge trees.
\newblock {\em Discrete and Computational Geometry}, 49:22--45, 2013.

\bibitem{munch2013applications}
E.~Munch.
\newblock {\em Applications of persistent homology to time varying systems}.
\newblock PhD thesis, Duke University, 2013.

\bibitem{parrish1997animal}
J.~K. Parrish and W.~M. Hamner.
\newblock {\em Animal groups in three dimensions: how species aggregate}.
\newblock Cambridge University Press, 1997.

\bibitem{patel2010reeb}
A.~Patel.
\newblock {\em Reeb spaces and the robustness of preimages}.
\newblock Duke University, 2010.

\bibitem{patel2018generalized}
A.~Patel.
\newblock Generalized persistence diagrams.
\newblock {\em Journal of Applied and Computational Topology}, 1(3):397--419,
  2018.

\bibitem{puuska2017erosion}
V.~Puuska.
\newblock Erosion distance for generalized persistence modules.
\newblock {\em Homology, Homotopy and Applications}, 22(1):233--254, 2020.

\bibitem{boids}
C.~W. Reynolds.
\newblock Flocks, herds and schools: A distributed behavioral model.
\newblock {\em ACM SIGGRAPH computer graphics}, 21(4):25--34, 1987.

\bibitem{rolle2020stable}
A.~Rolle and L.~Scoccola.
\newblock Stable and consistent density-based clustering.
\newblock {\em arXiv preprint arXiv:2005.09048}, 2020.

\bibitem{rossetti2018community}
G.~Rossetti and R.~Cazabet.
\newblock Community discovery in dynamic networks: a survey.
\newblock {\em ACM Computing Surveys (CSUR)}, 51(2):1--37, 2018.

\bibitem{rota1964foundations}
G.-C. Rota.
\newblock On the foundations of combinatorial theory i. theory of {M}{\"o}bius
  functions.
\newblock {\em Probability theory and related fields}, 2(4):340--368, 1964.

\bibitem{rubenstein2012kilobot}
M.~Rubenstein, C.~Ahler, and R.~Nagpal.
\newblock Kilobot: A low cost scalable robot system for collective behaviors.
\newblock In {\em 2012 IEEE international conference on robotics and
  automation}, pages 3293--3298. IEEE, 2012.

\bibitem{schmiedl2015shape}
F.~Schmiedl.
\newblock {\em Shape matching and mesh segmentation}.
\newblock PhD thesis, Technische Universit{\"a}t M{\"u}nchen, 2015.

\bibitem{schmiedl2017computational}
F.~Schmiedl.
\newblock Computational aspects of the {G}romov--{H}ausdorff distance and its
  application in non-rigid shape matching.
\newblock {\em Discrete \& Computational Geometry}, 57(4):854--880, 2017.

\bibitem{sinhuber2017phase}
M.~Sinhuber and N.~T. Ouellette.
\newblock Phase coexistence in insect swarms.
\newblock {\em Physical review letters}, 119(17):178003, 2017.

\bibitem{sumpter-collective}
D.~J. Sumpter.
\newblock {\em Collective animal behavior}.
\newblock Princeton University Press, 2010.

\bibitem{topaz}
C.~M. Topaz, L.~Ziegelmeier, and T.~Halverson.
\newblock Topological data analysis of biological aggregation models.
\newblock {\em PloS one}, 10(5):e0126383, 2015.

\bibitem{vangoethem2016grouping}
A.~van Goethem, M.~van Kreveld, M.~L{\"o}ffler, B.~Speckmann, and F.~Staals.
\newblock {Grouping Time-Varying Data for Interactive Exploration}.
\newblock In {\em 32nd International Symposium on Computational Geometry (SoCG
  2016)}, volume~51 of {\em Leibniz International Proceedings in Informatics
  (LIPIcs)}, pages 61:1--61:16. Schloss Dagstuhl--Leibniz-Zentrum fuer
  Informatik, 2016.

\bibitem{van2018refined}
M.~van Kreveld, M.~L{\"o}ffler, F.~Staals, and L.~Wiratma.
\newblock A refined definition for groups of moving entities and its
  computation.
\newblock {\em International Journal of Computational Geometry \&
  Applications}, 28(02):181--196, 2018.

\bibitem{vehlow2015visualizing}
C.~Vehlow, F.~Beck, P.~Auw{\"a}rter, and D.~Weiskopf.
\newblock Visualizing the evolution of communities in dynamic graphs.
\newblock {\em Computer Graphics Forum}, 34(1):277--288, 2015.

\bibitem{vieira2009line}
M.~R. Vieira, P.~Bakalov, and V.~J. Tsotras.
\newblock On-line discovery of flock patterns in spatio-temporal data.
\newblock In {\em Proceedings of the 17th ACM SIGSPATIAL international
  conference on advances in geographic information systems}, pages 286--295.
  ACM, 2009.

\bibitem{wang2008efficient}
Y.~Wang, E.-P. Lim, and S.-Y. Hwang.
\newblock Efficient algorithms for mining maximal valid groups.
\newblock {\em The VLDB Journal—The International Journal on Very Large Data
  Bases}, 17(3):515--535, 2008.

\bibitem{wiki}
Wikipedia.
\newblock Formicarium --- {W}ikipedia{,} the free encyclopedia, 2021.
\newblock [Online; accessed 12-December-2021].

\bibitem{wiratma2019experimental}
L.~Wiratma, M.~van Kreveld, M.~L{\"o}ffler, and F.~Staals.
\newblock An experimental evaluation of grouping definitions for moving
  entities.
\newblock In {\em Proceedings of the 27th ACM SIGSPATIAL International
  Conference on Advances in Geographic Information Systems}, pages 89--98,
  2019.

\bibitem{xian2020capturing}
L.~Xian, H.~Adams, C.~M. Topaz, and L.~Ziegelmeier.
\newblock Capturing dynamics of time-varying data via topology.
\newblock {\em Foundations of Data Science}, 4(1):1--36, 2022.

\end{thebibliography}
\end{document}